\documentclass[oneside,11pt]{article}

\usepackage{blindtext}

\usepackage[preprint]{jmlr2e}

\usepackage[applemac]{inputenc} 		
\usepackage[T1]{fontenc}    		
\usepackage{url}            			
\usepackage{booktabs}       		
\usepackage{amsfonts}       		
\usepackage{amsmath}
\usepackage{nicefrac}       		
\usepackage{microtype}      		
\usepackage{mathtools}
\usepackage[ruled]{algorithm2e}
\usepackage{algpseudocode}
\usepackage{indentfirst,latexsym,bm}
\usepackage{subfigure}
\usepackage{xcolor}
\usepackage{tcolorbox}
\usepackage{comment}
\usepackage{enumitem}
\usepackage{bbm}
\usepackage{tabularx}
\usepackage{longtable}

\usepackage{multirow}
\usepackage{dsfont}
\usepackage{nicematrix}
\usepackage{pifont}	
\usepackage{ragged2e}
\usepackage{tabularx}
\usepackage{array}
\usepackage{colortbl}
\usepackage[flushleft]{threeparttable}
\usepackage{tikz}
\usetikzlibrary{arrows.meta,positioning}
\usepackage{pgfplots}
\pgfplotsset{compat=newest}
\usepgfplotslibrary{fillbetween}
\usetikzlibrary{shapes,decorations}
\usetikzlibrary{fit}

\usepackage{booktabs}

\colorlet{color1}{blue}
\colorlet{color2}{red!50!black}

\definecolor{ivory}{RGB}{218,215,203}

\definecolor{cuhkp}{RGB}{98,56,105} 	
\definecolor{cuhkpl}{RGB}{152,24,147} 	
\definecolor{cuhkb}{RGB}{219,160,1} 	
\definecolor{cuhkbd}{RGB}{178,129,0} 	
\definecolor{cuhkr}{RGB}{88,35,155}  	
\definecolor{ruri}{RGB}{0,92,175}  		
\definecolor{enji}{RGB}{159,53,58}  		
\definecolor{mygray}{RGB}{250,250,250}
\definecolor{myred}{RGB}{205,27,41}
\definecolor{myblue}{RGB}{0,38,112}
\definecolor{lavender}{rgb}{0.9, 0.9, 0.98}

\usepackage{mdframed}

\definecolor{ivory}{RGB}{218,215,203}

\definecolor{MyBlue}{RGB}{123,144,210} 	
\definecolor{MyRed}{RGB}{224,60,138} 	
\definecolor{DeepGreen}{RGB}{92,172,129} 	
\definecolor{MyOrange}{RGB}{226,148,59} 	
\definecolor{DeepBlue}{RGB}{0,92,175}  
\definecolor{DeepRed}{RGB}{167,67,67}  
\definecolor{LightBlue}{RGB}{229,232,247}   

\newcommand{\xmark}{\color{myred}\ding{55}}

\mdfdefinestyle{questionbox}{
backgroundcolor=LightBlue!10,
    linecolor=blue, 
    linewidth=1pt,
    topline=true,
    bottomline=true,
    rightline=true,
    leftmargin=0pt,
    innertopmargin=8pt,
    innerbottommargin=8pt,
    innerrightmargin=10pt,
    innerleftmargin=10pt,
    shadow=false,
    shadowcolor=gray!10,
    shadowsize=2pt,
    roundcorner=0pt,
}

\mdfdefinestyle{proofbox}{
    backgroundcolor=LightBlue!10,
    linecolor=ivory, 
    linewidth=8pt,
    topline=false,
    bottomline=false,
    rightline=false,
    leftline=true,
    leftmargin=0pt,
    innertopmargin=8pt,
    innerbottommargin=8pt,
    innerrightmargin=10pt,
    innerleftmargin=10pt,
    shadow=false,
    shadowcolor=gray!10,
    shadowsize=2pt,
    roundcorner=0pt,
}

\hypersetup{
	colorlinks=true,
	frenchlinks=false,
	pdfborder={0 0 0},
	naturalnames=false,
	hypertexnames=false,
	breaklinks,
	allcolors = blue!80!black,
	urlcolor = color2,
}

\RequirePackage[capitalize,nameinlink]{cleveref}

\crefformat{equation}{\textup{#2(#1)#3}}
\crefrangeformat{equation}{\textup{#3(#1)#4--#5(#2)#6}}
\crefmultiformat{equation}{\textup{#2(#1)#3}}{ and \textup{#2(#1)#3}}
{, \textup{#2(#1)#3}}{, and \textup{#2(#1)#3}}
\crefrangemultiformat{equation}{\textup{#3(#1)#4--#5(#2)#6}}%
{ and \textup{#3(#1)#4--#5(#2)#6}}{, \textup{#3(#1)#4--#5(#2)#6}}{, and \textup{#3(#1)#4--#5(#2)#6}}
\Crefformat{equation}{#2Equation~\textup{(#1)}#3}
\Crefrangeformat{equation}{Equations~\textup{#3(#1)#4--#5(#2)#6}}
\Crefmultiformat{equation}{Equations~\textup{#2(#1)#3}}{ and \textup{#2(#1)#3}}
{, \textup{#2(#1)#3}}{, and \textup{#2(#1)#3}}
\Crefrangemultiformat{equation}{Equations~\textup{#3(#1)#4--#5(#2)#6}}%
{ and \textup{#3(#1)#4--#5(#2)#6}}{, \textup{#3(#1)#4--#5(#2)#6}}{, and \textup{#3(#1)#4--#5(#2)#6}}

\DeclareMathOperator*{\argmin}{argmin}

\newcommand{\rmn}[1]{\textup{\textrm{#1}}}

\newcommand{\cO}{\mathcal{O}}

\newcommand{\R}{\mathbb{R}}

\newcommand{\sA}{{\sf A}}
\newcommand{\sB}{{\sf B}}
\newcommand{\sC}{{\sf C}}
\newcommand{\sD}{{\sf D}}
\newcommand{\sE}{{\sf E}}
\newcommand{\sL}{{\sf L}}

\newcommand{\sQ}{{\sf Q}}
\newcommand{\sP}{{\sf P}}

\newcommand{\flb}{f_{\mathrm{lb}}}
\newcommand{\vplb}{\vp_{\mathrm{lb}}}
\newcommand{\psilb}{\psi_{\mathrm{lb}}}

\newcommand{\CL}{\mathsf{C}}

\newcommand{\N}{\mathbb{N}}
\newcommand{\Rn}{\mathbb{R}^d}

\newcommand{\Rex}{(-\infty,\infty]}

\newcommand{\Exp}{\mathbb{E}}
\newcommand{\vp}{\varphi}

\newcommand{\dist}{\mathrm{dist}}
\newcommand{\crit}[1]{\mathrm{crit}(#1)}

\newcommand{\cT}{\mathcal{T}}
\newcommand{\half}{\frac{1}{2}}
\newcommand{\iprod}[2]{\langle #1, #2 \rangle}
\newcommand{\mer}{H_\tau}

\newcommand{\SGD}{\mathsf{SGD}}

\newcommand{\RR}{\mathsf{RR}}
\newcommand{\PSGD}{\mathsf{PSGD}}
\newcommand{\PGD}{\mathsf{PGD}}
\newcommand{\NRR}{\mathsf{norm}\text{-}\mathsf{PRR}}

\newcommand{\EPRR}{\mathsf{e}\text{-}\mathsf{PRR}}

\newcommand{\be}{\begin{equation}}
\newcommand{\ee}{\end{equation}}

\newcommand{\addes}{\varrho_\vartheta}

\newcommand{\prox}[1]{\mathrm{prox}_{#1}}

\newcommand\proxl{\mathrm{prox}_{\lambda\vp}}

\newcommand\proxi[1]{\mathrm{prox}_{#1\vp}}

\newcommand\Fnor{F^{\lambda}_{\mathrm{nor}}}

\newcommand{\dom}[1]{\mathrm{dom\\}(#1)}

\definecolor{minrev}{RGB}{66,175,55}

\newcommand{\norm}[1]{\|{#1}\|}
\newcommand{\prt}[1]{\left({#1}\right)}

\newtheorem{assumption}[theorem]{Assumption}

\newcolumntype{C}[1]{>{\centering\arraybackslash}p{#1}}
\newcolumntype{M}[1]{>{\arraybackslash}p{#1}}

\usepackage{lastpage}
\jmlrheading{26}{2025}{1-\pageref{LastPage}}{6/24; Revised
3/25}{7/25}{24-0891}{Junwen Qiu, Xiao Li, and Andre Milzarek}

\ShortHeadings{Normal map-based RR for Nonsmooth Nonconvex Finite-sum Problems}{Qiu, Li, and Milzarek}
\firstpageno{1}

\begin{document}

\title{A New Random Reshuffling Method for Nonsmooth Nonconvex Finite-sum Optimization}

\author{\name Junwen Qiu \email jwqiu@nus.edu.sg \\
       \addr School of Data Science\\
       The Chinese University of Hong Kong, Shenzhen\\
       Guangdong, 518172, P.R. China\\[2mm]
       Industrial Systems Engineering \& Management\\
       National University of Singapore\\
       Singapore, 119077, Singapore 
       \AND
       \name Xiao Li \email lixiao@cuhk.edu.cn
	   \AND
       \name Andre Milzarek \email andremilzarek@cuhk.edu.cn \\
       \addr School of Data Science\\
       The Chinese University of Hong Kong, Shenzhen \\
       Guangdong, 518172, P.R. China
       }
\editor{Nicolas Le Roux}

\maketitle

\begin{abstract}%
Random reshuffling techniques are prevalent in large-scale applications, such as training neural networks. While the convergence and acceleration effects of random reshuffling-type methods are fairly well understood in the smooth setting, much less studies seem available in the nonsmooth case. In this work, we design a new normal map-based proximal random reshuffling ($\NRR$) method for nonsmooth nonconvex finite-sum problems. We show that $\NRR$ achieves the iteration complexity $\cO(n^{-1/3}T^{-2/3})$ where $n$ denotes the number of component functions $f(\cdot,i)$ and $T$ counts the total number of iterations. This improves the currently known complexity bounds for this class of problems by a factor of $n^{-1/3}$ in terms of the number of gradient evaluations.  Additionally, we prove that $\NRR$ converges linearly under the (global) Polyak-{\L}ojasiewicz condition and in the interpolation setting. We further complement these non-asymptotic results and provide an in-depth analysis of the asymptotic properties of $\NRR$. Specifically, under the (local) Kurdyka-{\L}ojasiewicz  inequality, the whole sequence of iterates generated by $\NRR$ is shown to converge to a single stationary point. Moreover, we derive last-iterate convergence rates that can match those in the smooth, strongly convex setting. Finally, numerical experiments are performed on nonconvex classification tasks to illustrate the efficiency of the proposed approach. 
\end{abstract}


\section{Introduction}
	In this work, we consider the composite optimization problem
	\begin{equation}
		\label{SO} \min_{w\in \Rn}~\psi(w) := f(w) + \vp(w),
	\end{equation}
	where $f:\Rn\to\R$ is differentiable but not necessarily convex and $\vp:\Rn\to\Rex$ is a weakly convex, lower semicontinuous (lsc.), and proper mapping. Composite models of the form \cref{SO} are ubiquitous in structured large-scale applications and stochastic optimization, including, e.g., machine learning \citep{bottou-2010,bottou2018optimization,lecun2015}, statistical learning and sparse regression \citep{tibshirani1996regression,hastibfri09,SST2011}, image and signal processing \citep{ComPes11,ChaPoc11}.
 
    We are interested in the case where the smooth part of the objective function has a finite-sum structure, i.e., $f$ can be represented as follows:
	\begin{equation}
		\label{finite-sum}
		f(w):=\frac{1}{n}\sum_{i=1}^{n}f(w,i).
	\end{equation}
In machine learning tasks, the number of component functions, $n$, is typically connected to the number of underlying data points and is often prohibitively large. In this scenario, computing the full gradient $\nabla f$ is highly expensive or even impossible. Nonetheless, many classical approaches for solving \cref{SO} require exact gradient information in each step, see \citep{lions1979splitting,mine1981prox,AttBolRedSou10,AttBolSva13}. In this paper, we adopt a stochastic approximation-based perspective and assume that the gradient of only one single function $f(\cdot,i)$ is available at each iteration. Furthermore, the evaluation of the proximity operator of $\vp$ is assumed to be tractable (as is the case in many applications, \citep{combettes2005signal}). Our main objective is to develop a novel proximal random reshuffling algorithm for problem \cref{SO} with convincing practical and theoretical properties.
 
\subsection{Related Works and Motivations}

\noindent\textit{Smooth problems.} When $\vp \equiv 0$, \cref{SO} reduces to a standard smooth optimization problem. 
In this case, the stochastic gradient descent method ($\SGD$) proposed in the seminal work by \citet{robbins1951stochastic} is the prototypical approach for solving \cref{SO}. The core step of $\SGD$ is given by:
  \begin{equation}
	\label{alg:SGD}
	w^{k+1}=w^k - \alpha_k \nabla f(w^k,i_k),\quad \text{where $i_k$ is chosen randomly from $[n]$.}  
	\tag{$\SGD$}
\end{equation} 
$\SGD$ has been studied extensively  during the past decades; we refer to \citep{chung1954stochastic,rakhlin2012making,nguyen2018sgd,gower2019sgd,gower2021sgd}. In the strongly convex setting, $\SGD$ was proven to converge to the optimal solution $w^*$ almost surely with the rate $\|w^k-w^*\|=\cO(1/\sqrt{k})$, \citep{chung1954stochastic}. Moreover, the lower bounds derived in \citep{nemirovskij1983,agarwal2009lower,nguyen2019lower} indicate that the rate $\cO(1/\sqrt{k})$ is tight for $\SGD$ up to a constant. In the nonconvex case, convergence of $\SGD$ can be expressed in terms of complexity bounds,  
\[
	\min_{k=1,\ldots,T} \; \Exp[\|\nabla f(w^k)\|^2] = \cO(T^{-1/2}), 	
\]
see, e.g., \citep{ghadimi2013stochastic}. Recently, empirical evidence in \citep{bottou2012,shamir2016,mishchenko2020random} suggests that variants of $\SGD$ with \emph{without-replacement sampling} can attain faster convergence. Incorporating such sampling scheme leads to the so-called random reshuffling method ($\RR$), which is shown below: 
 \begin{equation}
	\label{RR}
	\left[ 
	\begin{array}{l} 
		\text{Set } w_1^k  = w^k \text{ and generate a random permutation $\pi^k$ of $[n]$;}\\[1ex]
		\;\;\;\textbf{For}\;\;i=1,2,\ldots,n\;\;\textbf{do:}
		\quad w_{i+1}^k  =w_i^k-\alpha_k \nabla  f(w_i^k,\pi_i^k); \\[1ex]
		\text{Set } w^{k+1} = w_{n+1}^k.
	\end{array}
	\right.
	\tag{$\RR$}
\end{equation}
For problems with the finite-sum structure \cref{finite-sum}, $\RR$ exhibits superior theoretical guarantees compared to $\SGD$ \citep{gurbu2019,mishchenko2020random,nguyen2021unified,li2023convergence}. In \citep{gurbu2019}, the first theoretical guarantee for faster convergence of $\RR$ is provided. In particular, in the strongly convex case, the sequence of q-suffix averaged iterates is shown to converge to the unique optimal solution at a rate of $\cO(1/k)$ with high probability. Following this pioneering work, subsequent research has began to explore and understand the theoretical behavior of $\RR$, \citep{haochen2019random,nagaraj2019sgd,mishchenko2020random,nguyen2021unified}. In the nonconvex case and under a general variance bound, the authors in \citep{mishchenko2020random,nguyen2021unified} established the complexity bound 
\[
	{\min}_{k=1,\ldots,T} \; \Exp[\|\nabla f(w^k)\|^2] = \cO(n^{-1/3}T^{-2/3}),	
\]
for $\RR$, when a uniform without-replacement sampling scheme\footnote{Every $\nabla f(\cdot,i)$, $i\in[n]$ has equal probability to be selected without-replacement.} is applied. 
In addition, \cite{nguyen2021unified} show $\liminf_{k\to\infty} \|\nabla f(w^k)\| = 0$ for step sizes $\{\alpha_k\}_k$ satisfying $\sum_{k=0}^{\infty} \alpha_k = \infty$ and $\sum_{k=0}^\infty \alpha_k^3 < \infty$. In \citep{li2023convergence}, this was strengthened to full asymptotic convergence, $\nabla f(w^k) \to 0$, and  to last-iterate convergence of the form  $\|w^k-w^*\|=\cO(1/k)$, $k\to \infty$, under the Kurdyka-{\L}ojasiewicz (KL) inequality. Here, $w^*$ generally denotes a stationary point of the problem $\min_w f(w)$, i.e., $\nabla f(w^*) = 0$. \\[1mm]
\noindent\textit{Composite problems.} The proximal stochastic gradient method ($\PSGD$) is a standard stochastic approach for solving the composite problem \cref{SO}. The update of $\PSGD$ is given by:
 \begin{equation}
	\label{alg:PSGD}
	w^{k+1}=\proxi{\alpha_k}(w^k - \alpha_k \nabla f(w^k,i_k)),\quad \text{where $i_k$ is chosen randomly 
 from $[n]$.}
		\tag{$\PSGD$}
\end{equation}
In contrast to $\SGD$, the convergence behavior of $\PSGD$ is less understood, especially when $f$ is nonconvex. \cite{davis2019stochastic} present one of the first complexity results for $\PSGD$ in the nonconvex setting, i.e., it holds that 
\begin{equation} \label{eq:comp-proxsgd} {\min}_{k=1,\dots,T} \ \Exp[\|\mathcal G_\lambda(w^{k})\|^2] = \mathcal O(\textstyle{\sum}_{k=1}^{T}\alpha_k^2 \,/\, \textstyle{\sum}_{k=1}^{T}\alpha_k), \quad \lambda > 0, \end{equation}
where $\mathcal G_\lambda(w) := \lambda^{-1}(w - \prox{\lambda\vp}(w-\lambda\nabla f(w)))$ is a basic stationarity measure for \cref{SO}. Earlier studies of $\PSGD$ for nonconvex $f$ appeared in \citep{GhaLanZha16}, where convergence is shown if the variance vanishes as $k \to \infty$. Asymptotic convergence guarantees are discussed in \citep{MajMiaMou18,duchiruan2018,davis2020stochastic,LiMil22} under (almost) sure boundedness of the iterates $\{w^k\}_k$ or global Lipschitz continuity of $\vp$. The easier convex and strongly convex cases have been investigated, e.g., in \citep{GhaLanZha16,AtcForMou17,rosasco2020convergence,patrascu2021stochastic}. If $f$ is convex and $\psi$ is strongly convex, convergence in expectation to the unique solution $w^*$, $\Exp[\|w^k - w^*\|^2] = \mathcal O(1/k)$, can be ensured, \citep{rosasco2020convergence,patrascu2021stochastic}.

Naturally and in order to improve the performance of $\PSGD$, we may consider suitable combinations of without-replacement sampling schemes and $\PSGD$. An intuitive and straightforward extension of $\RR$ is to perform additional proximal steps ``$\prox{\alpha_k\vp}(\cdot)$'' after each inner iteration $i= 1,\dots,n $ of $\RR$. However, as highlighted in \cite[Example 1]{mishchenko2022proximal}, such a combination prevents the accurate approximation of the full gradient $\nabla f$ after one epoch. Based on this observation, \cite{mishchenko2022proximal} develop a different proximal-type $\RR$. Unlike $\PSGD$, the proposed method performs a proximal step only once after each epoch. Therefore, we will refer to this approach as epoch-wise proximal random reshuffling method ($\EPRR$). The main iterative step of $\EPRR$ is shown below:
\begin{equation}
\label{alg:PRR}
\left[ 
\begin{array}{l} 
	\text{Set } w_1^k  = w^k \text{ and generate a random permutation $\pi^k$ of $[n]$;}\\[1ex] 
	\;\;\;\textbf{For}\;\;i=1,2,\ldots,n\;\;\textbf{do:}
	\quad w_{i+1}^k  = w_i^k-\alpha_k \nabla  f(w_i^k,\pi_i^k); \\[1ex]
	\text{Set } w^{k+1} = \proxi{n\alpha_k}(w_{n+1}^k).
\end{array}
\right. 
\tag{$\EPRR$}
\end{equation}
In the nonconvex setting and under an additional bound connecting $\nabla f$ and the stationarity measure $\mathcal G_{\lambda}$, \citet{mishchenko2022proximal} derive the complexity bound\footnote{The explicit bound in \cite[Theorem 3]{mishchenko2022proximal} is ${\min}_{k=1,\ldots,T}~\Exp[\|\mathcal G_{n\gamma}(w^k)\|^2] = \cO((n\gamma T)^{-1} + n^2\gamma^2 + n \gamma^2)$ with step sizes $\alpha_k = \gamma \lesssim \frac{1}{\sL n}$. By replacing $\lambda = n\gamma$, we obtain ${\min}_{k=1,\ldots,T}~\Exp[\|\mathcal G_{\lambda}(w^k)\|^2] = \cO(T^{-2/3}+n^{-1}T^{-2/3})$ for the optimal choice $\lambda \sim T^{-1/3}$.} 
\begin{equation} \label{eq:complexity-eppr}
{\min}_{k=1,\ldots,T}~\Exp[\|\mathcal G_{\lambda}(w^k)\|^2] = \cO(T^{-2/3}+n^{-1}T^{-2/3}), \quad \lambda \sim T^{-1/3}, 
\end{equation}
for $\EPRR$, where $\lambda > 0$ is a step size parameter.  In a recent study of $\EPRR$, \citet{liu2024last} provide additional convergence rates for the objective function values $\{\psi(w^k)\}_k$ in the convex setting (i.e., each $f(\cdot,i)$ and $\vp$ are assumed to be convex).  
\begin{table}[t]
\centering
{\footnotesize

\begin{tabular}{|C{1.3cm}|C{3.47cm}C{2.6cm}C{3cm}|C{2.86cm}|}
\hline
\multirow{2}{*}{\textbf{Alg.}} & \multicolumn{3}{c|}{\textbf{Convergence: Nonconvex Setting}} & \multirow{2}{*}{\textbf{Reference}} \\ \cline{2-4}
& complexity\,${}^{\textcolor{red}{*}}$ & global conv. & local rate (KL) & \\ \hline
\vspace{5mm}\multirow{3}{*}{\parbox{1.3cm}{\centering$\RR$\\ {\color{blue} (smooth, $\vp\equiv 0$)}}} & 
\cellcolor{blue!3}{ \vspace{.5mm} $\frac{\sL \sqrt{n}}{\varepsilon^2}\max\{\sqrt{n},\frac{\sqrt{\sA}+\sB}{\varepsilon}\}$\,${}^{\textcolor{red}{\text{(a)}}}$} & \vspace{.5mm}
$-$ & \vspace{.5mm}
\cellcolor{blue!3}$-$ & \vspace{-1mm}
\cite{mishchenko2020random} \\[-1mm]
& \cellcolor{blue!3}$\frac{\sL \sqrt{n}}{\varepsilon^2}\cdot \frac{\sB\sqrt{\sA/n+1}}{\varepsilon}$\,${}^{\textcolor{red}{\text{(b)}}}$ & $\|\nabla f(w^k)\| \to 0$\,${}^{\textcolor{red}{\text{(c)}}}$ & \cellcolor{blue!3}$-$ & \cite{nguyen2021unified} \\[2.5mm]
& \cellcolor{blue!3}$-$ & $\|\nabla f(w^k)\| \to 0$ & \cellcolor{blue!3}$\|w^k-w^*\| = \mathcal O(\frac{1}{k})$ & \cite{li2023convergence} \\[.5mm] \hline \vspace{1cm}
\multirow{3}{*}{$\PSGD$} & \vspace{.5mm}
\cellcolor{blue!3}{$\frac{\sL}{\varepsilon^2}\max\{1,\frac{\sB^2}{\varepsilon^2}\}$\,${}^{\textcolor{red}{\text{(d)}}}$} & \vspace{.5mm}
{$-$} & 
\vspace{.5cm} \cellcolor{blue!3} & \vspace{-0.4mm}
\cite{davis2019stochastic} \\[-8mm] &\multicolumn{2}{l}{\dots\dots\dots\dots\dots\dots\dots\dots\dots\dots\dots\dots\dots\dots\dots}&&\dots\dots\dots\dots\dots\dots\\[2mm]
&\vspace{4mm} \cellcolor{blue!3}{$-$} &\vspace{4mm} $\|\mathcal G_\lambda(w^k)\| \to 0^{\textcolor{red}{\text{(e)}}}$ & \vspace{-2mm}\cellcolor{blue!3}{\xmark} &\vspace{-4mm} \cite{duchiruan2018,MajMiaMou18,davis2020stochastic,LiMil22} \\ \hline \vspace{2mm}
{$\EPRR$} &  \vspace{.5mm}
\cellcolor{blue!3}$\frac{\sL \sqrt{n}}{\varepsilon^2}\max\{\sqrt{n},\frac{\sB}{\varepsilon},\frac{\sqrt{n}\zeta}{\sqrt{\sL}\varepsilon}\}$\,${}^{\textcolor{red}{\text{(f)}}}$ &  \vspace{.5mm}
\xmark &  \vspace{.5mm}
\cellcolor{blue!3}\xmark &  \vspace{.1mm} 
\cite{mishchenko2022proximal} \\[-2mm] 
& \cellcolor{blue!3} & & \cellcolor{blue!3} & \\ \hline
\vspace{.5mm}{$\mathsf{norm}\text{-}$ $\mathsf{PRR}$} & \vspace{1.5mm}
\cellcolor{blue!3}$
   \frac{\sL \sqrt{n}}{\varepsilon^2}\max\{\sqrt{n},\frac{\sqrt{\sL}}{\varepsilon}\} 
$\,${}^{\textcolor{red}{\text{(g)}}}$ & \vspace{1mm}
$\dist(0,\partial \psi(w^k))$ $\to 0$ & \vspace{.5mm}
\cellcolor{blue!3}$\|w^k-w^*\| = \mathcal O(\frac{1}{k})$ & \vspace{3mm}
this work \\[-2.5mm]
& \cellcolor{blue!3} & & \cellcolor{blue!3}$|\psi(w^k)-\psi^*| = \mathcal O(\frac{1}{k^2})$ \vspace{-1.5mm}& \\
\hline
\end{tabular}
}
\caption{Comparison of convergence guarantees for $\RR$ and proximal-type methods.}
\begin{tablenotes}
\item  {\footnotesize \textcolor{red}{\text{*}} This column shows the number of individual gradient evaluations $K = T$ ($\PSGD$) and $K = nT$ ($\RR$, $\EPRR$, $\NRR$) required to reach an $\varepsilon$-accurate solution satisfying $\min_{k=1,\dots, T} \Exp[\|\mathcal G_\lambda(w^k)\|] \leq \varepsilon$ where $\lambda$ is a (step size) parameter. Note that both $\PSGD$ and $\NRR$ execute a proximal step after each (stochastic) gradient step, whereas $\EPRR$ only performs a proximal step after each epoch. Hence, $\EPRR$ generally has a better complexity in terms of proximity operator evaluations.}
\item {\footnotesize \textcolor{red}{\text{(a)}} Based on the variance condition $\frac{1}{n} \sum_{i=1}^n \|\nabla f(w,i)-\nabla f(w)\|^2 \leq 2\sA[f(w)-\flb]+\sB^2$.}
\item {\footnotesize \textcolor{red}{\text{(b)}}  Based on the assumption $\frac{1}{n} \sum_{i=1}^n \|\nabla f(w,i)-\nabla f(w)\|^2 \leq \sA\|\nabla f(w)\|^2+\sB^2$ and if $\sqrt{n} \lesssim \frac{\sB}{\varepsilon}$.} 
\item {\footnotesize\textcolor{red}{\text{(c)}} \cite{nguyen2021unified} provide the weaker result $\liminf_{k\to\infty}\|\nabla f(w^k)\|=0$.} 
\item {\footnotesize \textcolor{red}{\text{(d)}} For $\PSGD$, $\sL$ denotes the Lipschitz constant of $\nabla f$; For the other $\RR$-based methods, $\sL$ denotes the common Lipschitz constant of all $\nabla f(\cdot,i)$. Based on the condition $\frac{1}{n} \sum_{i=1}^n \|\nabla f(w,i)-\nabla f(w)\|^2 \leq \sB^2$ and $\lambda \lesssim \frac1{\rho}$, where $\rho$ is the weak convexity parameter of $\vp$.} 
\item {\footnotesize \textcolor{red}{\text{(e)}} The results in \citep{duchiruan2018,MajMiaMou18,davis2020stochastic} require almost sure boundedness of $\{w^k\}_k$; Alternatively, the analysis in \citep{LiMil22} uses Lipschitz continuity of $\vp$.} 
\item {\footnotesize \textcolor{red}{\text{(f)}} The result in \citep{mishchenko2022proximal} holds for the choice $\lambda = \frac{1}{\sL}\min\{\frac15,\frac{\varepsilon}{n^{-1/2}\sB+\sL^{-1/2}\zeta}\}$ if $\|\nabla f(w)\|^2 \leq \|\mathcal G_{\lambda}(w)\|^2 + \zeta^2$ for all $w \in \mathrm{dom}(\vp)$ and if $\frac{1}{n} \sum_{i=1}^n \|\nabla f(w,i)-\nabla f(w)\|^2 \leq \sB^2$ for all $w \in \Rn$.}   \item {\footnotesize \textcolor{red}{\text{(g)}} Our work does not make any explicit bounded variance assumptions. Instead, in \cref{lem:merit-descent-1}, it is shown that when $\alpha_k = \frac{\eta_k}{n}$, $\sum_{i=1}^\infty \eta_k^3 \lesssim \frac{1}{\sL^3}$, $\lambda \lesssim \frac{1}{\sL}$, we have $\frac{1}{n}\sum_{i=1}^{n}\|\nabla f(w^k,i) - \nabla f(w^k)\|^2 \leq \sB^2$ with $\sB \lesssim \sqrt{\sL}$; see also \cref{thm:compexity} and \cref{remark:table}. Note that we provide complexity bounds in terms of $\dist(0,\partial \psi(w^k))^2$, cf. \cref{coro:compexity}. By \eqref{eq:stat}, these bounds can be readily expressed using the natural residual $\|\mathcal G_\lambda(w^k)\|^2$ if $\lambda \lesssim \frac1{\rho}$.}
\end{tablenotes}

\label{table:super-nice}
\end{table}

\subsection{Contributions}

We design a new proximal random reshuffling method ($\NRR$) for nonconvex composite problems. In contrast to existing stochastic proximal methods, our approach is based on the so-called \emph{normal map} which swaps the order of evaluating $\nabla f$ and the proximity operator $\prox{\lambda\vp}$. We show that this exchanged order generally exhibits a better compatibility with without-replacement sampling schemes. In particular, similar to $\PSGD$ but different from $\EPRR$, $\NRR$ performs proximal steps at each inner iteration which allows {maintaining feasibility} or the {structure} induced by $\vp$. We now list some of our core contributions:

\begin{itemize}
\item We derive finite-time complexity bounds for $\NRR$ in the nonconvex, nonsmooth setting. In contrast to previous related works \citep{davis2019stochastic,mishchenko2022proximal}, our convergence results are formulated in terms of the subdifferential $\partial \psi$ of the original objective function $\psi$, rather than using the natural stationarity measure $\mathcal{G}_\lambda$. Specifically, under standard assumptions, we establish 
\begin{equation} \label{eq:contri-exp}
\begin{aligned}
    & {\min}_{k=1,\ldots,T}\;\dist(0,\partial \psi(w^k))^2 = \cO(T^{-2/3}) \;\;\; \text{and} \;\;\; \\ & {\min}_{k=1,\ldots,T}\;\Exp[\dist(0,\partial \psi(w^k))^2] = \cO(n^{-1/3}\,T^{-2/3}),
\end{aligned}
\end{equation}
where $T$ counts the total number of iterations\footnote{The measures $\mathcal{G}_\lambda$ and $\dist(0,\partial\psi(\cdot))$ are closely connected, i.e., we have $(1-\lambda\rho)\|{\mathcal G}_\lambda(w^k)\| \leq \dist(0,\partial\psi(w^k))$ for all $k$ and $\lambda < {\rho^{-1}}$, cf. \cref{eq:stat}. Hence, our results can also be naturally expressed in terms of $\mathcal G_\lambda$.}. The first result in \cref{eq:contri-exp} is a worst-case deterministic complexity bound that is applicable to any shuffling scheme. The second result shown in \cref{eq:contri-exp} holds for a uniform without-replacement sampling strategy. Both bounds match those of $\RR$ in the nonconvex smooth setting, \citep{mishchenko2020random,nguyen2021unified}. Compared to the complexity \cref{eq:complexity-eppr} for $\EPRR$\footnote{The bounds \cref{eq:complexity-eppr} and \cref{eq:contri-exp} capture the complexity of $\EPRR$ and $\NRR$ in terms of gradient evaluations. By design, $\EPRR$ has a better per-iteration dependence on proximal evaluations compared to $\NRR$.}, the in-expectation bound for $\NRR$ has a better dependence on $n$. We further prove that the sequence $\{\psi(w^k)\}_k$ can converge linearly to an optimal function value $\psi(w^*)$ under the Polyak-{\L}ojasiewicz condition and in the interpolation setting $\nabla f(w^*,1) = \dots = \nabla f(w^*,n)$.
\item  We provide an in-depth asymptotic analysis of $\NRR$. For diminishing step sizes $\{\alpha_k\}_k$ with $\sum_{k=0}^\infty \alpha_k = \infty$ and $\sum_{k=0}^\infty \alpha_k^3 < \infty$, we derive $\dist(0,\partial \psi(w^k)) \to 0$ and $\psi(w^k) \to \bar \psi$.
To our knowledge, this is the first asymptotic convergence guarantee for a proximal-type $\RR$ method. Moreover, when $\psi$ satisfies the KL inequality, we establish convergence of the whole sequence of iterates, i.e., $w^k \to w^*$, where $w^*$ is a stationary point of $\psi$. We further quantify the asymptotic rate of convergence of $\NRR$ when using polynomial step sizes $\alpha_k \sim k^{-\gamma}$, $\gamma\in(\half,1]$. The derived rates depend on the local geometry of $\psi$ around the stationary point $w^*$ which is captured by the KL exponent $\theta\in[0,1)$. In the case $\gamma = 1$, $\theta\in[0,\frac12]$, we obtain
\[ \|w^k-w^*\|=\cO(k^{-1}),\quad \dist(0,\partial \psi(w^k)) = \cO(k^{-1}),\quad |\psi(w^k)-\psi(w^*)|=\cO(k^{-2}). \]
These rates match existing results in the smooth or strongly convex setting, cf${.}$ \citep{gurbu2019,li2023convergence}. Comparable asymptotic rates for $\PSGD$ and $\EPRR$ do not seem to be available in the general nonconvex case. 
\item Finally, we present experiments that corroborate our theoretical findings on feasibility and potential linear convergence of $\NRR$ and we conduct numerical comparisons on a nonconvex binary classification and deep learning image classification problem.
\end{itemize}

An additional overview and discussion of the obtained results is provided in \cref{table:super-nice}.

\subsection{Basic Notations}\label{subsec:notation}
By $\iprod{\cdot}{\cdot}$ and $\|\cdot\| := \|\cdot\|_2$, we denote the standard Euclidean inner product and norm. For $n\in\N$, we define $[n] := \{1,2,\dots,n\}$. By convention, we set ${\sum}_{i=k}^{k-1}\, \eta_i =0$ for any $\{\eta_k\}_k$. 

\section{Preliminaries, the Full Algorithm, and Preparatory Lemmas}\label{sec:alg_lem}

We now present technical preliminaries, the main algorithm, and first preparatory results. 

\subsection{Basic Nonsmooth Concepts and First-Order Optimality}\label{subsec:notion}

We first recall several useful concepts from nonsmooth and variational analysis.  For a function $h : \Rn \to \Rex$, the Fr\'{e}chet (or regular) subdifferential of $h$ at $x$ is given by
\[\partial h(x) := \{g \in \Rn : h(y) \geq h(x) + \iprod{g}{y-x}+o(\|y-x\|) \;\; \text{as} \;\; y \to x\}, \] 
see, e.g., \cite[Chapter 8]{RocWet98}. If $h$ is convex, then the Fr\'{e}chet subdifferential coincides with the standard (convex) subdifferential. The mapping $h$ is said to be $\rho$-weakly convex if $h+\frac{\rho}{2}\|\cdot\|^2$, $\rho > 0$, is convex. The $\rho$-weak convexity of $h$ is equivalent to 
\begin{equation}
	\label{eq:weak-cvx}
	h(y) \geq h(x) + \iprod{s}{y-x} - \frac{\rho}{2}\norm{x-y}^2, \quad \forall~x, y, \;\; \forall~s\in \partial h(x),
\end{equation}
see, e.g., \citep{Via83,davis2019stochastic}. 

The first-order necessary optimality condition for the composite problem \cref{SO} is given by
\begin{equation} \label{eq:first-opti} 0 \in \partial \psi(w) = \nabla f(w) + \partial\varphi(w). \end{equation}
A point satisfying this inclusion is called a stationary point and  $\crit{\psi} := \{w \in \dom{\varphi}: 0 \in \partial\psi(w)\}$ denotes the set of all stationary points of $\psi$. It is well-known that the condition \cref{eq:first-opti} can be equivalently represented as a nonsmooth equation, \citep{RocWet98},
\[
\mathcal G_\lambda(w) := \lambda^{-1}({w - \prox{\lambda\varphi}(w - \lambda \nabla f(w))}) = 0, \quad \lambda > 0,
\]
where $\mathcal G_\lambda$ is the so-called \textit{natural residual}. The stationarity measure $\mathcal G_\lambda$ is widely used in the analysis of proximal methods.
If $\vp$ is $\rho$-weakly convex, then the proximity operator $\proxl(w):= {\argmin}_{y \in \Rn}~\vp(y) + \frac{1}{2\lambda}\|w-y\|^2$, $\lambda \in (0,\rho^{-1})$,
is $(1-\lambda\rho)$-cocoercive, i.e.,  
\begin{equation} \label{eq:prox-coco} \iprod{w-y}{\prox{\lambda\vp}(w)-\prox{\lambda\vp}(y)}\geq (1-\lambda\rho)\norm{\prox{\lambda\vp}(w)-\prox{\lambda\vp}(y)}^2, \quad \forall~w,y, \end{equation}
see, e.g., \cite[Proposition 3.3]{HohLabObe20}. In particular, $\prox{\lambda\vp}$ is Lipschitz continuous with modulus $(1-\lambda\rho)^{-1}$. In this work, we use the \emph{normal map} \citep{robinson1992normal}, 
\be \label{eq:def-normal} \Fnor(z) := \nabla f(\proxl(z)) + \lambda^{-1}(z-\proxl(z)) \in \partial \psi(\proxl(z)), \quad \lambda > 0, \ee
as an alternative stationarity measure for \cref{SO}. Here, the condition $\Fnor(z) \in \partial \psi(\proxl(z))$ follows directly from $z-\proxl(z) \in \lambda\partial\vp(\proxl(z))$, see \cite[Proposition 3.1]{HohLabObe20}. The normal map and the natural residual are closely related via
\be \label{eq:stat} (1-\lambda\rho)\|\mathcal G_\lambda(\proxl(z))\| \leq \dist(0,\partial\psi(\proxl(z))) \leq \|\Fnor(z)\| \quad \forall~z, \ee
where the first inequality can be shown by applying \cref{eq:prox-coco} and following the proof of \cite[Theorem 3.5]{drulew18}. Throughout this work and motivated by \cref{eq:stat}, we will typically measure and express convergence in terms of the distance $w \mapsto \dist(0,\partial\psi(w))$.

\begin{algorithm}[t]
	\caption{{$\NRR$}: \textbf{Nor}mal \textbf{m}ap-based \textbf{p}roximal \textbf{r}andom \textbf{r}eshuffling}\label{algo:nprr}
	\KwIn{Initial point $z^1\in\Rn$, $w^1 = \proxi{\lambda}(z^1)$ and parameters $\{\alpha_k\}_k\subset \R_{++}$, $\lambda>0$;}
	\BlankLine
	\For{$k=1,2,\ldots$}{
		Generate a permutation $\pi^{k}$ of $[n]$. Set $z_1^k = z^k$ and $w_1^k = w^k$\;\vspace{1mm}
		\For{$i = 1,2,\ldots, n$}
		{Compute $z_{i+1}^{k} = z_{i}^{k}-\alpha_{k}({\nabla {f}( w_{i}^{k},\pi^{k}_{i})+\frac{1}{\lambda}(z_i^k - w_i^k)})$ and
			$w_{i+1}^k = \proxi{\lambda}(z_{i+1}^{k})$\;
		}
		Set $z^{k+1} =  z_{n+1}^{k}$ and $w^{k+1} =  w_{n+1}^{k}$\;
	}
\end{algorithm}

\subsection{Algorithm Design}\label{subsec:alg}

\noindent\emph{From proximal gradient to normal map steps.} We motivate our approach by introducing an alternative representation of the traditional proximal gradient descent ($\PGD$) method, \citep{mine1981prox}, that separates the gradient and proximal steps. Let us define the auxiliary variable $z^{k+1} = w^k - \lambda \nabla f(w^k)$. The $\PGD$ update, $w^{k+1} = \prox{\lambda}(w^k - \lambda \nabla f(x^k))$, can then be expressed in the following form:
\begin{equation}\label{eq:equiv-form}
 z^{k+1} = z^{k} - \alpha [\nabla f(w^{k}) + \lambda^{-1}(z^{k} - w^k)] \;\;\; \text{and} \;\;\; w^{k+1}  = \prox{\lambda\vp}(z^{k+1}) \;\;\; \text{where} \;\;\; \alpha=\lambda.
\end{equation}
This equivalent formulation naturally introduces the \emph{normal map} $\Fnor(z) = \nabla f(w) + \lambda^{-1}(z-w)$ where $w = \prox{\lambda\vp}(z)$. Normal maps have been extensively used in the context of classical variational inequalities for the special case where the proximity operator $\prox{\lambda\vp}$ is given as the Euclidean projection onto a closed, convex set. We refer to \citep{facpan03} for more detailed background. \vspace{.5ex}

\noindent\emph{Advantages.} A remarkable feature of the normal map is its direct connection to the subdifferential of the objective function $\psi$. Specifically, based on \cref{eq:def-normal}, we have $\Fnor(z^k) \in \partial \psi(w^k)$, i.e., the normal map $\Fnor(z^k)$ is a special subgradient of $\psi$ at $w^k$. By contrast, it holds that $\mathcal G_\lambda(w^k) \in \nabla f(w^k) + \partial \vp(w^{k+1})$, i.e., the natural residual $\mathcal G_\lambda(w^k)$ is not necessarily a subgradient of $\psi$. Our aim is to leverage this new perspective and to study the behavior of the auxiliary iterates $\{z^k\}_k$ using the normal map $\Fnor$ as the underlying stationarity measure. Indeed, by \cref{eq:def-normal}, 
the stationarity condition $\|\Fnor(z^k)\| < \varepsilon$ immediately ensures that $\dist(0,\partial \psi(w^k))<\varepsilon$. Such a direct connection does not seem to exist between the traditional natural residual $\|\mathcal G_\lambda(w^k)\|$ and $\dist(0,\partial \psi(w^k))$.

In addition, the formulation \cref{eq:equiv-form} facilitates the decoupling of the proximal parameter $\lambda$ from the step size $\alpha$. In particular, the constant step size $\alpha$ can be substituted with a sequence of varying step sizes $\{\alpha_k\}_k$ without necessitating adjustments of the proximal parameter $\lambda$. This methodological flexibility can be advantageous in stochastic settings where diminishing step sizes are typical to mitigate stochastic errors. \vspace{.5ex}

\noindent\emph{Incorporating reshuffling.} We now discuss the full procedures of $\NRR$. Let $\Pi = \{\pi: \pi \text{ is a permutation of } [n]\}$ denote the set of all possible permutations of $[n]$. At each iteration $k$, a permutation $\pi^k$ is sampled from $\Pi$. The algorithm then updates $w^{k}$ to $w^{k+1}$ through $n$ consecutive normal map-type steps by using the stochastic gradients $\{\nabla f(\cdot, \pi^k_1), \ldots, \nabla f(\cdot, \pi^k_{n})\}$ sequentially: 
\be \label{eq:norm-am-update} z_{i+1}^{k} = z_{i}^{k}-\alpha_{k}({\nabla {f}( w_{i}^{k},\pi^{k}_{i})+\lambda^{-1}(z_i^k - w_i^k)}) \quad \text{and} \quad w_i^k = \proxl(z_i^k), \quad i = 1,\dots,n. \ee
Here, $\pi^k_i$ represents the $i$-th element of the permutation $\pi^k$ and the term $\nabla {f}( w_{i}^{k},\pi^{k}_{i})+\lambda^{-1}(z_i^k - w_i^k)$ approximates $\Fnor(z_i^k)$ by replacing the true gradient $\nabla f(w_i^k)$ with the component gradient $\nabla {f}( w_{i}^{k},\pi^{k}_{i})$. In each step of $\NRR$, only one single gradient component $\nabla f(\cdot, \pi^k_i)$, $i \in [n]$ and one proximal operator is evaluated. The pseudocode of $\NRR$ is shown in \Cref{algo:nprr}. In the case $\vp\equiv0$, $\NRR$ coincides with the original $\RR$ method. 

Based on \cref{eq:norm-am-update}, we can express the update of $z^k$ compactly via
\begin{equation}
		\label{eq:update-z}
		z^{k+1} = z^k -n\alpha_k \Fnor(z^k) + e^k,  
	\end{equation}
	where the error term $e^k$ is given by
	\begin{equation}
		\label{eq:error-e}
		e^k:=-\; \alpha_k \left[{\sum}_{i=1}^{n}(\Fnor(z^k_i) - \Fnor(z^k))+ {\sum}_{i=1}^{n}(\nabla f(w^k_i,\pi^k_i) - \nabla f(w^k_i)) \right].
	\end{equation}
We note that, in the deterministic case $n=1$ and $e^k=0$, $\NRR$ reduces to $\PGD$ once setting $\lambda \equiv \alpha_k$. As motivated, the procedure \cref{eq:update-z} can be interpreted as a special proximal gradient-type method with errors. 

\subsection{Assumptions and Error Estimates}\label{sec:descent}

\begin{assumption}[Functions \& Sampling] \label[assumption]{assumption:1}
	We consider the following basic conditions:

	\begin{enumerate}[label=\textup{\textrm{(F.\arabic*)}},topsep=0pt,itemsep=0ex,partopsep=0ex,leftmargin=8ex]
		\item \label{A1} Each mapping $\nabla f(\cdot,i)$, $i\in[n]$, is Lipschitz continuous on $\dom{\vp}$ with modulus $\sL>0$.
		\item \label{A2} The function $\varphi:\Rn\to(-\infty,\infty]$ is $\rho$-weakly convex, lsc., and proper. 
		\item \label{A3} There are $\flb,\vplb\in\R$ such that $f(w,i)\geq \flb$ and $\vp(w)\geq \vplb$, for all $w\in\dom{\vp}$ and $i\in[n]$. This also implies   $\psi(w) \geq \psilb :=\flb+\vplb$ for all $w\in\dom{\vp}$.
	\end{enumerate}
\begin{enumerate}[label=\textup{\textrm{(S.\arabic*)}},topsep=1ex,itemsep=0ex,partopsep=0ex,leftmargin=8ex]
		\item \label{S1} The permutations $\{\pi^k\}_k$ are sampled independently (for each $k$) and uniformly without replacement from $[n]$.
\end{enumerate}
\end{assumption}

The conditions \ref{A1} and \ref{A2} are standard in nonconvex and nonsmooth optimization, see, e.g., \citep{GhaLanZha16,davis2019stochastic,yang2021stochastic}. In \ref{A2}, we only assume $\vp$ to be weakly convex. The class of weakly convex functions is rich and allows us to cover important nonconvex regularizations, including, e.g., the student-$t$ loss \citep{aravkin2011robust}, the minimax concave penalty \citep{zhang2010nearly}, and the smoothly clipped absolute deviation penalty \citep{fan2001variable}. Combining \ref{A1} and \ref{A3} and using the descent lemma, we can derive the following bound for the gradients $\nabla f(\cdot,i)$: 
\[\norm{\nabla f(w,i)}^2 \leq 2\sL [f(w,i)-\flb],\quad \forall \; i\in [n], 
\]
see, e.g., \citep{nesterov2018lectures} or \cite[eqn. (2.5)]{li2023convergence}. Assumption \ref{S1} is a standard requirement on the sampling scheme used in random reshuffling methods, \citep{mishchenko2020random,mishchenko2022proximal,nguyen2021unified}.

We now establish first estimates for the error term $e^k$ defined in \cref{eq:update-z}. In particular, we provide a link between the errors $\{e^k\}_k$ defined in \eqref{eq:error-e}, the step sizes $\{\alpha_k\}_k$, the variance
\begin{equation}
    \label{eq:variance}
    \sigma_k^2 := \frac{1}{n}{\sum}_{i=1}^{n}\norm{\nabla f(w^k,i) - \nabla f(w^k) }^2,
\end{equation}
and the normal map $\Fnor$. To proceed, let us also formally define the filtration $\{\mathcal F_k\}_k$ where $\mathcal F_{k} := \sigma(\pi^1,\dots,\pi^{k})$ is the $\sigma$-algebra generated by the permutations $\pi^1,\dots,\pi^k$. (We may also set $\mathcal F_0 := \sigma(z^1)$). Then, it follows $z^{k+1}, w^{k+1}, e^k \in \mathcal F_k$ for all $k \geq 1$. Let us further introduce $\Exp_k[\cdot] := \Exp[\cdot\,|\,\mathcal F_{k-1}]$.

\begin{lemma}[Error Estimates]
	\label[lemma]{lem:est-err}
	Let the conditions \ref{A1}--\ref{A2} be satisfied and let $\{w^k\}_k$ and $\{z^k\}_k$ be generated by $\NRR$ with $\lambda\in(0,\frac{1}{\rho})$, $0<\alpha_k \leq \frac{1}{\sqrt{2\CL}n}$, 
 and $\CL:=4[\frac{3\sL + 2\lambda^{-1}-\rho}{1-\lambda\rho}]^2$. 
	\begin{enumerate}[label=\textup{(\alph*)},topsep=1ex,itemsep=0ex,partopsep=0ex]
		\item It holds that $\norm{e^k}^2 \leq  \CL n^4 \alpha_k^4 [\norm{\Fnor(z^k)}^2 + \sigma_k^2 \big]$ for all $k\geq 1$.
		\item Additionally, under \ref{S1}, it follows
  \[
 \Exp_k[\norm{e^k}^2] \leq  \CL n^4 \alpha_k^4 [\norm{\Fnor(z^k)}^2 + n^{-1}{\sigma_k^2}] \quad \forall~k \geq 1 \quad (\text{almost surely}). \]       
\end{enumerate}
\end{lemma}

The proof of \cref{lem:est-err} is presented in \cref{proof:est-err}. Next, we provide an upper bound for the variance terms $\{\sigma_k^2\}_k$.
\begin{lemma}[Variance Bound]
	\label[lemma]{lem:var-bound}
	Let $\{w^k\}_k \subseteq \dom{\vp}$ be given and assume that \ref{A1} and \ref{A3} are satisfied. Then, it holds that
 \[
 \sigma_k^2 \leq 2\sL[f(w^k) - \flb] \leq 2\sL[\psi(w^k) - \psilb] \quad \forall~k.
 \]
\end{lemma}

\begin{proof}
Using $\|a-b\|^2 = \|a\|^2 - 2\iprod{a}{b} + \|b\|^2$, we have 
%
\begin{align*}
     \sigma_k^2 &= \|\nabla f(w^k)\|^2 + \frac{1}{n}{\sum}_{i=1}^{n} \|\nabla f(w^k,i)\|^2 - \frac{2}{n}{\sum}_{i=1}^{n}\iprod{\nabla f(w^k,i)}{\nabla f(w^k)} \\
     &\leq \frac{1}{n}{\sum}_{i=1}^{n}\|\nabla f(w^k,i)\|^2 \leq 2\sL[f(w^k) - \flb] \leq 2\sL[\psi(w^k) - \psilb],
\end{align*}
where the last line is due to ${\sum}_{i=1}^{n}\|\nabla f(w,i)\|^2 \leq 2\sL {\sum}_{i=1}^{n} [f(w,i) -\flb] = 2\sL n [f(w) - \flb]$ and $f(w) - \flb \leq \psi(w) - \psilb$ for all $w \in \dom{\vp}$.
\end{proof}

\subsection{Merit Function and Approximate Descent}
	\label{subsec:descent-time-win}	
 Descent-type properties serve as a fundamental cornerstone when establishing iteration complexity and asymptotic convergence of algorithms. In the classical random reshuffling method and in its proximal version $\EPRR$, descent is measured directly on the objective function $f$ and $\psi$, respectively \citep{mishchenko2020random,mishchenko2022proximal,nguyen2021unified,li2023convergence}. By contrast and motivated by the subgradient condition $\Fnor(z) \in \partial\psi(\prox{\lambda\vp}(z))$, we analyze descent of $\NRR$ on an auxiliary merit function $\mer$ that is different from $\psi$. The merit function $\mer$ was initially introduced by \cite{ouyang2021trust}.

	\begin{definition}[Merit Function]\label{def:mer-fun} Let the constants $\tau,\lambda>0$ be given. The merit function $H_\tau:\Rn\to\R$ is defined as follows:
		\[ \mer(z):=\psi(\prox{\lambda\vp}(z)) + \frac{\tau\lambda}{2}\norm{\Fnor(z)}^2.\]
	\end{definition}
Provided that $\lambda\in (0,\frac{1}{4\rho})$, we will typically work with the following fixed choice of $\tau$: 
\begin{equation} \label{eq:tau} \tau:=\frac{1-4\lambda\rho}{2(1-2\lambda\rho+\lambda^2\sL^2)}. \end{equation}
We now present a first preliminary descent-type property for $\NRR$ using $\mer$. The detailed proof of \cref{lem:merit-descent-0} can be found in \cref{subsec:proof-lem:merit-descent-0}.

\begin{lemma}\label[lemma]{lem:merit-descent-0}
	Suppose \ref{A1}--\ref{A2} are satisfied and let the iterates $\{w^k\}_k$ and $\{z^k\}_k$ be generated by $\NRR$ with $\lambda\in(0,\frac{1}{4\rho})$ and  $0< \alpha_k \leq \frac{1}{10\sL n}$. Setting $\tau=\frac{1-4\lambda\rho}{2(1-2\lambda\rho+\lambda^2\sL^2)}$, it holds that
 \[\mer(z^{k+1})-\mer(z^k) 
	\leq - \frac{\tau  n \alpha_k}{4}\left[{1+\frac{ n \alpha_k}{\lambda}}\right]\norm{\Fnor(z^k)}^2 - \frac{1}{8n \alpha_k}\norm{w^{k+1}-w^k}^2 + \frac{1}{n \alpha_k} \norm{e^k}^2.\]
\end{lemma}
Based on \cref{lem:var-bound,lem:est-err,lem:merit-descent-0}, we can establish approximate descent of $\NRR$. 
\begin{lemma}[Approximate Descent] \label[lemma]{lem:merit-descent-1}
	Let \ref{A1}--\ref{A2} hold and let $\{w^k\}_k$, $\{z^k\}_k$ be generated by $\NRR$ with $\lambda\in(0,\frac{1}{4\rho})$ and step sizes $\{\alpha_k\}_k$ satisfying 
	\[0<\alpha_k \leq \frac{1}{n}\cdot \max\left\{\sqrt{2\CL},\; 10\sL,\; {4\CL\lambda\tau^{-1}} \right\}^{-1}=: \frac{\bar\alpha}{n}.\] 
Here, $\CL$ and $\tau$ are defined in \cref{lem:est-err} and \cref{eq:tau}. We set $\Delta(t):= 2\sL \CL \exp(2\sL \CL t)  [\mer(z^1) -\psilb]$.\\[-1mm]
	\begin{enumerate}[label=\textup{(\alph*)},topsep=0pt,itemsep=0ex,partopsep=0ex]
	\item The following descent-type estimate holds for all $k \geq 1$:
 \[
 \mer(z^{k+1}) \leq \mer(z^k) - \frac{1}{8 n \alpha_k}\norm{w^{k+1}-w^k}^2 - \frac{\tau n \alpha_k}{4}\norm{\Fnor(z^k)}^2 + \CL n^3\alpha_k^3 \sigma_k^2.
 \]
Furthermore, if \ref{A3} holds and we have ${\sum}_{k=1}^{\infty}\alpha_k^3<\infty$, then  $\CL\sigma_k^2 \leq \Delta(n^3\sum_{i=1}^\infty\alpha_i^3)$.
 \item Additionally, if the sampling scheme satisfies condition \ref{S1}, it holds that 
 \[
 \Exp_k[\mer(z^{k+1})] \leq  \mer(z^k) - \frac{\tau n \alpha_k}{4} \norm{\Fnor(z^k)}^2 + \CL n^2\alpha_k^3\sigma_k^2 \quad \text{(almost surely)}.
 \]
 Moreover, under \ref{A3} and ${\sum}_{k=1}^{\infty}\alpha_k^3<\infty$, we have $\CL\Exp[\sigma_k^2] \leq \Delta(n^2\sum_{i=1}^\infty\alpha_i^3)$.
%
	\end{enumerate}
%
\end{lemma}

\begin{proof} Applying \cref{lem:merit-descent-0} and \cref{lem:est-err}~(a), we obtain
\begin{equation}\label{eq:lem-descent-(a)}
    \begin{aligned}
        &\hspace{-8ex} \mer(z^{k+1})-\mer(z^k) + \frac{1}{8n \alpha_k}\norm{w^{k+1}-w^k}^2 + \frac{\tau n \alpha_k}{4}\norm{\Fnor(z^k)}^2\\
        & \leq  \left[{\CL n \alpha_k - \frac{\tau}{4\lambda} }\right] n^2\alpha_k^2\norm{\Fnor(z^k)}^2 + \CL n^3\alpha_k^3 \sigma_k^2 \leq \CL n^3\alpha_k^3 \sigma_k^2,
    \end{aligned}
\end{equation}
  where the last inequality follows from $\alpha_k \leq \frac{\tau}{4\CL\lambda n}$. Moreover, when \ref{A3} is satisfied, we can use \cref{lem:var-bound} in \cref{eq:lem-descent-(a)}. This yields
\begin{equation}\label{eq:lem-3-est-6}
\mer(z^{k+1})-\mer(z^k) + \frac{1}{8 n \alpha_k}\norm{w^{k+1}-w^k}^2 + \frac{\tau n \alpha_k}{4}\norm{\Fnor(z^k)}^2 \leq  2\sL \CL n^3\alpha_k^3 [\psi(w^k) - \psilb].
 \end{equation}
Subtracting $\psilb$ on both sides of \cref{eq:lem-3-est-6} and noting $\psi(w^k) \leq \mer(z^k)$, this further implies 
\[ \mer(z^{k+1}) - \psilb \leq (1+2\sL \CL n^3\alpha_k^3)[\mer(z^k)-\psilb] \quad \forall~k\geq 1.
\]
Hence, using $1+x \leq \exp(x)$, $x \geq 0$, we can infer    
\begin{align*}
    \mer(z^{k+1}) - \psilb & \leq {\prod}_{i=1}^{k}\prt{1 + 2\sL \CL n^3\alpha_i^3} [\mer(z^1) -\psilb]
    \\ & \hspace{-6ex} \leq \exp\Big({2\sL \CL n^3{\sum}_{i=1}^{k} \alpha_i^3}\Big) [\mer(z^1) -\psilb] \leq \exp\Big({2\sL \CL n^3{\sum}_{i=1}^{\infty} \alpha_i^3}\Big) [\mer(z^1) -\psilb]. 
\end{align*}
Thus, we have $\psi(w^k) - \psilb \leq \mer(z^k) - \psilb \leq \exp({2\sL \CL n^3{\sum}_{i=1}^{\infty} \alpha_i^3}) [\mer(z^1) -\psilb] $ for all $k \geq 1$ and it follows $\CL\sigma_k^2 \leq \Delta(n^3\sum_{i=1}^\infty\alpha_i^3)$. 

We continue with the proof of part (b). Taking the conditional expectation in \cref{lem:merit-descent-0} and using \cref{lem:est-err}~(b) and $\alpha_k \leq \frac{\tau}{4\CL\lambda n}$, we obtain
\begin{align*}
    &\hspace{-4ex}\Exp_k[\mer(z^{k+1})]-\mer(z^k) + \frac{\tau  n \alpha_k}{4} \norm{\Fnor(z^k)}^2 \\
		&\leq - \frac{\tau  n^2 \alpha_k^2}{4\lambda}\norm{\Fnor(z^k)}^2 + \frac{1}{ n \alpha_k} \Exp_k[\norm{e^k}^2] \leq \CL n^2\alpha_k^3 \sigma_k^2.
\end{align*}
Taking the total expectation and applying \cref{lem:var-bound}, this yields 
\begin{equation}\label{eq:lem-3-desent-exp}
    \Exp[\mer(z^{k+1})]\leq \Exp[\mer(z^k)] - \frac{\tau n \alpha_k}{4}\Exp[\norm{\Fnor(z^k)}^2] + 2\sL \CL n^2\alpha_k^3 \Exp[\mer(z^k) - \psilb].
\end{equation}
Mimicking the previous steps, we can infer $\Exp[\mer(z^k) - \psilb]\leq \exp({2\sL \CL n^2{\sum}_{i=1}^{\infty} \alpha_i^3}) [\mer(z^1) -\psilb] $ and $\CL\Exp[\sigma_k^2] \leq \Delta(n^2\sum_{i=1}^\infty\alpha_i^3)$ which finishes the proof. \end{proof}

\section{Iteration Complexity and Global Convergence}\label{sec:comp}
Based on the approximate descent properties of the merit function $\mer$ in
\cref{lem:merit-descent-1}, we now establish the iteration complexity of $\NRR$. Applications to constant and polynomial step sizes are presented in \cref{coro:compexity,coro:compexity-2}

\begin{theorem}[Complexity Bound for $\NRR$]\label{thm:compexity}
	Let the conditions \ref{A1}--\ref{A3} hold and let $\{w^k\}_k$ and $\{z^k\}_k$ be generated by $\NRR$ with  $\lambda\in(0,\frac{1}{4\rho})$ and step sizes $\{\alpha_k\}_k$ satisfying $\alpha_k = \frac{\eta_k}{n}$ and $0<\eta_k \leq  \bar \alpha$. Then, the following statements are valid: 
	\begin{enumerate}[label=\textup{(\alph*)},topsep=0pt,itemsep=0ex,partopsep=0ex]
	\item If $\sum_{k=1}^\infty \eta_k^3 \leq \frac{1}{2\sL \CL }$, then, for all $k\geq 1$, it holds that
\[\min_{k=1,\dots,T} \;\dist(0,\partial \psi(w^k))^2 \leq \frac{4+ 24\sL \CL{\sum}_{k=1}^{T} \eta_k^3}{\tau{\sum}_{k=1}^{T}\eta_k} \cdot [\mer(z^1)-\psilb].\]
	\item In addition, under the sampling condition \ref{S1} and if $\sum_{k=1}^\infty \eta_k^3 \leq \frac{n}{2\sL \CL }$, it holds that  
	\[ \min_{k=1,\dots,T}\; \Exp[\dist(0,\partial \psi(w^k))^2] \leq \frac{4n+ 24\sL \CL {\sum}_{k=1}^{T} \eta_k^3}{\tau n {\sum}_{k=1}^{T}\eta_k} \cdot [\mer(z^1)-\psilb].\]
\end{enumerate}
Here, 
the constants $\CL,\tau,\bar{\alpha} >0$ are defined in \cref{lem:est-err}, \cref{eq:tau}, and \cref{lem:merit-descent-1}, respectively.
\end{theorem}

\begin{proof}
	Applying \cref{lem:merit-descent-1}~(a) with $\alpha_k = \frac{\eta_k}{n}$ and dropping the term $\|w^{k+1}-w^k\|^2$, we have 
 \begin{equation}
    \label{eq:thm-com2-1}
    \begin{aligned}
        [\mer(z^{k+1}) - \psilb] + \frac{\tau \eta_k}{4}\norm{\Fnor(z^k)}^2 &\leq [\mer(z^k)- \psilb]   + \Delta({\textstyle \sum_{i=1}^\infty} \eta_i^3) \cdot \eta_k^3\\
        & \leq [\mer(z^k) - \psilb]   + 6\sL \CL [\mer(z^1)-\psilb]\eta_k^3,
    \end{aligned}
\end{equation}
where the last line is due to $\Delta({\textstyle \sum_{i=1}^\infty} \eta_i^3) \leq 2\sL \CL \exp(1) [\mer(z^1) -\psilb] \leq 6\sL \CL   [\mer(z^1) -\psilb]$. Summing \cref{eq:thm-com2-1} from $k=1$ to $T$, we obtain
  \[
		\frac{\tau}{4} {\sum}_{k=1}^{T}\eta_k\norm{\Fnor(z^k)}^2 \leq [\mer(z^1) -\psilb] + 6\sL \CL [\mer(z^1)-\psilb]{\sum}_{k=1}^{T} \eta_k^3,
  \]
	which further implies $\min_{k=1,\ldots,T} \norm{\Fnor(z^k)}^2 \leq \frac{4[\mer(z^1)-\psilb]+ 24\sL \CL [\mer(z^1)-\psilb]{\sum}_{k=1}^{T} \eta_k^3}{\tau{\sum}_{k=1}^{T}\eta_k}$. Noticing $\dist(0,\partial \psi(w^k)) \leq \norm{\Fnor(z^k)}$, this completes the proof of part (a).

 To prove (b), we invoke \cref{lem:merit-descent-1}~(b), take total expectation, and set $\alpha_k = \frac{\eta_k}{n}$:
 \begin{align*}
    \Exp[\mer(z^{k+1}) - \psilb] + \frac{\tau \eta_k}{4}\cdot \Exp[\norm{\Fnor(z^k)}^2] &\leq \Exp[\mer(z^k)- \psilb]   + \Delta(n^{-1}{\textstyle \sum_{i=1}^\infty} \eta_i^3) \cdot n^{-1}\eta_k^3\\
    & \leq \Exp[\mer(z^k) - \psilb]   + 6\sL \CL [\mer(z^1)-\psilb]n^{-1}\eta_k^3.
\end{align*}
The rest of the verification is identical to part (a).
  \end{proof}

We now study specific complexity results for $\NRR$ under constant step size schemes. 

\begin{corollary}\label[corollary]{coro:compexity}
Assume \ref{A1}--\ref{A3} and let $\{w^k\}_k$, $\{z^k\}_k$ be generated by $\NRR$ with $\lambda\in(0,\frac{1}{4\rho})$ and $\alpha_k \equiv \frac{\alpha}{n}$ for all $k$.
	\begin{enumerate}[label=\textup{(\alph*)},topsep=0pt,itemsep=0ex,partopsep=0ex]
	\item If $\alpha=\frac{\eta}{T^{1/3}}$ with $0<\eta\leq\min\{(2\sL \CL )^{-\frac13},\bar\alpha T^{\frac13}\}$, then we have
\[
		{\min}_{k=1,\ldots,T} \; \dist(0,\partial \psi(w^k))^2 \leq {16}({\tau \eta T^{2/3}})^{-1}[\mer(z^1) -\psilb]= \cO({{T^{-2/3}}}).
\]	
\item Moreover, if the sampling condition \ref{S1} is satisfied and if $\alpha=\frac{\eta n^{1/3}}{T^{1/3}}$ with $0<\eta\leq\min\{(2\sL \CL )^{-\frac13}, \bar\alpha n^{-\frac13}  T^{\frac13} \}$, it holds that
\[		\min_{k=1,\ldots,T} \;\Exp[\dist(0,\partial \psi(w^k))^2] \leq \frac{16[\mer(z^1) - \psilb]}{\tau \eta n^{1/3}T^{2/3}} 
= \cO({{n^{-1/3}T^{-2/3}}}).
\]	
\end{enumerate}	
\end{corollary}

\begin{proof}
	We substitute $\eta_k =\eta/T^{1/3}$ for all $k\leq T$ and $\eta_k= 0$ for all $k>T$ in \cref{thm:compexity}~(a). Observing $\eta^3 \leq \frac{1}{2\sL \CL }$, this yields
 \[
		\min_{k=1,\ldots,T} \dist(0,\partial \psi(w^k))^2 \leq \frac{4 + 24\eta^3 \sL \CL }{\tau \eta T^{2/3}} \cdot [\mer(z^1) -\psilb] \leq \frac{16[\mer(z^1) -\psilb]}{\tau \eta T^{2/3}}.
\]	
Similarly, setting $\eta_k =\eta n^{1/3}/T^{1/3}$ when $k\leq T$ and $\eta_k = 0$ when $k>T$ in \cref{thm:compexity}~(b) allows completing the proof of part (b).
\end{proof}
\begin{remark}
In the smooth setting, it holds that $\dist(0,\partial \psi(w^k)) = \|\nabla f(w^k)\|$ and the results shown in \cref{thm:compexity} and \cref{coro:compexity} match the existing (nonconvex) complexity bounds presented in \cite[Theorem 4]{mishchenko2020random} and \cite[Theorem 3 \& Corollary 2]{nguyen2021unified}. 
    The complexity result in \cref{coro:compexity}~{(b)} improves the best bound for $\EPRR$ in \cite[Theorem 3]{mishchenko2022proximal}, in terms of gradient evaluations, by a factor of $n^{1/3}$. Moreover, our results do not require a link between the natural residual $\mathcal G_\lambda$ and the gradient mapping $\nabla f$ as assumed in \cite[Assumption 2 \& Theorem 3]{mishchenko2022proximal}. 
\end{remark}

\begin{remark} \label[remark]{remark:table} Let $\varepsilon > 0$ be given and let us set $\rho = 0$ (i.e., $\vp$ is convex) and $\lambda = \frac{1}{\sL}$. We then obtain $\CL = 4[3\sL+2\lambda^{-1}]^2 = 100\sL^2$, $\tau = \frac14$, and $\bar\alpha = \frac{1}{1600\sL}$. Let us consider the constant step size $\alpha_k \equiv \alpha = \frac{1}{\sL}\min\{\frac{1}{1600 n},\frac{1}{(200T)^{1/3}n^{2/3}}\}$. Thus, following our previous derivations, \cref{thm:compexity} (b) is applicable and we can infer
\begin{align*} 
\min_{k=1,\ldots,T} \;\Exp[\dist(0,\partial \psi(w^k))^2] & \leq \left[\frac{1}{Tn\alpha}+600\sL^3n\alpha^2\right] \cdot 16(\mer(z^1)-\psilb) \\ & \hspace{-16ex} \leq \left[\sL\max\Big\{\frac{1600}{T},\frac{{200}^{1/3}}{T^{2/3}n^{1/3}}\Big\}+\frac{600\sL}{(200T)^{2/3}n^{1/3}}\right] \cdot 16(\mer(z^1)-\psilb) = \cO(\varepsilon^2), 
\end{align*}
provided that the total number of gradient evaluations satisfies $Tn \geq \frac{10\sL\sqrt{n}}{\varepsilon^2}\max\{160\sqrt{n},\frac{\sqrt{2\sL}}{\varepsilon}\}$. This recovers the result shown in \cref{table:super-nice} and confirms that $\RR$ and $\NRR$ have similar complexities. Hence, the comparisons between the complexities of $\RR$ and $\SGD$ in \citep{mishchenko2020random,nguyen2021unified} can also be transferred to $\NRR$ and $\PSGD$.
\end{remark}

Next, we discuss the corresponding complexity results for polynomial step sizes $\alpha_k \sim k^{-\gamma}$, $\gamma\in(0,1)$. Step sizes of this form are common and popular in stochastic optimization, see, e.g., \citep{robbins1951stochastic,chung1954stochastic,bottou2018optimization,nguyen2021unified}. 

\begin{corollary}\label[corollary]{coro:compexity-2}
	Assume \ref{A1}--\ref{A3} and let the sequences $\{w^k\}_k$ and $\{z^k\}_k$ be generated by $\NRR$ with $\lambda\in(0,\frac{1}{4\rho})$. 
\begin{enumerate}[label=\textup{(\alph*)},topsep=0pt,itemsep=0ex,partopsep=0ex]
	\item If $\alpha_k = \frac{\alpha}{nk^\gamma}$ with $\gamma\in(\frac13,1)$ and
	$0<\alpha \leq \min\{\bar \alpha, (\frac{3\gamma-1}{6\sL \CL})^{\frac13}\}$,
	 then  
\[\min_{k=1,\dots,T} \;\dist(0,\partial \psi(w^k))^2 \leq \frac{16[\mer(z^1)-\psilb]}{\tau\alpha(T^{1-\gamma} - 1)}=\cO({T^{-(1-\gamma)}}).\]
\item Under the sampling condition \ref{S1} and if $\alpha_k = \frac{\alpha}{n^{2/3}k^\gamma}$ with $\gamma\in(\frac13,1)$ and
	$0<\alpha\leq \min\{\bar \alpha n^{-\frac13},(\frac{3\gamma-1}{6\sL \CL})^{\frac13}\}$, then it holds that
\[		\min_{k=1,\dots,T}\;\Exp[\dist(0,\partial \psi(w^k))^2] \leq \frac{16[\mer(z^1) - \psilb]}{n^{1/3} \tau \alpha(T^{1-\gamma}-1)}=\cO({n^{-1/3} T^{-(1-\gamma)}}).
\]	
\end{enumerate}	
\end{corollary}

As the proof is a routine application of the integral test and \cref{thm:compexity}, we will omit an explicit derivation here.

The complexity bounds presented in \cref{thm:compexity} and \cref{coro:compexity,coro:compexity-2} do not directly imply convergence of the stationarity measure $\dist(0,\partial \psi(w^k)) \to 0$, i.e., accumulation points of $\{w^k\}_k$ are not automatically stationary points of the problem \cref{SO}. This subtle technicality is caused by the presence of the $\min$-operation, i.e., the results in \cref{thm:compexity} and \cref{coro:compexity,coro:compexity-2} do not apply to the last iterate $w^T$. In the following, we close this gap and establish global convergence of $\NRR$ under suitable diminishing step size schemes.

\begin{theorem}[Global Convergence of $\NRR$]
	\label{thm:global_convergence}
	Let \ref{A1}--\ref{A3} hold and let $\{w^k\}_k$, $\{z^k\}_k$ be generated by $\NRR$ with  $\lambda\in(0,\frac{1}{4\rho})$ and step sizes $\{\alpha_k\}_k$ satisfying 
	\begin{equation}
		\label{eq:ass-step}
		0<\alpha_k\leq \frac{\bar\alpha}{n},\quad{\sum}_{k=1}^{\infty}\alpha_k=\infty,\quad\text{and}\quad {\sum}_{k=1}^{\infty}\alpha_k^3<\infty.
	\end{equation} 
	Then, $\Fnor(z^k)\to0$, $\dist(0,\partial \psi(w^k))\to0$, and $\psi(w^k)\to \bar \psi\in\R$ as $k$ tends to infinity.
\end{theorem}

\cref{thm:global_convergence} ensures that accumulation points of the iterates $\{w^k\}_k$ are stationary points of the objective function $\psi$---independent of the realization of the permutations $\{\pi^k\}_k$. 

\begin{mdframed}[style=proofbox]
\textbf{Proof sketch.} By the approximate descent property in \cref{lem:merit-descent-1}, we first show that the merit function values $\{\mer(z^k)\}_k$ converge and we have ${\sum}_{k=1}^\infty \,\alpha_k \norm{\Fnor(z^k)}^2<\infty$. In the next step, invoking \ref{A1} and ${\sum}_{k=1}^\infty \, \alpha_k = \infty$, we verify $\lim_{k\to\infty}\|\Fnor(z^k)\|\to 0$. The convergence of $\{\psi(w^k)\}_k$ then follows from the convergence of the sequences $\{\mer(z^k)\}_k$ and $\{\|\Fnor(z^k)\|\}_k$. A full proof of \cref{thm:global_convergence} can be found in \cref{proof:global_convergence}. 
\end{mdframed}


\section{Convergence under the Polyak-{\L}ojasiewicz (PL) Condition} \label{sec:pl}

We now study convergence of $\NRR$ under the well-known Polyak-{\L}ojasiewicz (PL) condition \citep{polyak1963,lojasiewicz1959,lojasiewicz1963,karimi2016}. 
\begin{definition}[PL Condition]
	\label{def:PL-condition}
	The function $\psi$ is said to satisfy the PL condition if there is $\mu>0$
	such that the following inequality holds:
	\begin{equation} \label{eq:pl-cond} \dist(0,\partial\psi(w))^2 \geq 2\mu[\psi(w)-\psi^*]\quad \text{where} \quad \psi^*:=\inf_{w\in\Rn} \psi(w).\end{equation} 
\end{definition}

The PL condition is a common tool in stochastic optimization; it provides a measure of the ``strong convexity-like'' behavior of the objective function without requiring convexity of $\psi$. Furthermore, under the PL condition, every stationary point $w^* \in \crit{\psi}$ is necessarily an optimal solution to the problem \cref{SO}. To proceed, we introduce the set of optimal solutions and the variance at optimal points:
\be \label{eq:opt-var} {\cal W}:=\{w\in\Rn:\psi(w) = \psi^*\}\quad \text{and}\quad \sigma_*^2:=\sup_{w\in{\cal W}}\,\Big[\frac1n{\sum}_{j=1}^{n}\norm{\nabla f(w,j) - \nabla f(w) }^2\Big].\ee

In the following, under the PL condition and using constant step sizes $\alpha_k \sim \alpha$, we show that $\NRR$ converges linearly to an $\cO(\alpha^2 \sigma_*^2)$-neighborhood of ${\cal W}$.

\begin{theorem}\label{thm:compexity-strong}
	Let \ref{A1}--\ref{A2} hold and assume that $\psi$ satisfies the PL condition with $\mathcal W \neq \emptyset$. Let $\{w^k\}_k$, $\{z^k\}_k$ be generated by $\NRR$ with $\lambda\in(0,\frac{1}{4\rho})$ and 
			\[ \alpha_k\equiv \frac{\alpha}{n}, \quad \alpha \leq \min\left\{ \bar \alpha, \; \nu^{-1}, \; \frac{\mu\sqrt{\tau}}{2\sL \sqrt{6\CL}} \right\}, \quad \nu := \frac{\mu\tau}{3(1+\mu\tau\lambda)}. \]
   \begin{enumerate}[label=\textup{(\alph*)},topsep=0pt,itemsep=0ex,partopsep=0ex]
   \item For all $T\geq 1$, it follows $\psi(w^{T+1})-\psi^* \leq \exp(- T\nu\alpha)\cdot\frac{\mer(z^1) - \psi^*}{1+\mu\tau\lambda} + \frac{6\CL \alpha^2 \sigma_*^2}{\mu\tau}$.
   %
 %
  \item In addition, under the sampling condition \ref{S1}, we have
   \[\Exp[\psi(w^{T+1})-\psi^*] \leq \exp(-T \nu\alpha) \cdot \frac{\mer(z^1) - \psi^*}{1+\mu\tau\lambda} + \frac{6\CL\alpha^2 \sigma_*^2}{n\mu\tau}, \quad \forall~T\geq1.\]
   \end{enumerate}
	Here, $\psi^* := \inf_{w}\psi(w)$, $\CL>0$ and $\bar\alpha$ are defined in \cref{lem:est-err,lem:merit-descent-1}, and $\tau$ is given in \cref{eq:tau}.
\end{theorem}

\begin{remark} \label[remark]{remark:pl-1} It is possible to express the convergence results shown in \cref{thm:compexity-strong} in terms of the distance, $\dist(w,\mathcal W)$, to the set $\mathcal W$. Indeed, by \cite[Theorem 5 and 27]{bolte2017error} or \cite[Theorem 2]{karimi2016}, the PL condition \cref{eq:pl-cond} implies the following error bound or quadratic growth condition:
\be \dist(w,\mathcal W)^2 \leq 2\mu^{-1}[\psi(w)-\psi^*] \quad \forall~w \in \Rn. \label{eq:error-bound} \ee
If ${\cal W}=\{w^*\}$ is a singleton (which is the case, e.g., if $\psi$ is strongly convex), then \cref{thm:compexity-strong} ensures linear convergence of the iterates $\{w^k\}_k$ to an $\mathcal O(\alpha^2\sigma_*^2)$-neighborhood of the optimal solution $w^*$. Hence, in the interpolation setting $\nabla f(w^*,1) = \dots = \nabla f(w^*,n)$, the iterates $\{w^k\}_k$ generated by $\NRR$ converge linearly to $w^*$.  
\end{remark}

\begin{remark}[Comparison with $\EPRR$] \label[remark]{remark:pl-2}
    As shown in \cite[Theorem 2]{mishchenko2022proximal}, 
    applying $\EPRR$ with constant step size in the strongly convex case yields linear convergence to an $\cO(\alpha^2 \sigma_{\text{rad}}^2 )$-neighborhood of the optimal solution $w^*$. Here, the shuffling radius $\sigma_{\text{rad}}^2$ is bounded by $\cO(\sigma_*^2) + \cO(\|\nabla f(w^*)\|^2)$. Hence, even if $f(\cdot,1) = \cdots = f(\cdot,n)$ (indicating $\sigma_*^2=0$), linear convergence of $\EPRR$ can not be guaranteed since $\nabla f(w^*)$ generally does not vanish in composite problems \cref{SO}. We illustrate this effect numerically in \cref{sec:toy-interpolation}.
\end{remark}

\noindent\textbf{Proof of \cref{thm:compexity-strong}\ }
Recalling $\Fnor(z^k)\in \partial \psi(w^k)$, the PL condition implies $\|\Fnor(z^k)\|^2 \geq 2\mu(\psi(w^k)-\psi^*)$. In addition, using the definition of the merit function, $\mer(z)=\psi(\proxi{\lambda}(z)) + \frac{\tau\lambda}{2}\|\Fnor(z)\|^2$, we can infer
\begin{equation}\label{eq:com-s-cvx-1} \psi(w^k)-\psi^* \leq \frac{1}{1+\tau\lambda\mu}[\mer(z^k)-\psi^*] \quad \text{and} \quad \norm{\Fnor(z^k)}^2 \geq \frac{2\mu}{1+\mu\tau\lambda}[\mer(z^k)-\psi^*]. \end{equation}
%
%
Thus, applying \cref{lem:merit-descent-1}~(a) with \cref{eq:com-s-cvx-1} and subtracting $\psi^*$, it follows 
\begin{equation}
	\label{eq:com-s-cvx-2}
	\begin{aligned}
		\mer(z^{k+1}) - \psi^* &\leq   \left[1 - \frac{\mu \tau n \alpha_k}{2(1+\mu\tau\lambda)}\right] [\mer(z^k) - \psi^*] + \CL n^3\alpha_k^3 \sigma_k^2.
	\end{aligned}
\end{equation}
We now bound the variance $\sigma_k^2= \frac{1}{n}{\sum}_{i=1}^{n}\norm{\nabla f(w^k,i) - \nabla f(w^k) }^2$ in terms of $\sigma_*^2$, cf. \cref{eq:opt-var}. Setting $w^*_k:=\mathrm{proj}_{\cal W}(w^k)$ and using Young's inequality, \ref{A1}, \cref{eq:error-bound}, and \cref{eq:com-s-cvx-1}, we have 
\begingroup
\allowdisplaybreaks
\begin{align*}
\sigma_k^2 &= \frac{1}{n}{\sum}_{i=1}^{n}\norm{\nabla f(w^k,i)-\nabla f(w^*_k,i) + \nabla f(w^*_k,i)-\nabla f(w^*_k) + \nabla f(w^*_k)  - \nabla f(w^k) }^2 \\ &= \frac{1}{n}{\sum}_{i=1}^{n} \Big[\norm{\nabla f(w^k,i)-\nabla f(w^*_k,i)}^2 + 2\iprod{\nabla f(w^k,i)-\nabla f(w^*_k,i)}{\nabla f(w^*_k,i)-\nabla f(w^*_k)} \\ & \hspace{8ex} + 2\iprod{\nabla f(w^k,i)-\nabla f(w^*_k,i)}{\nabla f(w^*_k)  - \nabla f(w^k)} + \|\nabla f(w^*_k,i)-\nabla f(w^*_k)\|^2 \\ & \hspace{8ex} + 2\iprod{\nabla f(w^*_k,i)-\nabla f(w^*_k)}{\nabla f(w^*_k)  - \nabla f(w^k)} + \|\nabla f(w^*_k)  - \nabla f(w^k)\|^2 \Big] \\ & \leq 2\sL^2 \|w^k-w^*_k\|^2 + 2\sigma_*^2 \leq 4\sL^2\mu^{-1}[\psi(w^k)-\psi^*] + 2\sigma_*^2 \leq \tfrac{4\sL^2\mu^{-1}}{1+\tau\lambda\mu}[\mer(z^k)-\psi^*] + 2\sigma_*^2.
\label{eq:sigma-bound*}
\end{align*}
\endgroup
Hence, we obtain
\begin{equation}
	\label{eq:com-s-cvx-3}
	\begin{aligned}
		\mer(z^{k+1}) - \psi^* &\leq   \left[1 - \tfrac{\mu \tau n \alpha_k}{2(1+\mu\tau\lambda)} + \tfrac{4\CL \sL^2 n^3 \alpha_k^3}{\mu(1+\mu\tau\lambda)}\right] [\mer(z^k) - \psi^*]  + 2\CL n^3\alpha_k^3 \sigma_*^2\\
  &\leq \left[1 - \tfrac{\mu \tau n \alpha_k}{3(1+\mu\tau\lambda)}\right] [\mer(z^k) - \psi^*]  + 2\CL n^3\alpha_k^3 \sigma_*^2,
	\end{aligned}
\end{equation}
where we applied $\alpha_k^2 \leq \frac{\mu^2\tau}{24\CL\sL^2n^2}$. Fixing $\alpha_k = \frac{\alpha}{n}$ and unfolding the recursion \cref{eq:com-s-cvx-3}, it follows	
	\begin{align}
			\nonumber \mer(z^{T+1}) - \psi^* &\leq [1 - \nu\alpha]^T (\mer(z^1) - \psi^*)   + 2\CL  \alpha^3 \sigma_*^2 \cdot {\sum}_{k=0}^{T} [1 - \nu\alpha]^k \\
			&\leq \exp(-T\nu\alpha)(\mer(z^1) - \psi^*) + 2\CL\nu^{-1} \alpha^2 \sigma_*^2,
   \label{eq:com-s-cvx-5}
	\end{align}
where we used $(1-x)^T = \exp(T\log(1-x)) \leq \exp(-Tx)$ and ${\sum}_{k=0}^{T}(1-x)^k=\frac{1-(1-x)^{T+1}}{x} \leq \frac{1}{x}$ in the last line. Thus, invoking \cref{eq:com-s-cvx-1}, we can infer
\begin{align*}
  \psi(w^{T+1})-\psi^* \leq \exp(-T\nu\alpha)\cdot \frac{\mer(z^1)-\psi^*}{1+\mu\tau\lambda} + \frac{6\CL\alpha^2\sigma_*^2}{\mu\tau}.  
\end{align*}
Part (b) can be shown in a similar way using \cref{lem:merit-descent-1} (b). \hfill \BlackBox

\section{Convergence under the Kurdyka-{\L}ojasiewicz (KL) Condition}\label{sec:asymptotic}

The PL condition is generally restrictive and the verification of \cref{eq:pl-cond} is often not possible in many practical applications. To overcome these limitations, we now study the asymptotic behavior of $\NRR$ under a weaker Kurdyka-\L ojasiewicz (KL) setting. 

\subsection{KL Inequality and Accumulation Points}
We first present the definition of the KL inequality for nonsmooth functions; see, e.g., \citep{AttBolRedSou10,AttBol09,AttBolSva13,BolSabTeb14}.

\begin{definition}[KL Property]
	\label[definition]{def:KL-property}
	The function $\psi : \Rn \to \Rex$ is said to have the KL property at a point $\bar w\in \dom{\psi}$ if there exist $c>0$, $\eta\in(0,\infty],$ and a neighborhood $U$ of $\bar w$ such that for all $w \in U \cap \{w: 0 < |\psi(w) - \psi(\bar w)| < \eta\}$, the KL inequality\footnote{The specific KL inequality, we use here, is also referred to the (local) {\L}ojasiewicz inequality, see \citep{bochnak1998real,ha2017genericity}.},
	\be \label{eq:kl-ineq} |\psi(w)-\psi(\bar w)|^\theta \leq c\cdot \dist(0,\partial \psi(w)), \ee
	holds. Here, $\theta\in[0,1)$ is the KL exponent of $\psi$ at $\bar w$.
\end{definition}

The KL inequality is satisfied for semialgebraic, subanalytic, and log-exp functions, underlining its broad generality; see \citep{kur98,BolDanLew-MS-06,BolDanLew06,BolDanLewShi07}. Particular applications include, e.g., least-squares, logistic and Poisson regression \citep{li2018calculus}, deep learning loss models \citep{davis2020stochastic,dereich2021convergence}, and principal component analysis \citep{liu2019quadratic}. Intuitively, the KL inequality implies that $\psi$ can be locally reparameterized as a sharp function near its critical points. It provides a quantitative measure of how quickly the function decreases as it approaches a stationary point which is controlled by the constants $c>0$ and $\theta\in[0,1)$. In stark contrast to the PL condition, the KL inequality in \cref{def:KL-property} is a \emph{local property}. It does not necessitate $\bar w \, (\in \crit{\psi})$ to be a global solution of \cref{SO}. 
In \cref{table:kl-exponent}, we list several popular and exemplary optimization models together with their corresponding worst-case KL exponent $\theta$.

\begin{table}[t]
\centering
{\footnotesize
\begin{tabular}{lcp{6.6cm}} 
\specialrule{1pt}{0pt}{0pt} \\[-2ex]
\multicolumn{1}{c}{\textbf{Optimization model}} & \textbf{KL exponent} & \multicolumn{1}{c}{\textbf{Reference}} \\[0.5ex]
\hline \\[-1.5ex]
$\ell_1$-regularized least-squares & $\frac12$ & \cite[Lemma 10]{bolte2017error} \\[0.5ex]
$\ell_1$-regularized logistic regression & $\frac12$ & \cite[Remark 5.1]{li2018calculus} \\[0.5ex] 
Quadratic optimization with & \multirow{2}*{$\frac12$} & \multirow{2}*{\cite[Theorem 1]{liu2019quadratic}} \\
orthogonality constraints & & \\[0.5ex]
Semidefinite programming & $\frac12$ & \cite[Theorem 4.1]{yulipon19} \\[0.5ex]
Polynomials of degree $r$ & $1-\frac{1}{r(3r-3)^{d-1}}$ &  \cite[Theorem 4.2]{DAcuntoKurdyka2005} \\[1.5ex]
\specialrule{1pt}{0pt}{0pt} 
\end{tabular}
}
\caption{Optimization models and the corresponding KL exponent.}
\label{table:kl-exponent}
\end{table}

In \cite[Lemma 5.3]{ouyang2021trust}, it was shown that the KL property can be transferred from the objective function $\psi$ to the merit function $\mer$. We restate this result in  \cref{lem:KL-mer}. 
\begin{lemma}
	\label[lemma]{lem:KL-mer}
	Suppose that $\psi:\Rn\to\Rex$ satisfies the KL property at a stationary point $\bar{w}=\proxl(\bar{z})$ with exponent $\theta$ and constant $c$. The merit function $\mer:\Rn\to\R$, $\tau > 0$, then satisfies the following KL-type property at $\bar{z}$ with the exponent $\tilde \theta := \max\{\theta, \frac{1}{2}\}$
 \[
 |\mer(z)-\mer(\bar z)|^{\tilde \theta} \leq \tilde c \cdot \norm{\Fnor(z)},\quad \forall\, z\in V\cap \{z\in\Rn:|\mer(z)-\mer(\bar z)|<\tilde\eta\},
 \]
 for $\tilde c := c+\max\{1,\frac{\tau\lambda}{2}\}$, some $\tilde\eta\in(0,\infty]$, and some neighborhood $V$ of $\bar z$.
\end{lemma}

Let $\{w^k\}_k$ and $\{z^k\}_k$ be generated by $\NRR$ and let us define the associated sets of accumulation points $\mathcal{A}_w$ and $\mathcal{A}_z$:
\begin{equation}
	\label{eq:accumulation-point-set}
	\begin{aligned}
	 \mathcal{A}_w&:=\{w\in\Rn:\exists\text{ a subsequence }\{\ell_k\}_k\subseteq\N\text{ such that }w^{\ell_k}\to w\},\quad \text{and}\\
	 \mathcal{A}_z&:=\{z\in\Rn:\exists\text{ a subsequence }\{\ell_k\}_k\subseteq\N\text{ such that }z^{\ell_k}\to z\}.
	\end{aligned}
\end{equation}
%
In \cref{lemma:limit point set}, we list basic properties of the sets $\mathcal{A}_w$ and $\mathcal{A}_z$, which are direct consequences of \cref{thm:global_convergence}. The proof is deferred to \cref{proof:limit point set}. 

\begin{lemma}[Accumulation Points] \label[lemma]{lemma:limit point set}
	Let the conditions stated in \cref{thm:global_convergence} be satisfied and let $\{w^k\}_k$ be bounded. Then, the following statements are valid:
	\begin{enumerate}[label=\textup{(\alph*)},topsep=1pt,itemsep=0ex,partopsep=0ex]
		\item  The sets $\mathcal{A}_w$ and $\mathcal{A}_z$ are nonempty and compact. 
		\item  We have $\mathcal{A}_z \subseteq \{z \in \Rn: \Fnor(z)= 0\}$ and $\mathcal{A}_w = \{w \in \Rn: \exists~z \in \mathcal A_z \; \text{such that} \; w = \proxi{\lambda}(z)\} \subseteq \crit{\psi}$. 
		\item The functions $\psi$ and $\mer$, $\tau > 0$ are finite and constant on $\mathcal{A}_w$ and $\mathcal{A}_z$, respectively. 
	\end{enumerate}
\end{lemma}

\subsection{Strong Convergence}\label{subsec:strong_conv}
Based on the KL property formulated in \cref{def:KL-property}, our aim is now to show that the whole sequence $\{w^k\}_k$ converges to a stationary point of $\psi$.

\begin{assumption}[KL Conditions]
	\label[assumption]{Assumption:2} 
	We consider the following assumptions:
	\begin{enumerate}[label=\textup{\textrm{(K.\arabic*)}},topsep=1pt,itemsep=0ex,partopsep=0ex,leftmargin=8ex]
		\item \label{B1} The sequence $\{w^k\}_k$ generated by $\NRR$ is bounded.
		\item \label{B2} The KL property holds on $\mathcal A_w$, i.e., \cref{eq:kl-ineq} holds for all $\bar w \in \mathcal A_w$.
	\end{enumerate}
\end{assumption}

Condition \ref{B1} is a typical and ubiquitous prerequisite appearing in the application of the KL inequality in establishing the convergence of optimization algorithms; see, e.g., \citep{AttBol09,AttBolRedSou10,AttBolSva13,BolSabTeb14,ochchebropoc14,BonLorPorPraReb17}. Moreover, \ref{B1} can be ensured if $\psi$ has bounded sub-level sets or $\dom{\vp}$ is compact. As mentioned, condition \ref{B2} is naturally satisfied for subanalytic or semialgebraic functions \cite[Theorem \L 1]{kur98}. 

\begin{theorem}[Strong Iterate Convergence]\label{thm:finite-sum}
Assume \ref{A1}--\ref{A3} and \ref{B1}--\ref{B2}. Let $\{w^k\}_k$, $\{z^k\}_k$ be generated by $\NRR$ with $\lambda\in(0,\frac{1}{4\rho})$ and step sizes $0 < \alpha_k \leq \frac{\bar\alpha}{n}$ satisfying
		\begin{equation}
		\label{eq:kl-step}
{\sum}_{k=1}^{\infty}\alpha_k=\infty
		\quad\text{and}\quad {\sum}_{k=1}^{\infty}\alpha_k\left[{{\sum}_{i=k}^{\infty}\alpha_i^3}\right]^{\xi}<\infty,\quad \text{for some }\xi\in(0,1).
	\end{equation} 
We then have \vspace{-1ex}
\begin{equation} \label{eq:kl-step-02}
{\sum}_{k=1}^{\infty}\alpha_k \cdot \dist(0,\partial \psi(w^k)) <\infty \quad \text{and}\quad {\sum}_{k=1}^{\infty}\norm{w^{k+1}-w^k}<\infty. \vspace{-1ex}
\end{equation}
In addition, the whole sequence $\{w^k\}_k$ converges to some stationary point $w^*\in\crit{\psi}$.
\end{theorem}

\cref{thm:finite-sum} not only ensures convergence of the iterates  $\{w^k\}_k$, but, invoking \cref{eq:kl-step-02}, we also have $\min_{i=1,\ldots,k}\dist(0,\partial \psi(w^i))^2=\cO(1/(\sum_{i=1}^k \alpha_i)^2)$. This bound is faster compared to the global complexity $\min_{i=1,\ldots,k}\dist(0,\partial \psi(w^i))^2=\cO(1/(\sum_{i=1}^k \alpha_i))$ obtained in \cref{thm:compexity}. 

\begin{mdframed}[style=proofbox]
\textbf{Proof sketch.} By \cref{thm:global_convergence}, we conclude that every accumulation point of $\{w^k\}_k$ is a stationary point of $\psi$. Based on the approximate descent property \cref{lem:merit-descent-1}~(a), the conditions in \cref{lemma:limit point set}, and the KL assumptions \ref{B1}--\ref{B2}, we can show that $\sum_{k=1}^{\infty}\alpha_k\cdot \dist(0,\partial \psi(w^k)) \leq \sum_{k=1}^{\infty}\alpha_k\|\Fnor(z^k)\|<\infty$ and ${\sum}_{k=1}^{\infty}\norm{w^{k+1}-w^k}<\infty$. The latter result indicates that $\{w^k\}_k$ is a Cauchy sequence and thus, $\{w^k\}_k$ converges to some stationary point of the objective function $\psi$. The full  proof and further details are presented in \cref{proof:finite-sum}.
\end{mdframed}

\subsection{Convergence Rates under Polynomial Step Sizes} \label{sec:rate}
In this section, we establish convergence rates for the sequences $\{w^k\}_k$, $\{\dist(0,\partial \psi(w^k))\}_k$, and $\{\psi(w^k)\}_k$ under more specific step size strategies. We consider polynomial step sizes, \citep{robbins1951stochastic,chung1954stochastic,bottou2018optimization}, of the form:
\begin{equation}
	\label{eq:step-size}
	\alpha_k=\frac{\alpha}{(\beta+k)^\gamma},\quad \text{ with } \alpha>0, \quad \beta\geq 0,\quad \gamma\in\left(\frac12,1\right].
\end{equation}
Note that the first two step size conditions in \cref{thm:finite-sum} are satisfied when $k \geq (n\alpha/\bar \alpha)^\frac{1}{\gamma}-\beta$. We now present several preparatory bounds to facilitate the derivation of the rates.
\begin{lemma}\label[lemma]{lemma:step size}
	Let $\xi \in [0,1)$ be given and let us consider step sizes $\{\alpha_k\}_k$ of the form \cref{eq:step-size}. 
\begin{enumerate}[label=\textup{(\alph*)},topsep=0pt,itemsep=0ex,partopsep=0ex]
\item For all $k \geq 1$, we have 
\begin{equation}
	\label{eq:esti-u}
	{\sum}_{j=k}^\infty\alpha_j^3 \leq  \frac{a_\gamma}{(k+\beta)^{3\gamma-1}} \quad \text{and} \quad \alpha_k \left[{\sum}_{j=k}^\infty\alpha_j^3\right]^{2\xi}\leq \frac{a_\gamma}{(k+\beta)^{(1+6\xi)\gamma-2\xi}},
\end{equation}
where $a_\gamma>0$ is a numerical constant depending on $\alpha$ and $\gamma$.
	\item Moreover, if $\xi >\frac{1-\gamma}{3\gamma-1}$, then for all $k\geq 1$, we have
	\be\label{eq:esti-eps} 
	{\sum}_{t=k}^\infty \alpha_t \left[{\sum}_{j=t}^\infty\alpha_j^3\right]^{\xi} \leq \frac{{a}_\xi}{(k+\beta)^{(1+3\xi)\gamma-(1+\xi)}},
	\ee 
	where ${a}_\xi > 0$ is a constant depending on $\xi$.
\end{enumerate}	
\end{lemma} 
\begin{proof}
Using the integral test and noting that $3\gamma>1$, we obtain 
\[
{\sum}_{j=k}^\infty\alpha_j^3 = {\sum}_{j=k}^\infty \;\frac{\alpha^3}{(j+\beta)^{3\gamma}} \leq \frac{\alpha^3}{(k+\beta)^{3\gamma}} + \alpha^3\int_{k}^\infty \frac{1}{(x+\beta)^{3\gamma}}\; \rmn{d}x \leq \frac{3\gamma \alpha^3}{3\gamma-1} \frac{1}{(k+\beta)^{3\gamma-1}}.
\]
In addition, it follows that $
\alpha_k [{\sum}_{j=k}^\infty\,\alpha_j^3]^{2\xi} \leq \alpha^{1+6\xi}(\frac{3\gamma}{3\gamma-1})^{2} \cdot \frac{1}{(k+\beta)^{2\xi(3\gamma-1)+\gamma}}$. Noting $\alpha^{1+6\xi} \leq (1+\alpha)^7$, this completes the proof of the statement (a). Part (b) can be shown by applying the results in (a) and the integral test;  we refer to \cite[Lemma 3.7]{li2023convergence}. 
\end{proof}

Based on \cref{thm:finite-sum} and \cref{lemma:step size}, we can ensure the strong-limit convergence of $\NRR$ under the polynomial step size rule \cref{eq:step-size}.  
\begin{corollary}[Strong Convergence: Polynomial Step Sizes] \label[corollary]{cor:conv} Assume that the conditions \ref{A1}--\ref{A3} and \ref{B1}--\ref{B2} hold. Let the iterates $\{w^k\}_{k}$ be generated by $\NRR$ with $\lambda\in(0,\frac{1}{4\rho})$ and polynomial step sizes $\{\alpha_k\}_k$ of the form \cref{eq:step-size}.
	%
	%
	Then, $\{w^k\}_{k}$ has finite length and converges to some stationary point $w^*\in\crit{\psi}$. 
\end{corollary}
\begin{proof}  
We need to verify that $\alpha_k = \frac{\alpha}{(k+\beta)^\gamma}$ satisfies \cref{eq:kl-step} in \cref{thm:finite-sum}. Due to $\gamma \leq 1$ and $\alpha_k \to 0$, we have $\alpha_k \leq \frac{\bar \alpha}{n}$ (for all $k$ sufficiently large) and $\sum_{k=1}^\infty \alpha_k = \infty$. Using \cref{lemma:step size} (b), it holds that  ${\sum}_{k=1}^{\infty}\alpha_k\prt{{\sum}_{i=k}^{\infty}\alpha_i^3}^{\xi}<\infty$ for all $\xi\in(\frac{1-\gamma}{3\gamma-1},1) \subseteq (0,1)$. Therefore, \cref{thm:finite-sum} is applicable and $\{w^k\}_k$ has finite length and converges to some stationary point $w^*$ of $\psi$.
\end{proof}

Next, we derive the convergence rates for $\{\psi(w^k)\}_k$, $\{\dist(0,\partial \psi(w^k))^2\}_k$ and $\{w^k\}_k$.
\begin{theorem}[Convergence Rates] \label{thm:convergence rate} Assume \ref{A1}--\ref{A3} and \ref{B1}--\ref{B2}. Let $\{w^k\}_{k}$ be generated by $\NRR$ with $\lambda\in(0,\frac{1}{4\rho})$ and step sizes satisfying \cref{eq:step-size}. Then, $\{w^k\}_{k}$ converges to some $w^*\in\crit{\psi}$ and for all $k\geq 1$ sufficiently large, it holds that
\[
	\max\{\dist(0,\partial \psi(w^k))^2,|\psi(w^k)-\psi(w^*)|\} = 
	\begin{cases} \mathcal O(k^{-(3\gamma-1)}) & \text{if} \; 0\leq \theta < \frac{\gamma}{3\gamma-1}, \\[2mm] 
		\mathcal O(k^{-\frac{1-\gamma}{2\theta-1}}) & \text{if} \; \frac{\gamma}{3\gamma-1} \leq  \theta <1,
	\end{cases} \quad \text{if} \ \gamma \in {\textstyle (\frac12,1)}
\]
and 
\[
	\|w^k - w^*\| = 
	\begin{cases} \mathcal O(k^{-(2\gamma-1)}) & \text{if} \; 0\leq \theta < \frac{\gamma}{3\gamma-1}, \\[2mm] 
		\mathcal O(k^{-\frac{(1-\theta)(1-\gamma)}{2\theta-1}}) & \text{if} \; \frac{\gamma}{3\gamma-1} \leq  \theta <1,
	\end{cases} \quad \text{if} \ \gamma \in {\textstyle (\frac12,1)}.
\]
Moreover, if $\theta \in [0, \frac12]$, $\gamma = 1$, and $\alpha > 16\tilde c^2/(\tau n)$, it follows 
\[\max\{\dist(0,\partial \psi(w^k))^2,|\psi(w^k)-\psi(w^*)|\} = \cO(k^{-2}) \quad \text{and} \quad \|w^k - w^*\| = \mathcal O(k^{-1}).\]
Here, $\theta\in[0,1)$ is the KL exponent of $\psi$ at $w^*$ and $\tilde c>0$ is the corresponding KL constant introduced in \cref{lem:KL-mer}.
\end{theorem}
\begin{remark}
\cref{thm:convergence rate} establishes the convergence rates of the function values and the stationarity measure for $\NRR$. To the best of our knowledge, this is the first time that such rates have been derived for an $\RR$ method in the nonconvex setting and under the KL inequality. Moreover, \cref{thm:finite-sum} and \cref{thm:convergence rate} show strong convergence of the iterates $\{w^k\}_k$, which is the first among proximal-type $\RR$ methods. If $\theta\in[0, \frac12]$, the rate in \cref{thm:convergence rate} aligns with the strongly convex smooth case (cf. \cite{gurbu2019}).
\end{remark}
\begin{remark}
As shown in \cref{coro:compexity-2} and using polynomial step sizes, the iteration complexity of $\NRR$ is given by $\min_{i=1,\dots,k} \, \dist(0,\partial \psi(w^i))^2 =\cO(k^{-(1-\gamma)})$. When $\gamma \in (\frac12,1)$, we clearly have $1-\gamma \leq 3\gamma-1$ and $1-\gamma \leq \frac{1-\gamma}{2\theta-1}$ (if $\theta>\frac{\gamma}{3\gamma-1}$). Consequently, the asymptotic rate derived in \cref{thm:convergence rate} is faster than the complexity bound as long as $\gamma > \frac12$. Furthermore, in stark contrast to \cref{thm:compexity}, \cref{thm:convergence rate} provides last-iterate convergence guarantees.  
\end{remark}

\subsection{\texorpdfstring{Proof of \cref{thm:convergence rate}}{Proof of Theorem 24}}

In this section, we present a detailed proof of \cref{thm:convergence rate}. Upon first reading of the manuscript, the reader may safely skip to \cref{sec:num-exp}. \vspace{1mm}
%
%
\begin{proof}
By \cref{cor:conv}, $\{w^k\}_k$ converges to some stationary point $w^*\in\crit{\psi}$. Let $\theta\in[0,1)$ and $c > 0$ denote the KL exponent and constant of $\psi$ at $w^*$. 
Applying \cref{thm:global_convergence} and the definition of $\Fnor$, we may infer $z^k = w^k - \lambda\nabla f(w^k) + \Fnor(z^k) \to w^*-\lambda\nabla f(w^*) =: z^* \in \mathcal A_z$. Moreover, we notice $\alpha_k \to 0$ and $\mer(z^k) \to \bar \psi:=\psi(w^*)$ (by \cref{lemma:limit point set}) as $k$ tends to infinity. 
%
Hence, there exists $\tilde k \geq 1$ such that $\alpha_k \leq \min\{1, \frac{\bar \alpha}{n}\}$, $|\mer(z^k)-\bar \psi| < 1$, and 
 	\begin{equation}
		\label{eq:adjusted-kl-0}
		  \tilde c\norm{\Fnor(z^k)} \geq |\mer(z^k)-\bar \psi|^{\tilde \theta}, \quad \tilde \theta = \max\{{\textstyle \frac12},\theta\},\quad \text{for all $k \geq \tilde k$},
	\end{equation}
where \cref{eq:adjusted-kl-0} follows from \cref{lem:KL-mer}. Clearly, due to  $|\mer(z^k)-\bar \psi|\leq 1$, \cref{eq:adjusted-kl-0} also holds for every exponent $\vartheta \geq \tilde \theta$. Thus, we may work with the following {\L}ojasiewicz inequality 
	\begin{equation}
		\label{eq:adjusted-kl}
		  \tilde c\norm{\Fnor(z^k)} \geq |\mer(z^k)-\bar \psi|^\vartheta,\quad \vartheta \in [\tilde\theta,1),\quad \text{for all $k \geq \tilde k$}.
	\end{equation}
In the following, we always assume $k \geq \tilde k$. Rearranging the terms in \cref{lem:merit-descent-1} (a), we have
	\begin{equation}
		\label{eq:descent-again}
		r_{k}-r_{k+1} \geq \frac{\tau n \alpha_k}{4}\norm{\Fnor(z^k)}^2\quad \text{where} \quad r_k:=\mer(z^k)+u_k-\bar \psi, \quad u_k:=\sD\,{\sum}_{i=k}^\infty \alpha_i^3,
	\end{equation}
 and $\sD:= n^3\Delta(n^3{\textstyle \sum}_{i=1}^{\infty}\alpha_i^3) < \infty$. Due to $\mer(z^k) \to \bar\psi$, $u_k\to0$, and \cref{eq:descent-again}, the sequence $\{r_k\}_k$ monotonically decreases to $0$ and it holds that $r_k \geq 0$. \\[1mm]
%
\noindent \textbf{Step 1: Rate for $\{r_k\}_k$}. We first establish a rate for $\{r_k\}_k$ through which we can easily derive the rates for $\{|\psi(w^k)-\bar \psi|\}_k$ and $\{\|\Fnor(z^k)\|^2\}_k$. 
	 Combining \cref{eq:adjusted-kl} and \cref{eq:descent-again}, it follows
  %
	 \begin{align*}
	    r_k-r_{k+1} &\geq \frac{\tau n \alpha_k}{4\tilde c^2}|\mer(z^k)-\bar \psi|^{2\vartheta}=\frac{\tau n \alpha_k}{4\tilde c^2}(|\mer(z^k)-\bar \psi|^{2\vartheta} + u_k^{2\vartheta}) - \frac{\tau n\alpha_k }{4\tilde c^2}u_k^{2\vartheta}
     \\[1mm]&\geq \frac{\tau n \alpha_k}{8\tilde c^2}|\mer(z^k)-\bar \psi + u_k|^{2\vartheta} - \frac{\tau n\alpha_k }{4\tilde c^2}u_k^{2\vartheta}= \frac{\tau n \alpha_k}{8\tilde c^2}r_k^{2\vartheta} - \frac{\tau n \alpha_k}{4\tilde c^2 }u_k^{2\vartheta},
     \end{align*}
	where the last line is due to Minkowski's inequality, i.e., $|a|^{2\vartheta} + |b|^{2\vartheta} \geq |a + b|^{2\vartheta}/2$ for all $\vartheta \in[\frac12,1)$, $a,b \in \R$. For simplicity, we limit our discussion to the case $\beta=0$. Substituting $\alpha_k = \frac{\alpha}{k^\gamma}$ and using \cref{lemma:step size} (a) and $\sD^{2\vartheta} \leq \max\{1,\sD\}^{2\vartheta}\leq \max\{1,\sD^2\}$, we obtain
	 \begin{equation}
	 	\label{eq:recursion-r}
	 	r_{k+1} \leq r_k - \frac{\tau n \alpha}{8\tilde c^2} \frac{r_k^{2\vartheta}}{k^\gamma} +  \frac{\sE}{k^{(1+6\vartheta)\gamma-2\vartheta}} \quad \text{where}\quad \sE := \frac{\tau n \alpha a_\gamma }{4\tilde c^2}\cdot \max\{1,\sD^2\}.
	 \end{equation}
	  In the following, we provide the convergence rates of $\{r_k\}_k$ based on different exponents $\vartheta$. 

Our derivations are based on classical convergence results for sequences of numbers shown in \cite[Lemma 4 and 5]{Pol87}. For ease of exposition, we state those results in \cref{lemma:rate}.

\begin{lemma} \label[lemma]{lemma:rate} Let $\{y_k\}_{k} \subseteq \R_+$ and $b \geq 0$, $d, p, q > 0$, $s \in (0,1)$, $t > s$ be given. 
	\begin{enumerate} [label=\textup{(\alph*)},topsep=0pt,itemsep=0ex,partopsep=0ex]
		\item Suppose that $\{y_k\}_{k}$ satisfies 
		\[ 
		y_{k+1} \leq \left(1- \frac{q}{k+b} \right) y_k + \frac{d}{(k+b)^{p+1}}, \quad \forall \ k\geq 1.
		\]
		If $q > p$, it holds that $ y_k \leq \frac{d}{q-p}\cdot (k+b)^{-p} + o((k+b)^{-p})$ for all sufficiently large $k$. 
		\item Let the sequence $\{y_k\}_{k}$ be given with $y_{k+1} \leq (1- \frac{q}{(k+b)^s} ) y_k + \frac{d}{(k+b)^{t}}$ for all $k \geq 1$. Then, it follows $y_k \leq \frac{d}{q} \cdot (k+b)^{s-t} + o((k+b)^{s-t})$.
	\end{enumerate}
\end{lemma}

\noindent Let us now continue with step 1 and with the proof of \cref{thm:convergence rate}. \\[1mm]
\noindent\textbf{Case 1:} $\vartheta =\frac12$. In this case, the estimate \cref{eq:recursion-r} simplifies to
\[		r_{k+1} \leq \left[1-\frac{\tau n \alpha}{8\tilde c^2} \frac{1}{k^\gamma} \right]r_{k} +  \frac{\sE}{k^{4\gamma-1}}.
\]
	If $\gamma < 1$, the rate is $r_k = \cO(k^{-(3\gamma-1)})$ by \Cref{lemma:rate} (b). Moreover, if $\gamma=1$ and $\alpha > \frac{16\tilde c^2}{\tau n}$, then \Cref{lemma:rate} (a) yields $r_k = \mathcal O(k^{-2})$. \\[1mm]
\noindent \textbf{Case 2:} $\vartheta \in (\frac12,1)$, $\gamma \neq 1$. In this case, the mapping $x \mapsto h_\vartheta(x):= x^{2\vartheta}$ is convex for all $x>0$ and we have
	\begin{equation}
		\label{eq:convexity-h}
		h_\vartheta(y) \geq  h_\vartheta(x) + h_\vartheta^\prime(x)(y-x) = 2\vartheta x^{2\vartheta-1} y +  (1 - 2\vartheta )h_\vartheta(x)  \quad \forall~x,y>0.
	\end{equation}
	Our next step is to reformulate the recursion \cref{eq:recursion-r} into a suitable form so that \cref{lemma:rate} is applicable. To that end, in \cref{eq:convexity-h}, we set $x = \bar c k^{-\sigma}$ and $y = r_k$, where $\bar{c} := ( \frac{8\tilde c^2\sigma}{\tau \alpha n\vartheta})^{{1}/({2\vartheta-1})}$ and $\sigma:=\min\{\frac{1-\gamma}{2\vartheta-1}, 3\gamma-1\}$. Then, it follows
  $r_k^{2\vartheta} \geq \frac{16\tilde c^2\sigma}{\tau \alpha n} \frac{r_k}{k^{(2\vartheta-1)\sigma}} + \frac{(1-2\vartheta)\bar{c}^{2\vartheta}}{k^{2\vartheta\sigma}}$. 
  Using this bound in \cref{eq:recursion-r}, we obtain
	 \[
	 r_{k+1} \leq \left[1- \frac{2\sigma}{k^{\gamma+(2\vartheta-1)\sigma}} \right]r_k + \frac{\sE}{k^{(1+6\vartheta)\gamma-2\vartheta}} + \frac{(2\vartheta-1)\tau n \alpha \bar{c}^{2\vartheta}}{8\tilde c^2}\frac{1}{k^{\gamma+2\vartheta\sigma}}.
	 \]
	Noticing $\gamma+2\vartheta \sigma \leq \gamma + 2\vartheta(3\gamma-1) = (1+6\vartheta)\gamma-2\vartheta$ (by definition of $\sigma$), there exists $\hat c > 0$ such that  
	\[
	 r_{k+1} \leq \left[1- \frac{2\sigma}{k^{\gamma+(2\vartheta-1)\sigma}} \right]r_k + \frac{\hat{c}}{k^{\gamma+2\vartheta\sigma}}.
	 \]
	 \cref{lemma:rate} then yields $r_k=\cO(k^{-\sigma})$. Since the parameter $\sigma$ is determined by the adjusted KL exponent $\vartheta \in [\max\{\tfrac12,\theta\} , 1)$, to maximize $\sigma$, we shall always choose $\vartheta = \theta$ when $\theta >\half$. On the other hand, if $\theta\in[0,\half]$, we set $\vartheta=\half$ and the results in \textbf{Case 1} apply.
	
	Therefore, we can express the rate of $\{r_k\}_k$ in terms of the original KL exponent $\theta$ and the step size parameter $\gamma$:
	\begin{align*}
		r_k = \cO(k^{-R(\theta,\gamma)})\quad \text{where}\quad R(\theta,\gamma) := \begin{cases}
			3\gamma-1 &\text{if}\; \theta \in [0,\frac{\gamma}{3\gamma-1}]\\[2mm]
			\frac{1-\gamma}{2\theta-1} &\text{if}\; \theta \in (\frac{\gamma}{3\gamma-1},1)
		\end{cases}\quad \text{when} \quad \gamma \in (\textstyle\frac12,1),
	\end{align*}
and $R(\theta,\gamma):=2$ when $\gamma=1, \theta \in [0,\half]$ and $\alpha > \frac{16\tilde c^2}{\tau n}$. \\[2mm]
\noindent \textbf{Step 2: Rate for $\{\norm{\Fnor(z^k)}^2\}_k$}. Based on our discussion of $\{r_k\}_k$, we now compute the rate for $\{\norm{\Fnor(z^k)}^2\}_k$ using the sufficient descent property \cref{eq:descent-again} and \Cref{lem:in-main-proof}: 

\begin{lemma}\label[lemma]{lem:in-main-proof}
Assume \ref{A1}--\ref{A3} and let $\{w^k\}_k$ and $\{z^k\}_k$ be generated by $\NRR$ with $\lambda\in(0,\frac{1}{\rho})$. Let $\varsigma > 0$ be given such that
\[ \sQ\varsigma\exp(\sQ\varsigma) \leq \tfrac{1}{2} \quad \text{where} \quad \sQ := (1-\lambda\rho)^{-1}(\sL+2\lambda^{-1}-\rho)n\max\{1,\sqrt{\CL(2\sL+1)} n\}  \]
and $\CL > 0$ is introduced in \cref{lem:est-err}. Consider the following additional assumptions:
\begin{itemize}
	\item There exists ${\sf P}>0$ such that we have $\max\{\|\Fnor(z^k)\|^2,\psi(w^k) - \psilb\} \leq {\sf P}$ for all $k \geq 1$.
	\item For $k \geq 1$, there exists $i = i(k) \geq 1$ such that $\sum_{j=0}^{i-1} \alpha_{k+j} \leq \varsigma$. 
\end{itemize}
Then, the relation 
$\norm{\Fnor(z^{k+i})}^2 \geq \frac18\norm{\Fnor(z^k)}^2 - \frac{\sP}{4\varsigma^2}({\sum}_{j=0}^{i-1}\;\alpha_{k+j}^2)^2$ holds.
\end{lemma}

A proof of the auxiliary results in \cref{lem:in-main-proof} is presented in \cref{proof:lem:in-main-proof}. The first condition, $\max\{\|\Fnor(z^k)\|^2,\psi(w^k) - \psilb\} \leq {\sP}$, in \Cref{lem:in-main-proof} is guaranteed by the convergence of $\{\psi(w^k)\}_k$ and $\{\Fnor(z^k)\}_k$ (cf. \Cref{thm:global_convergence}). The conditions $\alpha_k \to 0$ and $\sum_{k=0}^\infty \alpha_k=\infty$ imply that for every $k$ sufficiently large, there exists an integer $t = t(k) \geq 1$ such that $\frac{\varsigma}{2} \leq \sum_{j=0}^{t-1} \alpha_{k+j} \leq \varsigma$. Hence, the requirements in \cref{lem:in-main-proof} are satisfied for all $k$ sufficiently large and all $i = 1,\dots,t(k)$, i.e., we have $\norm{\Fnor(z^{k+i})}^2 \geq \frac18\norm{\Fnor(z^k)}^2 - \frac{\sP}{4\varsigma^2}({\sum}_{j=0}^{i-1}\;\alpha_{k+j}^2)^2$ for all $1 \leq i \leq t(k)$. Noting $t = t(k)$ and summing \cref{eq:descent-again} for $k,k+1,\dots,k+t$, this yields
\begingroup
\allowdisplaybreaks
\begin{align*}
	r_k \geq r_k - r_{k+t} &\geq \frac{\tau n }{4} {\sum}_{i=0}^{t-1}\; \alpha_{k+i}\norm{\Fnor(z^{k+i})}^2 \\ &\geq \frac{\tau n }{32} \big[{\sum}_{i=0}^{t-1}\; \alpha_{k+i}\big] \norm{\Fnor(z^{k})}^2 - \frac{\sP\tau n}{16\varsigma^2} {\sum}_{i=1}^{t-1}\; \alpha_{k+i}\big [{\sum}_{j=0}^{i-1}\;\alpha_{k+j}^2\big]^2.
\end{align*}
\endgroup
Since $\frac{\varsigma}{2} \leq \sum_{j=0}^{t-1} \alpha_{k+j} \leq \varsigma$ and $\{\alpha_k\}_k$ is monotonically decreasing, we further obtain
\begin{align*}
	\frac{\tau n \varsigma}{64}\norm{\Fnor(z^{k})}^2 &\leq \frac{\tau n }{32}\big[{\sum}_{i=0}^{t-1}\; \alpha_{k+i}\big] \norm{\Fnor(z^{k})}^2 \leq r_k + \frac{\sP\tau n}{16\varsigma^2} {\sum}_{i=1}^{t-1}\; \alpha_{k+i}\big[{\sum}_{j=0}^{i-1}\;\alpha_{k+j}^2\big]^2\\
	 &\leq r_k +  \frac{\sP\tau n}{16\varsigma^2} {\sum}_{i=1}^{t-1}\; \alpha_{k+i}\big[\alpha_k{\sum}_{j=0}^{t-1}\;\alpha_{k+j}\big]^2 \leq r_k + \frac{\sP\tau n \varsigma}{16} \alpha_k^2 = \cO(k^{-R(\theta,\gamma)}).
\end{align*}
The last line is due to $\alpha_k^2=\cO(k^{-2\gamma})$ and $2\gamma \geq R(\theta,\gamma)$. \\[1mm]
\noindent \textbf{Step 3: Rate for $\{\psi(w^k)\}_k$}. Recalling $r_k=\mer(z^k)+u_k-\bar \psi=\psi(w^k)-\bar \psi+\frac{\tau\lambda}{2}\norm{\Fnor(z^k)}^2+u_k$ and invoking the triangle inequality, it follows
\begin{align*}
 |\psi(w^k)-\bar \psi| = \left|r_k - \tfrac{\tau\lambda}{2}\norm{\Fnor(z^k)}^2 - u_k \right| \leq |r_k| + \tfrac{\tau\lambda}{2}\norm{\Fnor(z^k)}^2 + u_k = \cO(k^{-R(\theta,\gamma)}),
\end{align*}
where the last equality holds due to $u_k= \cO(\sum_{j=k}^\infty  \alpha_j^3)=\cO(k^{-3\gamma+1})$ (cf. \cref{lemma:step size} (a)) and $3\gamma-1 \geq  R(\theta,\gamma)$. \\[1mm]
\noindent The last step of the proof, i.e., the derivation of the rate of convergence of the iterates $\{w^k\}_k$, can be found in \cref{proof:step-4}. \end{proof}

\section{Numerical Experiments}\label{sec:num-exp}
In this section, we compare $\NRR$ with two prevalent proximal stochastic algorithms; the standard proximal stochastic gradient ($\PSGD$, \cite{duchi2009efficient}) and the epoch-wise proximal random reshuffling method ($\EPRR$, \cite{mishchenko2022proximal}). In the experiments, we typically evaluate the performance of the tested algorithms based on the following two criteria: (i) The \emph{relative error} is defined as ${(\psi(w^k)-\psi_{\min})}/{\max\{1,\psi_{\min}\}}$, where $\psi_{\min}$ is the smallest function value among all generated iterations of the algorithms; (ii) the \emph{natural residual} $\mathcal G_1(w):= w - \prox{\vp}(w - \nabla f(w))$. Similar to $\NRR$ and $\EPRR$, we will keep the step size $\alpha_k$ fixed in each epoch when applying $\PSGD$. 

\begin{figure}[t]
\centering
 	\includegraphics[width=8.0cm]{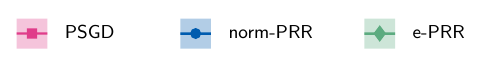} \vspace{-1ex}

 	\hspace*{-2ex}
 	\begin{tikzpicture}[scale=1]
 	\node[right] at (0.0,0) {\includegraphics[width=4.8cm,trim=1cm 0cm 0cm 0cm,clip]{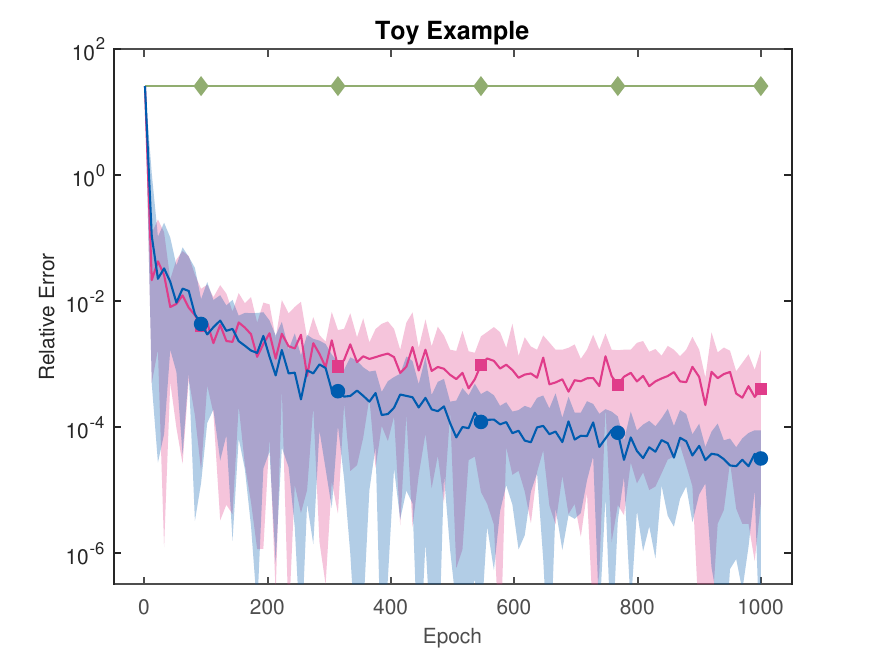}};
 	\node[right] at (4.9,0) {\includegraphics[width=4.8cm,trim=1cm 0cm 0cm 0cm,clip]{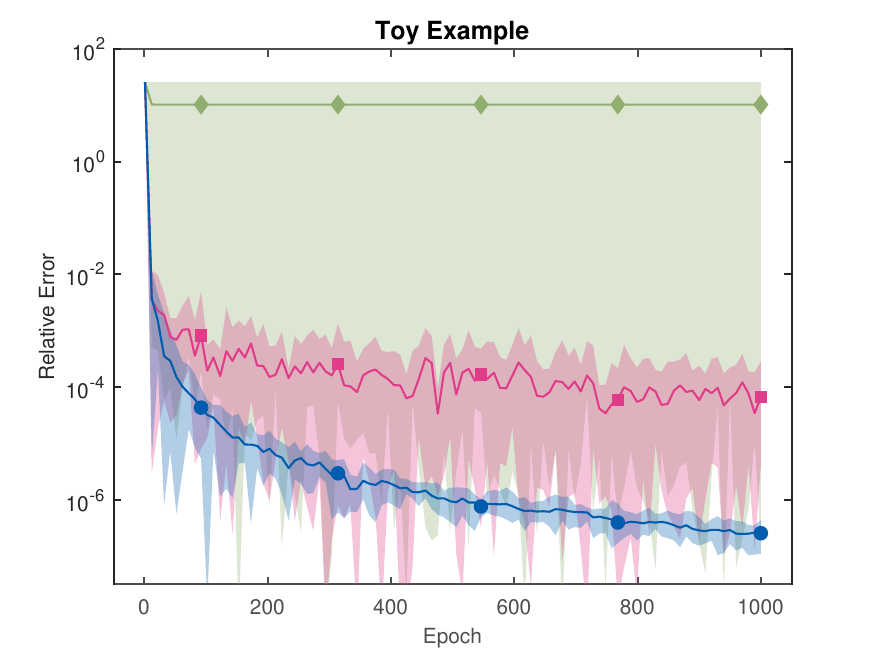}};
 	\node[right] at (9.8,0) {\includegraphics[width=4.8cm,trim=1cm 0cm 0cm 0cm,clip]{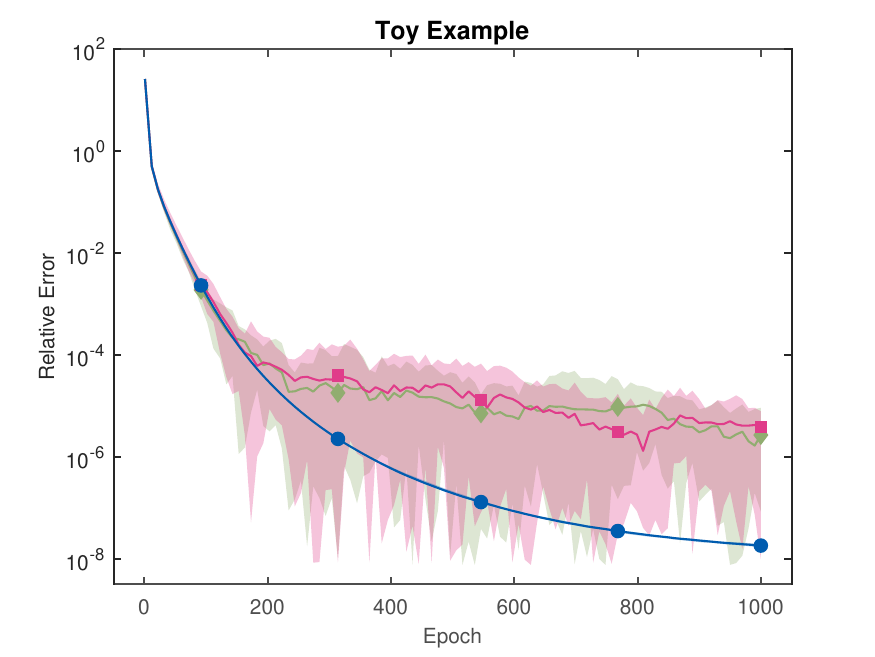}};
 	\node[right] at (0.0,-4.2) {\includegraphics[width=4.8cm,trim=1cm 0cm 0cm 0cm,clip]{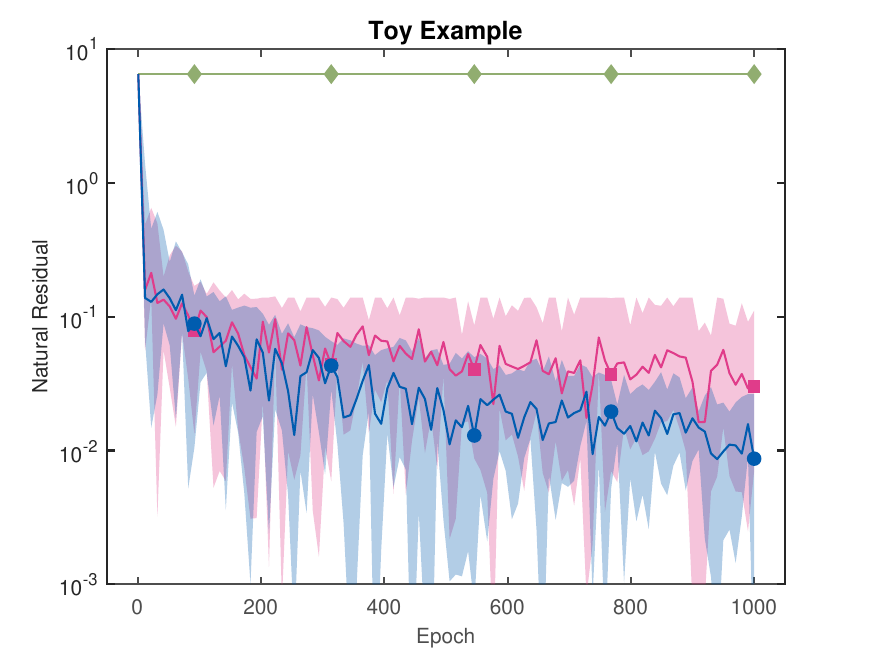}};
 	\node[right] at (4.9,-4.2) {\includegraphics[width=4.8cm,trim=1cm 0cm 0cm 0cm,clip]{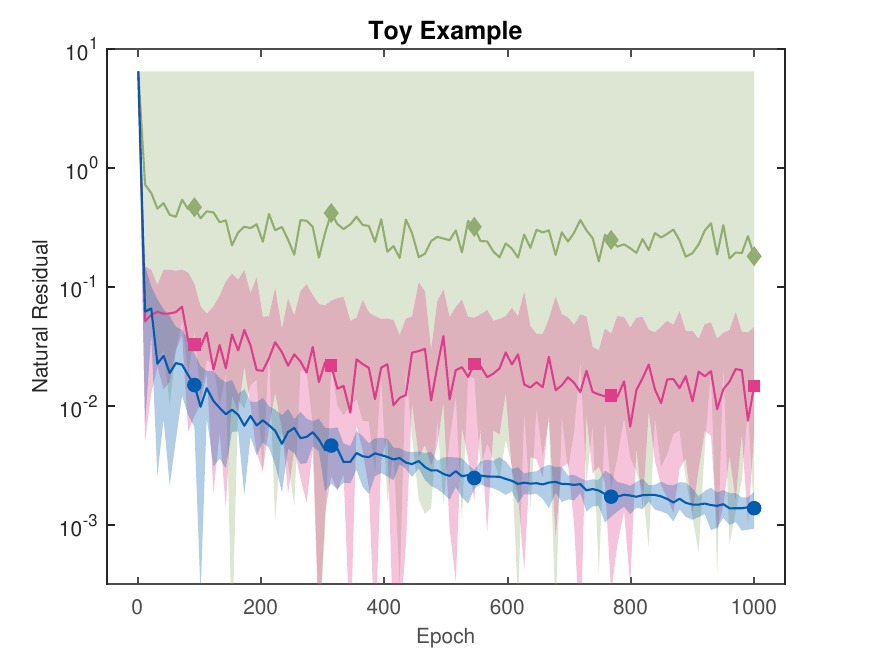}};
 	\node[right] at (9.8,-4.2) {\includegraphics[width=4.8cm,trim=1cm 0cm 0cm 0cm,clip]{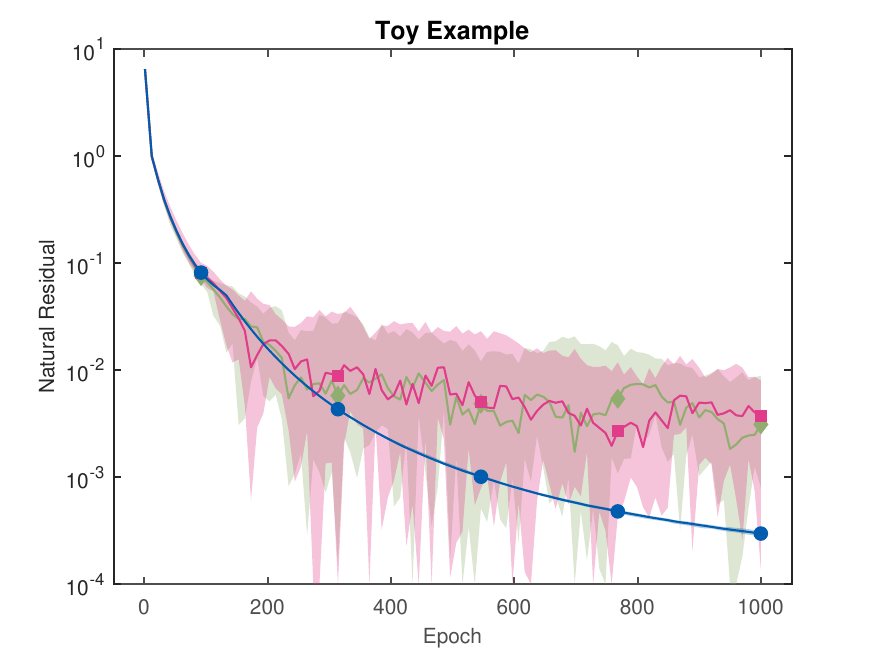}};
 	\node[right] at (1.6,-6.5) {{\footnotesize(a)~\texttt{$\alpha=1$}}};
 	\node[right] at (6.4,-6.5) {{\footnotesize(b)~\texttt{$\alpha=0.1$}}};
 	\node[right] at (11.2,-6.5) {{\footnotesize(c)~\texttt{$\alpha=0.01$}}};
 	\node at (14.8,0) {\rotatebox{-90}{{\footnotesize Relative error}}};
 	\node at (14.8,-4) {\rotatebox{-90}{{\footnotesize Natural residual}}};
 \end{tikzpicture}
 	\caption{Performance of $\PSGD$, $\EPRR$, and $\NRR$ on the toy example \cref{eq:exp-1-1}. We report 10 independent runs; the average performance is shown using a thicker line style.}
 	\label{fig:1}
 \end{figure}
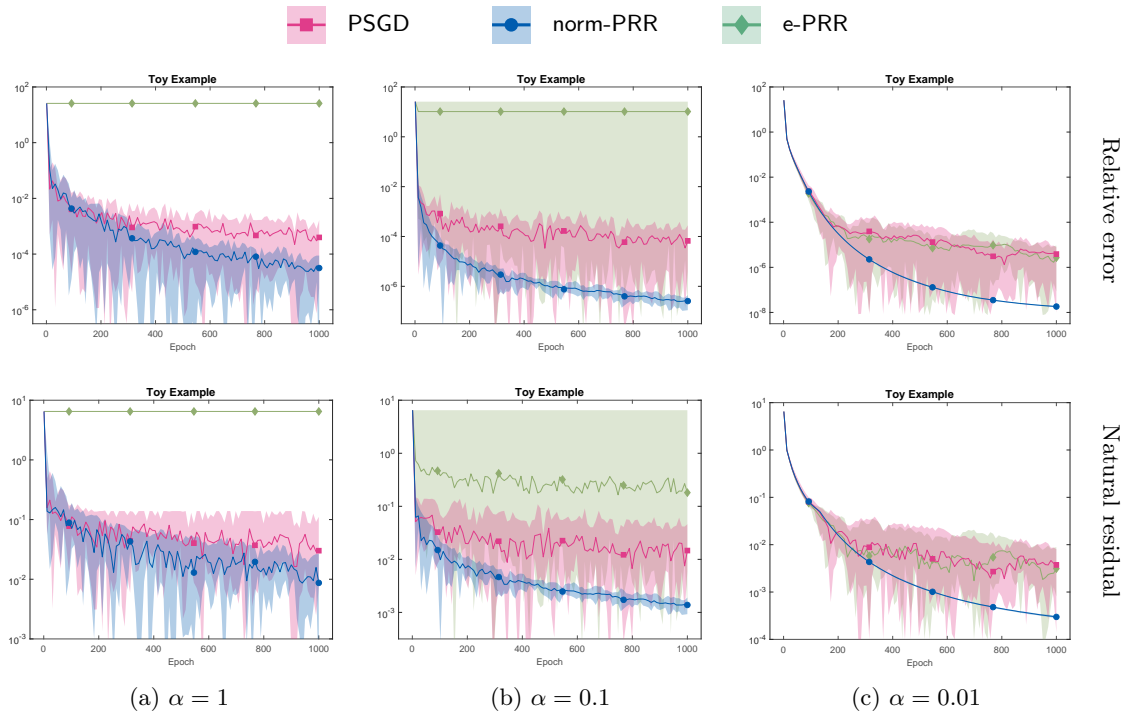
 
\subsection{A Toy Example: Testing feasibility}\label{subsec:exp-1}
We consider a one-dimensional toy example with smooth component functions $\R \ni w \mapsto f(w,i)$ that are not well-defined for $w \notin \dom{\vp}$:
\begin{equation}
 	\label{eq:exp-1-1}
 	\begin{aligned}
 		\min_{w\in\R}~\frac{1}{n}{\sum}_{i=1}^{n}f(w,i) + \vp(w) :=  \frac{1}{100}{\sum}_{i=1}^{100}\frac{\sin(\frac{i\pi}{100}) w^2+\log^2(w+\frac{i}{10})}{2} + \iota_{\R_+}(w),
 	\end{aligned}
 \end{equation}
where $\iota_{\R_+}$ denotes the indicator function of $\R_+$. In this case, the proximity operator reduces to the projection onto the half space $\R_+$. Notice that the function $f$ is not well-defined for $w\leq -\frac{1}{10}$. Hence, once we detect iterates $w^k$ with $w^k \leq -\frac{1}{10}$, we stop the tested algorithm and mark such run as ``failed''. (Due to the projection steps in $\PSGD$ and $\NRR$, failed runs can only occur when testing $\EPRR$). \\[1mm] 
\noindent \textbf{Implementation details.} We set $w^0=10$, $\lambda=1$, and use diminishing step sizes of the form $\alpha_k=\alpha/k$. Here, $\alpha$ takes values from the set $\{1,0.1,0.01\}$. We conduct 10 independent run of each algorithm and depict their overall performance in \Cref{fig:1}.

In \Cref{fig:1} (a), when $\alpha=1$, the success rate of $\EPRR$ is $0\%$, i.e., every run of $\EPRR$ returns invalid values. Both $\PSGD$ and $\NRR$ show larger fluctuations. In \Cref{fig:1} (b), when $\alpha=0.1$, the success rate of $\EPRR$ increases to $40\%$. $\NRR$ converges faster with less oscillations compared to the other two algorithms. In \Cref{fig:1} (c), when $\alpha=0.01$, the success rate of $\EPRR$ is $100\%$; its performance is similar to $\PSGD$. As expected, both $\PSGD$ and $\NRR$ are not affected by the potential infeasibility $\dom{\vp} \subseteq \dom{f} \neq \R$. 

\subsection{Linear Convergence and Interpolation} \label{sec:toy-interpolation}

\begin{figure}[t]
\centering
	\setlength{\abovecaptionskip}{-3pt plus 3pt minus 0pt}
	\setlength{\belowcaptionskip}{-10pt plus 3pt minus 0pt}
	\centering
	\includegraphics[width=8.0cm]{figs//legend.pdf} \vspace{-1ex}

 	\hspace*{-2ex}
	\begin{tikzpicture}[scale=1]
    \node[right] at (0.0,0) {\includegraphics[width=6.8cm,trim=0.45cm 0cm 0cm 0cm,clip]{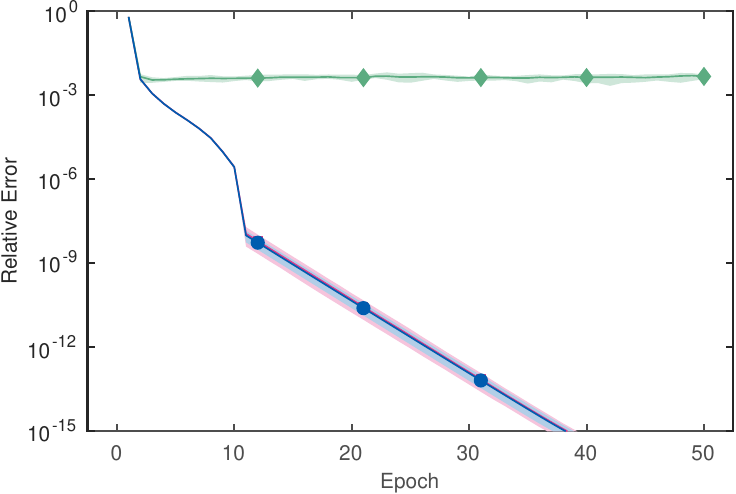}};
    \node[right] at (7.3,0) {\includegraphics[width=6.8cm,trim=0.45cm 0cm 0cm 0cm,clip]{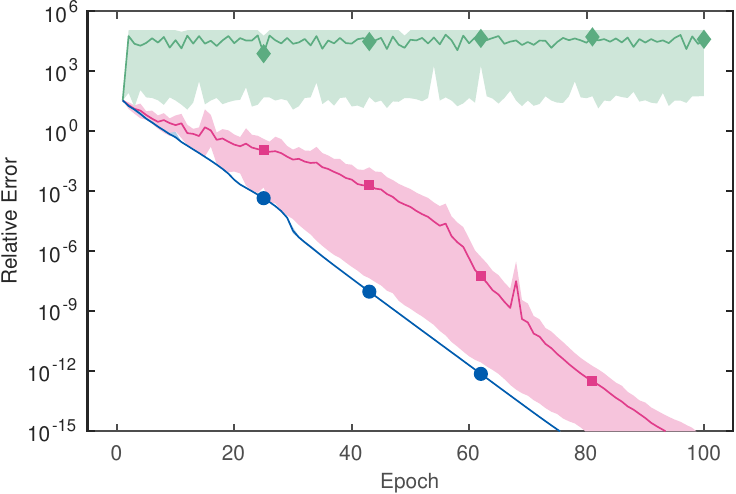}};
     \node at (14.8,0) {\rotatebox{-90}{{\footnotesize Relative error}}};
\end{tikzpicture}
	\caption{This plot illustrates that $\NRR$ (and $\PSGD$) can achieve linear convergence when $\sigma_*^2=0$, whereas $\EPRR$ converges to a neighborhood of $w^*$. Left: $A$ in \cref{eq:prob-toy} is generated using a uniform distribution and all methods use $\alpha_k = \frac{4}{\sL n}$. Right: $A$ follows a Student's $t$ distribution and we apply the step size $\alpha_k = \frac{0.04}{\sL n}$. We depict the relative error $(\psi(w^k)-\psi(w^*))/\max\{1,\psi(w^*)\} = \psi(w^k)$ of 10 independent runs. The average performance is shown using a thicker line style. }
 \label{fig:linear}
\end{figure}

We now numerically illustrate the convergence results obtained in \cref{sec:pl} and study convergence of $\NRR$ in the interpolation setting as discussed in \cref{remark:pl-1,remark:pl-2}. We consider the problem
\begin{equation} \label{eq:prob-toy}
	\min_{w\in \Rn}~\frac{1}{n}{\sum}_{i=1}^{n}f(w,i)+\vp(w):=\frac{1}{n}{\sum}_{i=1}^{n} \left[0.5 (a_i^\top w - b_i)^2 + c^\top w\right] + \iota_{\Delta^d}(w),
\end{equation}
where $\Delta^d := \{w: w_i \in [0,1] \, \text{and} \, \mathds 1^\top w = 1\}$ is the $d$-simplex. We choose $n = 5000$, $d = 250$, and select $A = [a_1,\dots,a_n]^\top \in \R^{n \times d}$ randomly following a uniform distribution, $a_{ij} \sim \mathcal U[0,1]$, or a Student's $t$ distribution with degree of freedom $1.5$. We generate $w^* \in \Delta^d$, $b$, and $c$, via
\[ w^*_i = 0.2, \; \forall~i \in \mathcal I, \quad w^*_i = 0, \; \forall~i \notin \mathcal I, \quad b = Aw^*, \quad c_i = 0, \; \forall~i \in \mathcal I, \quad c_i \sim \mathcal U[0,1], \; \forall~i \notin \mathcal I,  \]
where $\mathcal I \subseteq [d]$ is a random index set with $|\mathcal I| = 5$. We have $\nabla f(w^*,1) = \dots = \nabla f(w^*,n) = \nabla f(w^*) = c$ and by construction, it holds that $-\iprod{c}{y-w^*} = -\sum_{i\notin\mathcal I} c_iy_i \leq 0$ for all $y \in \Delta^d$, i.e., $-\nabla f(w^*) \in N_{\Delta^d}(w^*)$. Hence, $w^*$ is a solution to \cref{eq:prob-toy}. In the tests, we generate $A$ such that $A^\top A$ is invertible, i.e., problem \cref{eq:prob-toy} is strongly convex with $\sigma_*^2 = 0$. We run $\PSGD$, $\NRR$, and $\EPRR$ with $w^0 = e_n$, $\lambda = \frac{1}{\sL}$, and constant step sizes $\alpha_k = \frac{4}{\sL n}$ and $\alpha_k = \frac{0.04}{\sL n}$ where  $\sL = \frac1n \|A\|_2$. The results are presented in \Cref{fig:linear}. 

\subsection{Nonconvex Binary Classification}\label{subsec:exp-2}

\begin{figure}[t]
\centering
	\setlength{\abovecaptionskip}{-3pt plus 3pt minus 0pt}
	\setlength{\belowcaptionskip}{-10pt plus 3pt minus 0pt}
	\centering
	\includegraphics[width=8.0cm]{figs//legend.pdf} \vspace{-1ex}
	
	\hspace*{-2ex}
	\begin{tikzpicture}[scale=1]
	\node[right] at (0.0,0) {\includegraphics[width=3cm,trim=.5cm 0cm 0cm 0cm,clip]{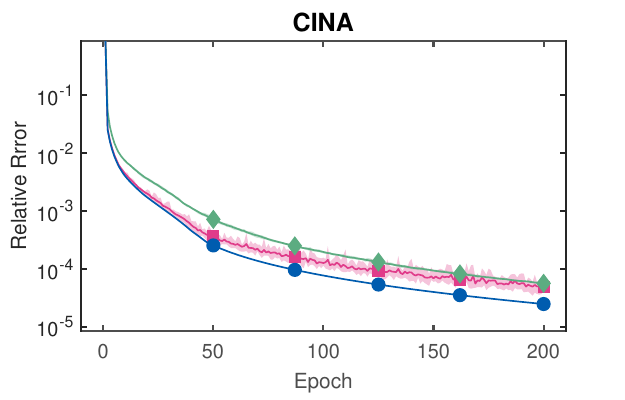}};
	\node[right] at (3,0) {\includegraphics[width=3cm,trim=.5cm 0cm 0cm 0cm,clip]{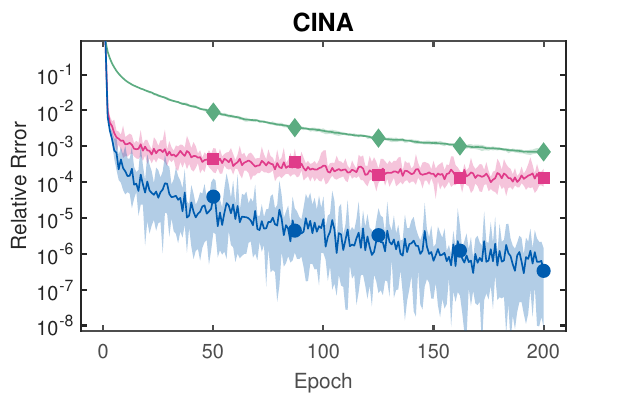}};
	\node[right] at (6,0) {\includegraphics[width=3cm,trim=.5cm 0cm 0cm 0cm,clip]{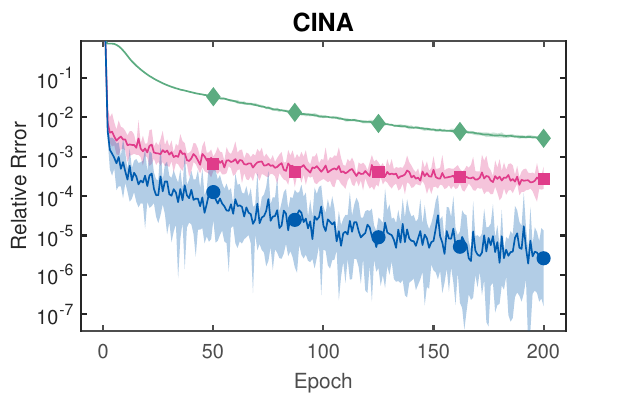}};
 \node[right] at (9,0) {\includegraphics[width=3cm,trim=.5cm 0cm 0cm 0cm,clip]{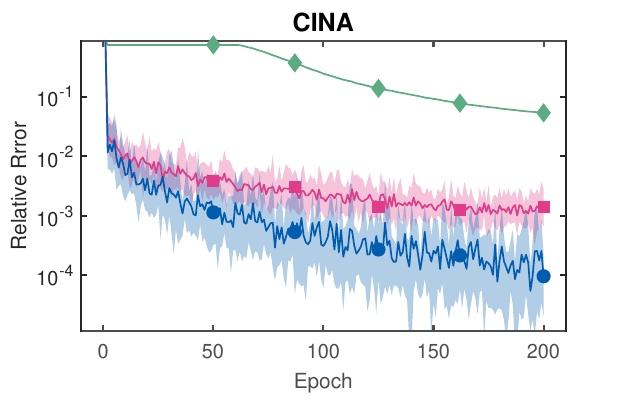}};
 \node[right] at (12,0) {\includegraphics[width=3cm,trim=.5cm 0cm 0cm 0cm,clip]{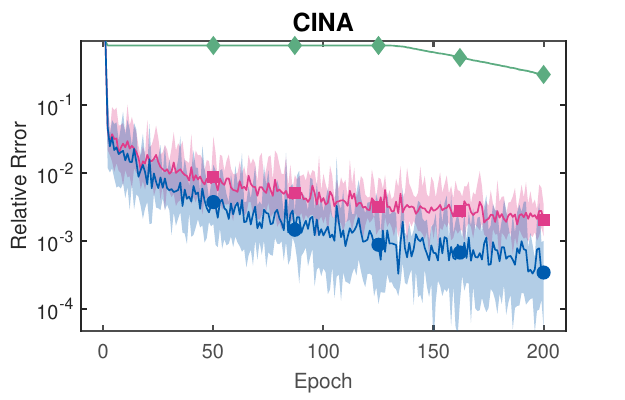}};
 \node[right] at (0.0,-2.2) {\includegraphics[width=3cm,trim=.5cm 0cm 0cm 0cm,clip]{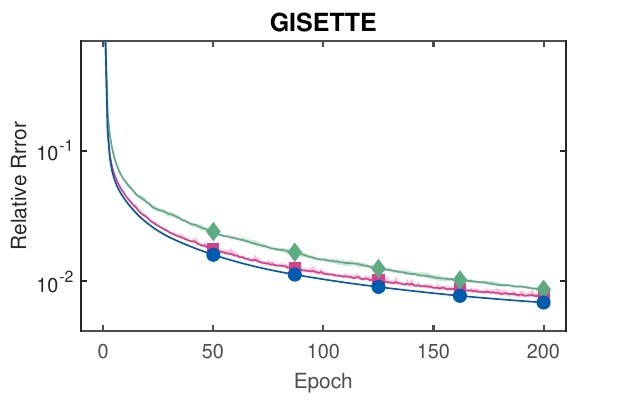}};
	\node[right] at (3,-2.2) {\includegraphics[width=3cm,trim=.5cm 0cm 0cm 0cm,clip]{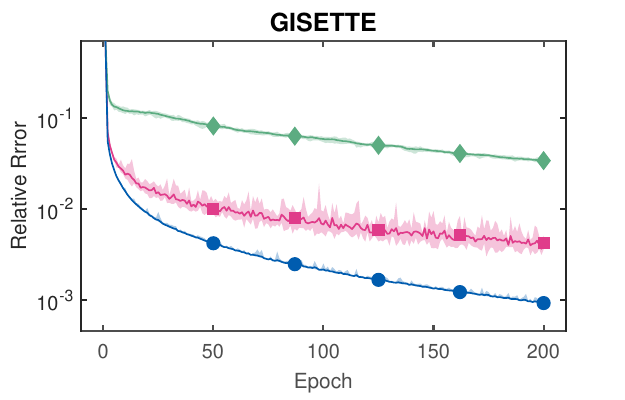}};
	\node[right] at (6,-2.2) {\includegraphics[width=3cm,trim=.5cm 0cm 0cm 0cm,clip]{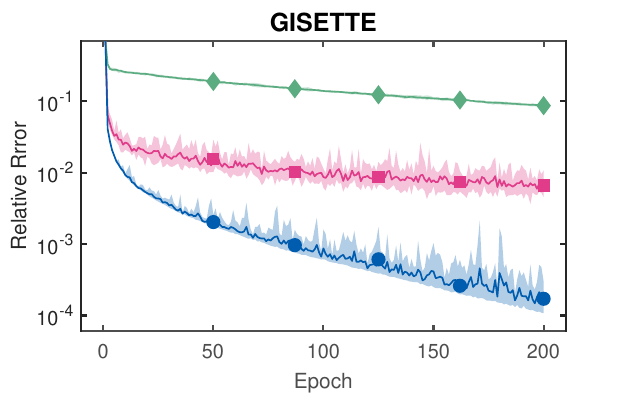}};
 \node[right] at (9,-2.2) {\includegraphics[width=3cm,trim=.5cm 0cm 0cm 0cm,clip]{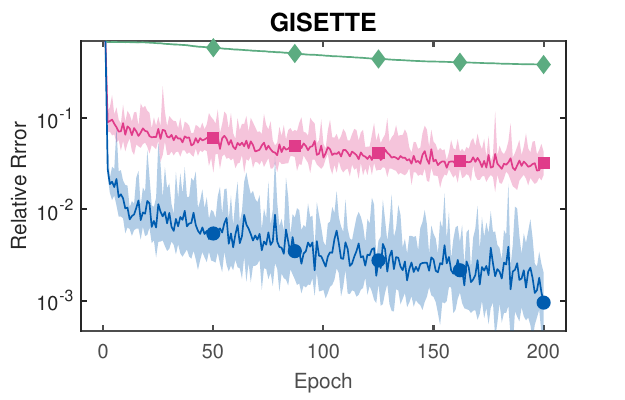}};
 \node[right] at (12,-2.2) {\includegraphics[width=3cm,trim=.5cm 0cm 0cm 0cm,clip]{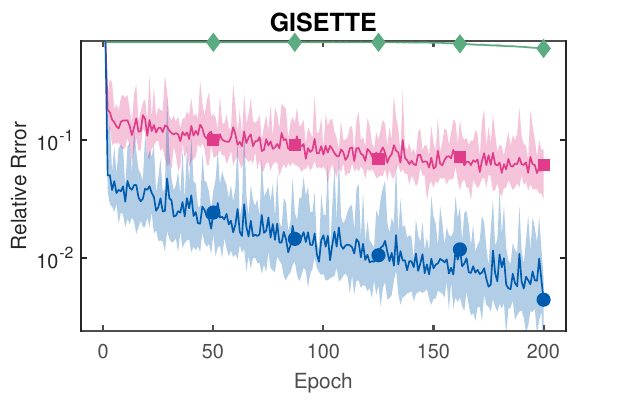}};

  \node[right] at (0.0,-4.4) {\includegraphics[width=3cm,trim=.5cm 0cm 0cm 0cm,clip]{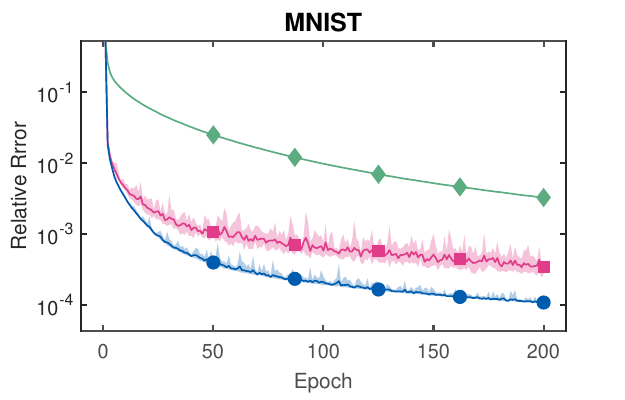}};
	\node[right] at (3,-4.4) {\includegraphics[width=3cm,trim=.5cm 0cm 0cm 0cm,clip]{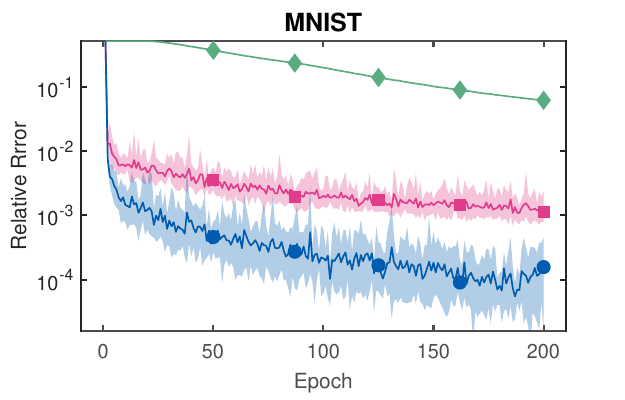}};
	\node[right] at (6,-4.4) {\includegraphics[width=3cm,trim=.5cm 0cm 0cm 0cm,clip]{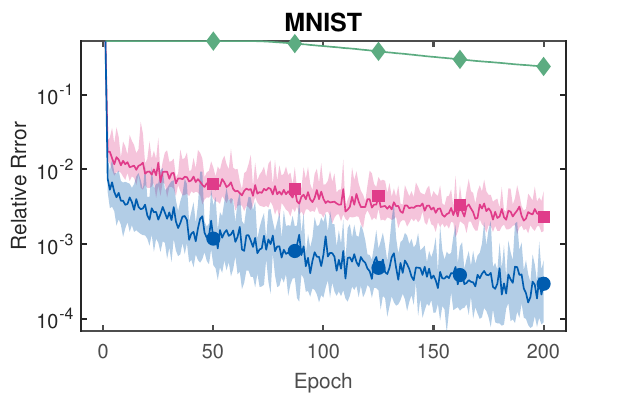}};
 \node[right] at (9,-4.4) {\includegraphics[width=3cm,trim=.5cm 0cm 0cm 0cm,clip]{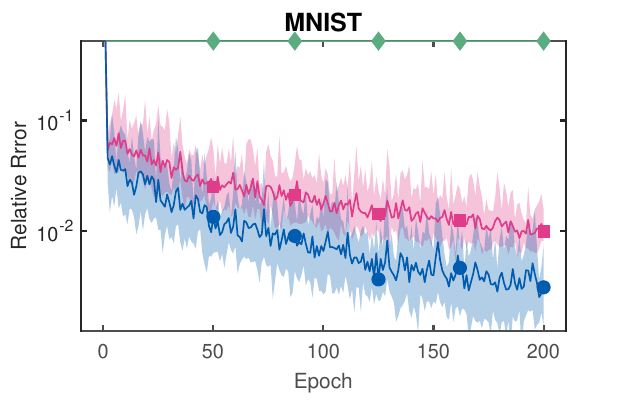}};
 \node[right] at (12,-4.4) {\includegraphics[width=3cm,trim=.5cm 0cm 0cm 0cm,clip]{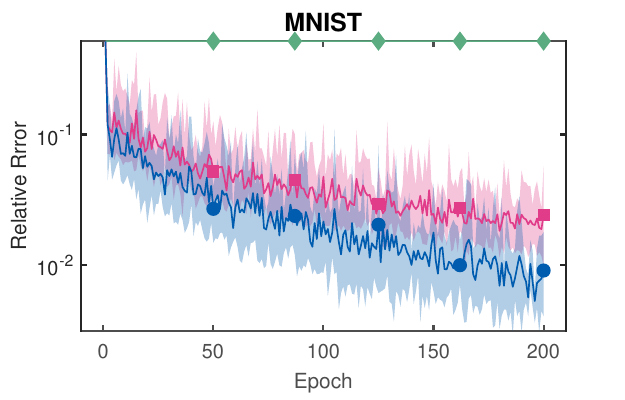}};

\node[right] at (0.0,-7) {\includegraphics[width=3cm,trim=.5cm 0cm 0cm 0cm,clip]{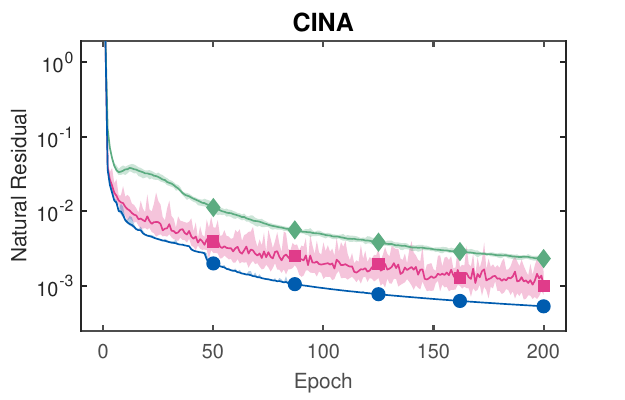}};
	\node[right] at (3,-7) {\includegraphics[width=3cm,trim=.5cm 0cm 0cm 0cm,clip]{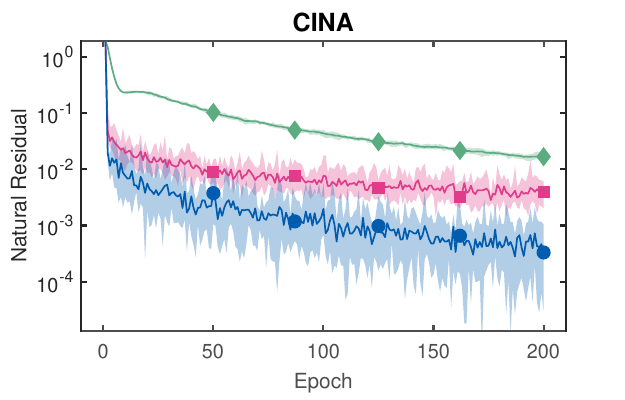}};
	\node[right] at (6,-7) {\includegraphics[width=3cm,trim=.5cm 0cm 0cm 0cm,clip]{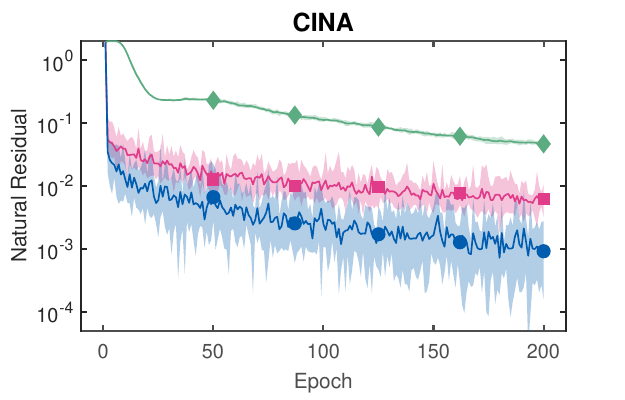}};
 \node[right] at (9,-7) {\includegraphics[width=3cm,trim=.5cm 0cm 0cm 0cm,clip]{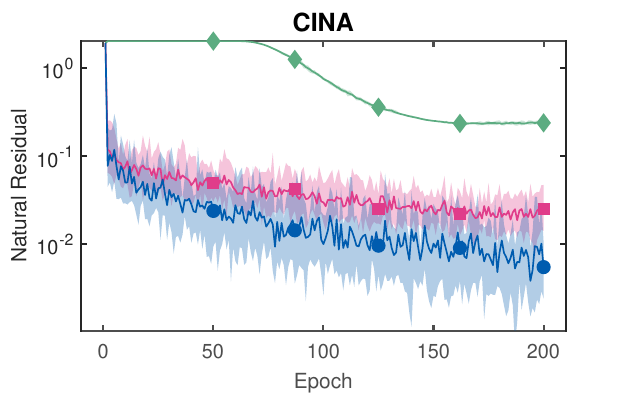}};
 \node[right] at (12,-7) {\includegraphics[width=3cm,trim=.5cm 0cm 0cm 0cm,clip]{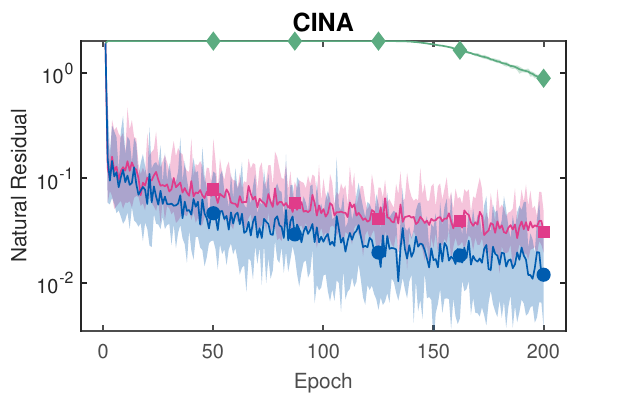}};
 \node[right] at (0.0,-9.2) {\includegraphics[width=3cm,trim=.5cm 0cm 0cm 0cm,clip]{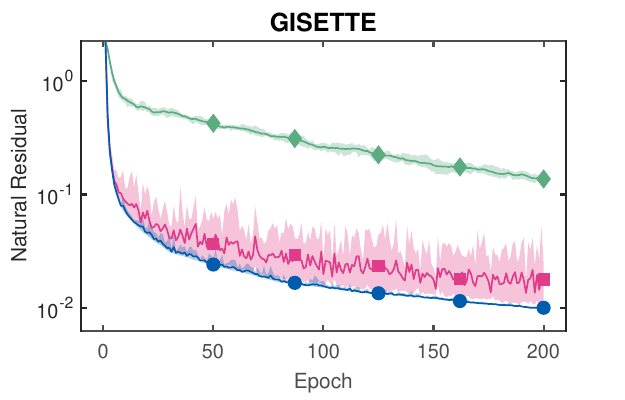}};
	\node[right] at (3,-9.2) {\includegraphics[width=3cm,trim=.5cm 0cm 0cm 0cm,clip]{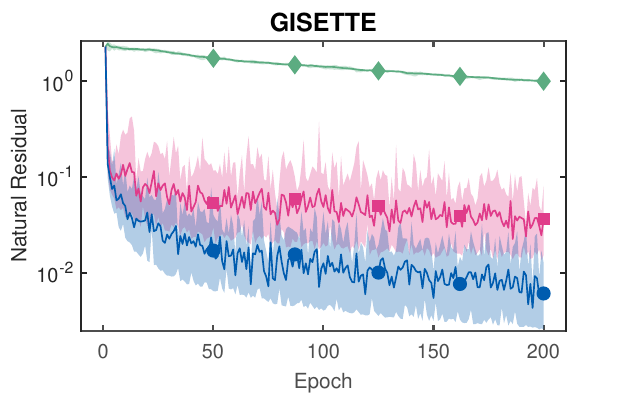}};
	\node[right] at (6,-9.2) {\includegraphics[width=3cm,trim=.5cm 0cm 0cm 0cm,clip]{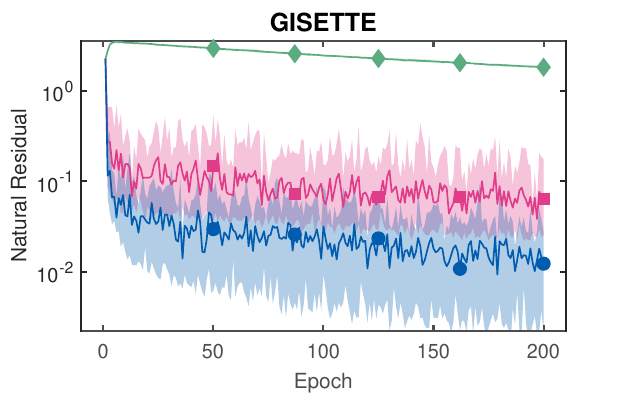}};
 \node[right] at (9,-9.2) {\includegraphics[width=3cm,trim=.5cm 0cm 0cm 0cm,clip]{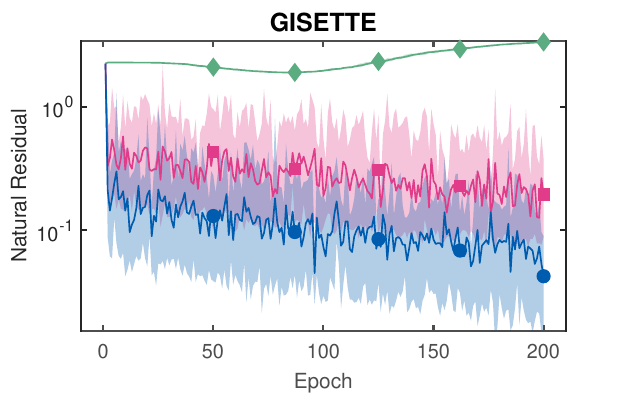}};
 \node[right] at (12,-9.2) {\includegraphics[width=3cm,trim=.5cm 0cm 0cm 0cm,clip]{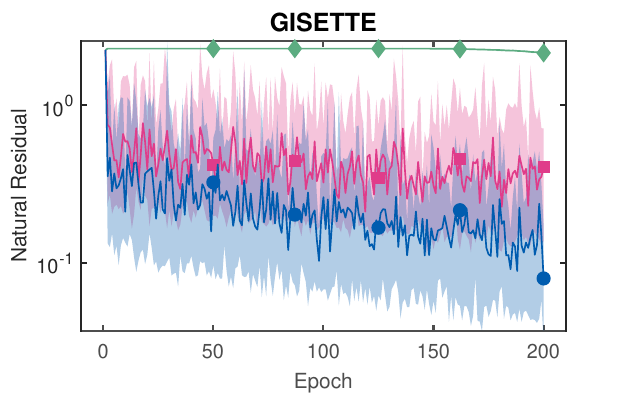}};

  \node[right] at (0.0,-11.4) {\includegraphics[width=3cm,trim=.5cm 0cm 0cm 0cm,clip]{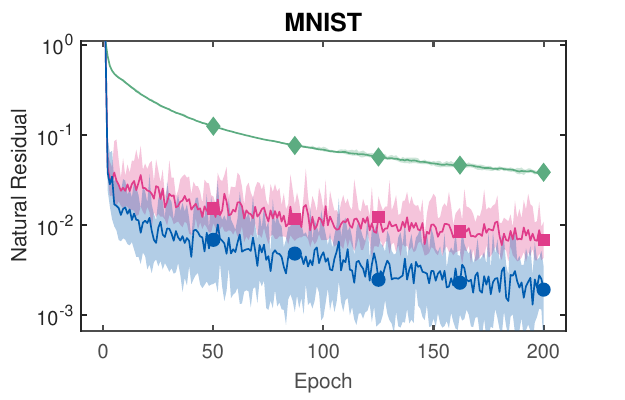}};
	\node[right] at (3,-11.4) {\includegraphics[width=3cm,trim=.5cm 0cm 0cm 0cm,clip]{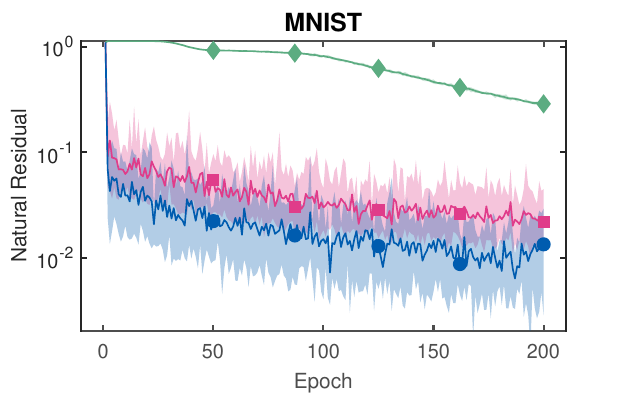}};
	\node[right] at (6,-11.4) {\includegraphics[width=3cm,trim=.5cm 0cm 0cm 0cm,clip]{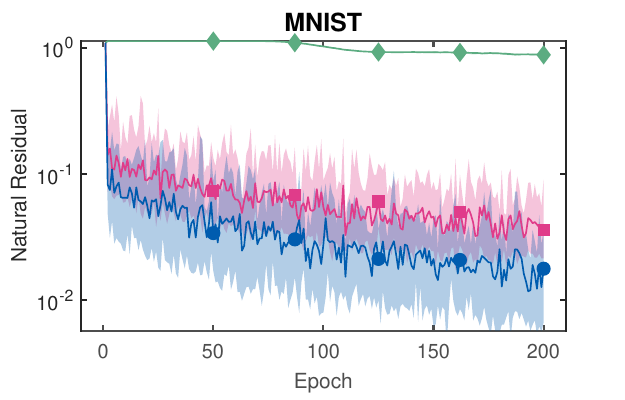}};
 \node[right] at (9,-11.4) {\includegraphics[width=3cm,trim=.5cm 0cm 0cm 0cm,clip]{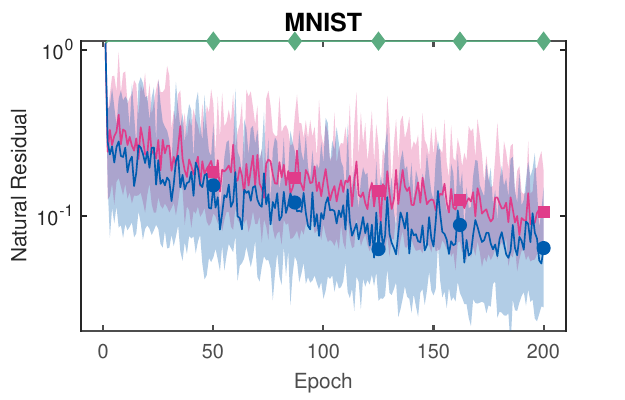}};
 \node[right] at (12,-11.4) {\includegraphics[width=3cm,trim=.5cm 0cm 0cm 0cm,clip]{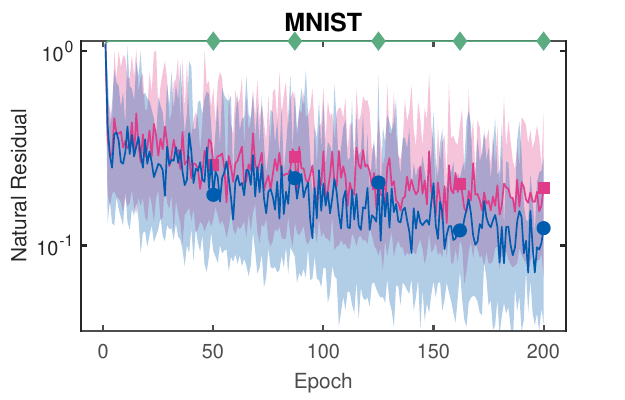}};
 
	\node at (15.3,-2.2) {\rotatebox{-90}{\footnotesize Relative error} };
   \draw [line width=0.2mm] (15.3,0) -- (15.3,-1.2); 
    \draw [line width=0.2mm] (15.2,0) -- (15.3,0);
    \draw [line width=0.2mm] (15.3,-3.2) -- (15.3,-4.4); 
    \draw [line width=0.2mm] (15.2,-4.4) -- (15.3,-4.4);
	\node at (15.3,-9.2) {\rotatebox{-90}{{\footnotesize Natural residual}}};
  \draw [line width=0.2mm] (15.3,-7) -- (15.3,-8); 
    \draw [line width=0.2mm] (15.2,-7) -- (15.3,-7);
    \draw [line width=0.2mm] (15.3,-10.4) -- (15.3,-11.4); 
    \draw [line width=0.2mm] (15.2,-11.4) -- (15.3,-11.4);
    \node[right] at (0.5,-12.8) {{\footnotesize(a)~\texttt{$\alpha=0.01$}}};
	\node[right] at (3.6,-12.8) {{\footnotesize(b)~\texttt{$\alpha=0.05$}}};
	\node[right] at (6.65,-12.8) {{\footnotesize(c)~\texttt{$\alpha=0.1$}}};
 	\node[right] at (9.7,-12.8) {{\footnotesize(d)~\texttt{$\alpha=0.5$}}};
  	\node[right] at (12.8,-12.8) {{\footnotesize(e)~\texttt{$\alpha=1$}}};
\end{tikzpicture}
	\caption{Performance of $\PSGD$, $\EPRR$, and $\NRR$ on the nonconvex binary classification problem \cref{eq:binary-clas} using different step size parameters $\alpha$.}
	\label{fig:2}
\end{figure}

Next, we consider a nonconvex binary classification problem with $\ell_1$-regularization \citep{mason1999boosting,wang2017stochastic,milzarek2019stochastic}:
\begin{equation}
	\label{eq:binary-clas}
	\min_{w\in \Rn}~\frac{1}{n}{\sum}_{i=1}^{n}f(w,i)+\vp(w):=\frac{1}{n}{\sum}_{i=1}^{n}[1-\tanh(b_i\cdot a_i^\top w)] + \nu\norm{w}_1.
\end{equation}
Here, $\tanh$ denotes the hyperbolic tangent function and the parameter $\nu$ is set to $\nu = 0.01$. We conduct the binary classification task on real datasets $(A,b) \in \R^{n \times d} \times \{0,1\}^n$. In our tests, we use the datasets \texttt{CINA}, \texttt{MNIST}, and \texttt{GISETTE}\footnote{Datasets are available at \url{http://www.causality.inf.ethz.ch/data} and \url{www.csie.ntu.edu.tw/~cjlin/libsvmtools/datasets}.}. (In \texttt{MNIST}, we only keep two features). \\[1mm]
\noindent\textbf{Implementation details.}  For all algorithms, we use polynomial step sizes of the form $\alpha_k=\alpha/(\sL+k)$ with $\alpha \in \{0.01,0.05,0.1,0.5,1\}$; the Lipschitz constant is $\sL = 0.8 \cdot \lambda_{\max}(A A^\top)/n$; the index $k$ represents the $k$-th epoch. We run each algorithm for $200$ epochs with $w^0 = 0$; this process is repeated $10$ times for each dataset. The parameter $\lambda$ is set to $\lambda = 1$. The results are reported in \Cref{fig:2}.

 Across all datasets, $\NRR$ appears to converge faster than $\PSGD$ and $\EPRR$ and is relatively robust w.r.t. the choice of $\alpha$. In the initial phases of the training, rapid convergence can be observed for all methods. However, $\NRR$ typically achieves a smaller relative error and natural residual than $\PSGD$ and $\EPRR$. This improved performance might originate from the design of $\NRR$: it incorporates without-replacement sampling and applies the $\ell_1$-proximity operator at each iteration to maintain sparsity. \vspace{0.5ex}

\noindent\textbf{Varying the parameter $\lambda$.}
 In \Cref{fig:exp-2 extra}, we additionally evaluate the performance of $\NRR$ using different values of the hyperparameter $\lambda \in \{0.1,1,10\}$. As illustrated in \Cref{fig:exp-2 extra}, larger values of $\lambda$ generally lead to an improved performance when solving the problem \eqref{eq:binary-clas} across the tested datasets. In this experiment, the overall convergence behavior $\NRR$ seems to show similar trends when using different $\lambda$. Note that in the previous tests presented in \Cref{fig:2}, we select $\lambda=1$. The results in \Cref{fig:exp-2 extra} indicate that better performance can be potentially achieved if the parameter $\lambda$ is tuned more carefully.
 
\begin{figure}[t]
\centering
	\setlength{\abovecaptionskip}{-3pt plus 3pt minus 0pt}
	\setlength{\belowcaptionskip}{-10pt plus 3pt minus 0pt}
	\centering
	\includegraphics[width=8.0cm]{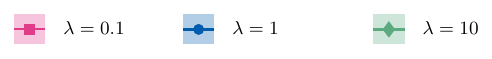} \vspace{0ex}
	
	\hspace*{-2ex}
	\begin{tikzpicture}[scale=1]
	\node[right] at (0.0,0) {\includegraphics[width=4.7cm,trim=0.27cm 0cm 0cm 0cm,clip]{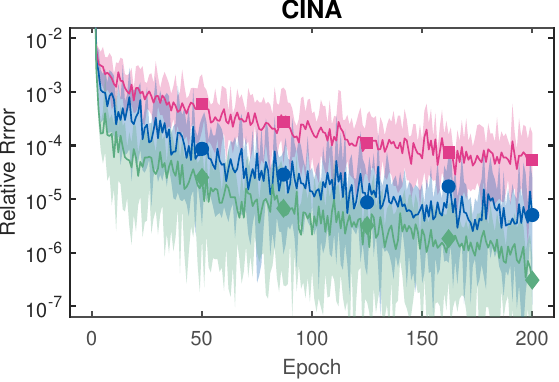}};
	\node[right] at (5,0) {\includegraphics[width=4.7cm,trim=0.28cm 0cm 0cm 0cm,clip]{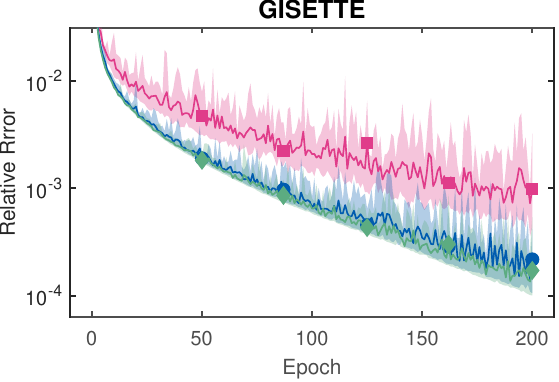}};
	\node[right] at (10,0) {\includegraphics[width=4.7cm,trim=0.26cm 0cm 0cm 0cm,clip]{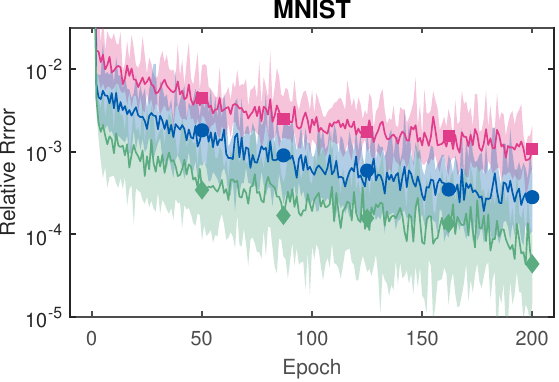}};
	\node[right] at (0.0,-3.8) {\includegraphics[width=4.7cm,trim=.27cm 0cm 0cm 0cm,clip]{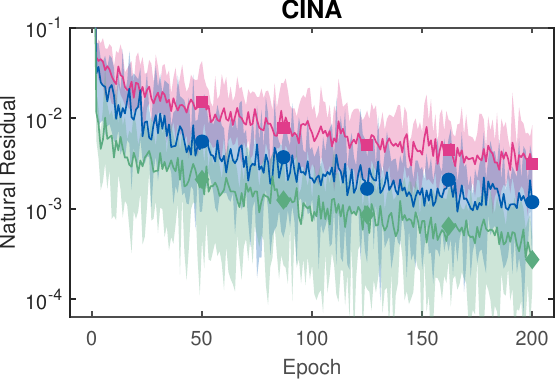}};
	\node[right] at (5,-3.8) {\includegraphics[width=4.7cm,trim=.28cm 0cm 0cm 0cm,clip]{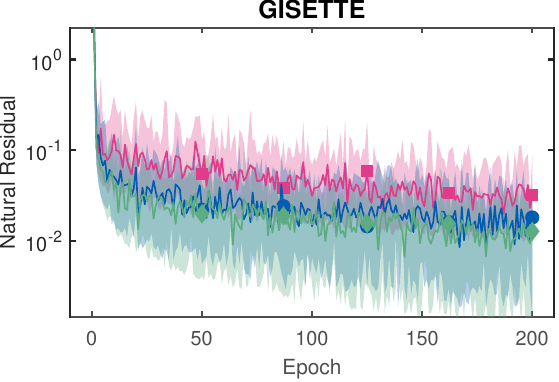}};
	\node[right] at (10,-3.8) {\includegraphics[width=4.7cm,trim=.26cm 0cm 0cm 0cm,clip]{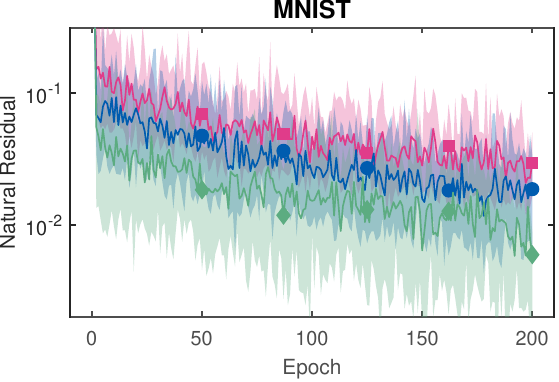}};
%
	%
	\node at (15.2,0.1) {\rotatebox{-90}{{\footnotesize Relative error}}};
	\node at (15.2,-3.6) {\rotatebox{-90}{{\footnotesize Natural residual}}};
\end{tikzpicture}
	\caption{Performance of $\NRR$ on the nonconvex binary classification problem \eqref{eq:binary-clas} for different choices of $\lambda>0$.}
	\label{fig:exp-2 extra}
\end{figure}

\begin{figure}[t]
\centering
	\setlength{\abovecaptionskip}{-3pt plus 3pt minus 0pt}
	\setlength{\belowcaptionskip}{-10pt plus 3pt minus 0pt}
	\centering
	\includegraphics[width=8.0cm]{figs//legend.pdf} \vspace{-1ex}
	
	\hspace*{-2ex}
	\begin{tikzpicture}[scale=1]
	\node[right] at (0.0,0) {\includegraphics[width=5cm,trim=.8cm 0cm 0cm 0cm,clip]{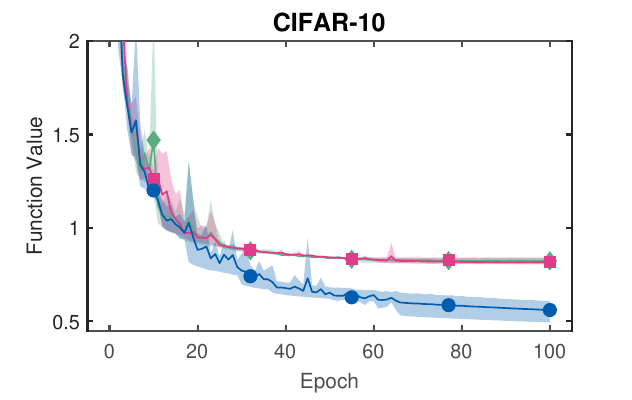}};
	\node[right] at (5,0) {\includegraphics[width=5cm,trim=.58cm 0cm 0cm 0cm,clip]{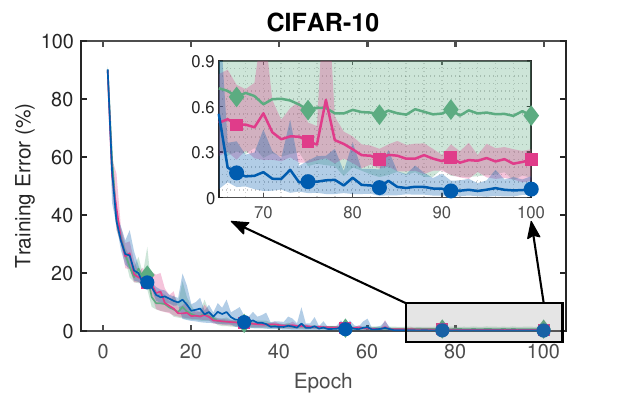}};
	\node[right] at (10,0) {\includegraphics[width=5cm,trim=.58cm 0cm 0cm 0cm,clip]{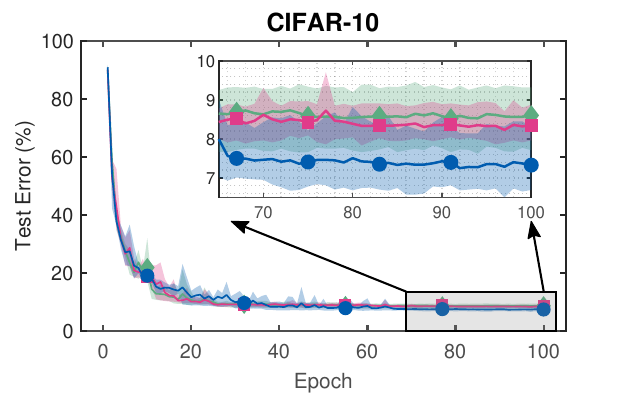}};
	\node[right] at (0.0,-3.8) {\includegraphics[width=5cm,trim=.8cm 0cm 0cm 0cm,clip]{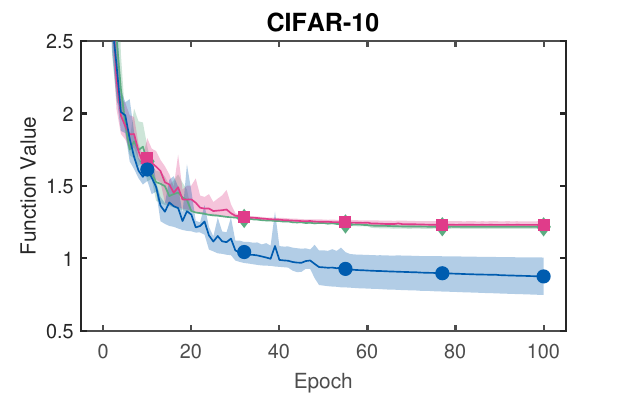}};
	\node[right] at (5,-3.8) {\includegraphics[width=5cm,trim=.58cm 0cm 0cm 0cm,clip]{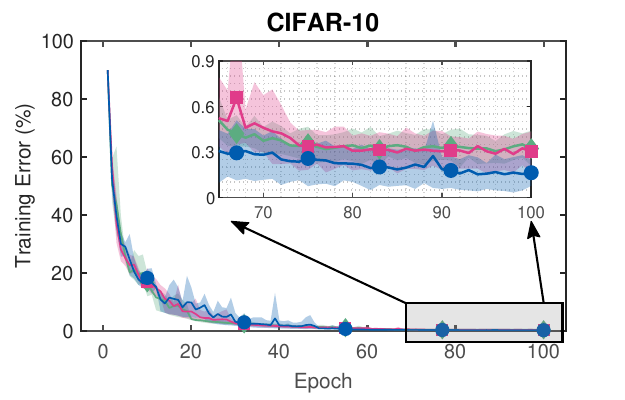}};
	\node[right] at (10,-3.8) {\includegraphics[width=5cm,trim=.58cm 0cm 0cm 0cm,clip]{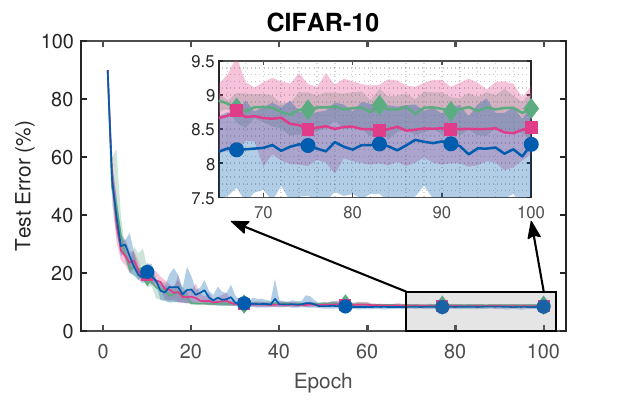}};
	\node[right] at (1.1,-6) {{\footnotesize(a)~{Function value.}}};
	\node[right] at (6.15,-6) {{\footnotesize(b)~Training error.}};
	\node[right] at (11.3,-6) {{\footnotesize(c)~{Test error.}}};
	\node at (15.2,0) {\rotatebox{-90}{{\footnotesize ResNet-18}}};
	\node at (15.2,-4) {\rotatebox{-90}{{\footnotesize VGG-16}}};
\end{tikzpicture}
	\caption{Performance of $\PSGD$, $\EPRR$, and $\NRR$ on the deep learning problem \cref{eq:deep} applying the same learning rate policy. The performance is evaluated on \texttt{CIFAR-10} using the neural network architectures ResNet-18 and VGG-16.}
	\label{fig:exp-3}
\end{figure}

\subsection{Deep Learning for Image Classification}
In this subsection, we study multiclass image classification for the dataset \texttt{CIFAR-10} \citep{krizhevsky2009learning}. \texttt{CIFAR-10} contains $10$ classes of images and each class contains $6,000$ images. In this experiment, we split the dataset into $n_{\text{train}}=50,000$ training samples and $n_{\text{test}}=10,000$ test samples. We consider standard ResNet-18 \citep{he2016deep} and VGG-16 \citep{simonyan2015deep} architectures with the cross-entropy loss function \citep{yang2021stochastic,tang2021data} and elastic net regularizer \citep{zou2005regularization}: 
\begin{equation} \label{eq:deep}
	\min_{w\in \Rn}~-\frac{1}{n}{\sum}_{i=1}^{n} \log \left(\frac{\exp(\cT(w,a_i)^\top e^{b_i} )}{{\sum}_{j=1}^{10} \exp(\cT(w,a_i)^\top e^j) } \right) + \nu_1 \norm{w}_1 + \nu_2 \norm{w}^2.
\end{equation}
Here, the tuple $(a_i,b_i)$ is a training sample with the color image $a_i\in\R^{32\times 32 \times 3}$ and the corresponding label $b_i\in [10]$. The operator $\cT(w,\cdot):\R^{32\times 32 \times 3} \mapsto \R^{10}$, which maps an image to a ten dimensional vector, represents the neural network architecture with the weights $w$. The vector $e^j\in \R^{10}$ is a unit vector with its $j$-th element equal to $1$. To avoid overfitting while maintaining the sparsity of the parameters of the model, the elastic net regularizer is applied with parameters set to $\nu_1=10^{-6}$ and $\nu_2=10^{-4}$. \\[1mm]  %
\noindent\textbf{Implementation details.} We use adaptive step sizes ($\mathsf{ReduceLROnPlateau}$  with initial step size $\alpha=0.1$ in PyTorch \citep{paszke2019pytorch}) and set $\lambda=10^{-2}$ for $\NRR$ (in both architectures). We train ResNet-18 and VGG-16 for $100$ epochs with batch size $128$ and run each algorithm $5$ times independently. 

The results in \Cref{fig:exp-3} show that $\NRR$ achieves the lowest training loss. While all tested methods reach a low training error at the end, $\NRR$ seems to slightly outperform $\PSGD$ and $\EPRR$. Similar trends can also be observed when considering the test error. 

\section{Conclusion}
In this paper, we propose a new proximal random reshuffling method ($\NRR$) for large-scale composite problems. Our approach adopts a normal map-based perspective and uses stochastic gradient information generated through without-replacement sampling. The obtained complexity results match the existing bounds for $\RR$ in the smooth case and improve other known complexity bounds for this class of problems in terms of gradient evaluations. Under the (global) Polyak-{\L}ojasiewicz condition and in the interpolation setting, $\NRR$ provably converges to an optimal function value at a linear rate. In addition, accumulation points of a sequence of iterates $\{w^k\}_k$ generated by $\NRR$ are shown to correspond to stationary points of the problem. Finally, under the (local) Kurdyka-{\L}ojasiewicz inequality, we establish last-iterate convergence for $\NRR$ and derive asymptotic rates of convergence. Numerical experiments illustrate that the proposed method effectively maintains feasibility and performs favorably on nonconvex, nonsmooth problems and learning tasks.

\acks{The authors would like to thank the Action Editor and two anonymous reviewers for their detailed and constructive comments, which have helped greatly to improve the quality and presentation of the manuscript.

Xiao Li was partly supported by the National Natural Science Foundation of China under grant No. 12201534 and by the Shenzhen Science and Technology Program under Grant No. RCYX20221008093033010. Andre Milzarek was partly supported by the National Natural Science Foundation of China (Foreign Young Scholar Research Fund Project) under Grant No. 12150410304, by the Shenzhen Science and Technology Program under Grant No. RCYX20210609103124047 and RCYX20221008093033010, by the Shenzhen Stability Science Program 2023, Shenzhen Key Lab of Multi-Modal Cognitive Computing, and by the Internal Project Fund from the Shenzhen Research Institute of Big Data under Grant T00120230001.}


%

\appendix
\section{Preparatory Results}
\subsection{Basic Mathematical Tools}
In the following, we present several fundamental results concerning sequences $\{y_k\}_k$ of real numbers and their convergence. We start with a discrete variant of Gronwall's inequality \cite[Appendix B]{borkar2009stochastic}.
	
\begin{lemma}[Gronwall's Inequality]
    \label[lemma]{Thm:G}
    Let $\{a_k\}_k$, $\{y_k\}_k \subseteq \R_+$ and $p,q\geq0$ be given. Assume $y_{k+1} \leq p + q \sum_{j=1}^k a_j y_j$ for all $k\geq 1$. 
    Then, $y_{k+1} \leq p \cdot \exp({q\sum_{j=1}^k a_j})$ for all $k\geq 1$.
\end{lemma}

In order to establish global convergence of $\NRR$, we will use the well-known supermartingale convergence theorem, see, e.g., \cite[Proposition A.31]{ber16}.

\begin{theorem} \label[theorem]{Thm:super-martingale-convergece}
Let $\{y_k\}_k$, $\{p_k\}_k$, $\{q_k\}_k$, and $\{\gamma_k\}_k \subseteq \R_+$ be non-negative sequences. Assume that $\sum_{k=1}^\infty \gamma_k < \infty$, $\sum_{k=1}^{\infty}q_k<\infty$, and $y_{k+1} \leq (1+\gamma_k)y_k-p_k+q_k$ for all $k\geq 1$. Then, $\{y_k\}_k$ converges to some $y \geq 0$ and it follows $\sum_{k=1}^{\infty}p_k<\infty$.
\end{theorem}

\subsection{Sampling Without Replacement}

In the following lemma, we restate results for the variance of sampling a collection of vectors from a finite set of vectors without replacement, \cite[Lemma 1]{mishchenko2020random}. 
\begin{lemma}
	\label[lemma]{lem:sample-without-re}
 Let $X_1,\dots,X_n \in \Rn$ be given and let $\bar X := \frac{1}{n} \sum_{i=1}^n X_i$ and $\sigma^2 := \frac{1}{n} \sum_{i=1}^n \|X_i -\bar X \|^2$ denote the associated average and population variance. Let $t \in [n]$ be fixed and let $X_{\pi_1},\dots,X_{\pi_t}$ be sampled uniformly without replacement from $\{X_1,\dots,X_n\}$. Then, we have
\[
	\Exp[\bar X_\pi] = \bar{X} \quad \text{and} \quad \Exp[\|\bar X_\pi-\bar{X}\|^2] = \frac{n-t}{t(n-1)}\cdot\sigma^2,
\quad \text{where} \quad \bar X_\pi = \frac{1}{t}{\sum}_{i=1}^t X_{\pi_i}. \]
\end{lemma}

\section{Approximate Descent of the Merit Function}

\subsection{Basic Estimates}
In this subsection, we provide two basic estimates that will be used in the derivation of the approximate descent property of $\mer$ stated in \cref{lem:merit-descent-1}. 
\begin{lemma}
		\label[lemma]{lem:first-descent}
		Let \ref{A1}--\ref{A2} hold and let $\{w^k\}_k$ and $\{z^k\}_k$ be generated by $\NRR$ with $\lambda\in(0,\frac1\rho)$ and step sizes $\{\alpha_k\}_k \subseteq \R_{++}$. Then, for all $k\geq1$, we have 
  \[\psi(w^{k+1})-\psi(w^k)\leq  \langle \Fnor(z^k) + \lambda^{-1}({z^{k+1}-z^k}), w^{k+1}-w^k \rangle + \left[{\frac{\sL+\rho}{2}-\frac{1}{\lambda}}\right]\norm{w^{k+1}-w^k}^2.\]
	\end{lemma}

\begin{proof}
	First, invoking \sL-smoothness of $f$, \cref{eq:weak-cvx}, and $\lambda^{-1}(z^{k+1}-w^{k+1}) \in \partial\vp(w^{k+1})$, we obtain
		\begin{align*}
			&\hspace{-2ex}\psi(w^{k+1})-\psi(w^k)\\ 
			& \leq \iprod{\nabla f(w^k)}{w^{k+1}-w^k}+\frac{\sL}{2}\norm{w^{k+1}-w^k}^2 + \vp(\prox{\lambda\vp}(z^{k+1}))-\vp(\prox{\lambda\vp}(z^k))\\& \leq \iprod{\nabla f(w^k)}{w^{k+1}-w^k}+\frac{\sL+\rho}{2}\norm{w^{k+1}-w^k}^2 + \frac{1}{\lambda}\iprod{z^{k+1}-w^{k+1}}{w^{k+1}-w^k},
		\end{align*}
	Using $\Fnor(z^k)=\nabla f(w^k) + \lambda^{-1}(z^k-w^k)$, this yields
	\begin{align*}
		&\hspace{-2ex}\psi(w^{k+1})-\psi(w^k) \\ &\leq \iprod{\nabla f(w^k)}{w^{k+1}-w^k}+\frac{\sL+\rho}{2}\norm{w^{k+1}-w^k}^2 + \iprod{\Fnor(z^k)-\nabla f(w^k) }{w^{k+1}-w^k} \\&\quad + {\lambda^{-1}}\iprod{z^{k+1}-w^{k+1}}{w^{k+1}-w^k} - {\lambda^{-1}}\iprod{z^k-w^k}{w^{k+1}-w^k}\\
		& = \left[{\frac{\sL+\rho}{2}-\frac{1}{\lambda}}\right]\norm{w^{k+1}-w^k}^2 + 	\iprod{\Fnor(z^k)+\lambda^{-1}(z^{k+1}-z^k)}{w^{k+1}-w^k},
	\end{align*}
which completes the proof.
\end{proof}
	\begin{lemma}
		\label[lemma]{lem:Fnor-bound}
		Under the conditions stated in \cref{lem:first-descent}, we have
        \begingroup
        \allowdisplaybreaks
		\begin{align*}
			\norm{\Fnor(z^{k+1})}^2&\leq \left[{1-\frac{n\alpha_k}{\lambda}}\right]^2\norm{\Fnor(z^k)}^2 + [{\sL+{\lambda^{-1}}}]^2 \norm{w^{k+1}-w^k}^2 + \frac{1}{\lambda^2}\norm{e^k}^2 \\
			&\quad+\frac{2}{\lambda}\left[{1-\frac{n\alpha_k}{\lambda}}\right]\iprod{\Fnor(z^k)}{\lambda({\nabla f(w^{k+1}) -\nabla f(w^k)}) - ({w^{k+1}-w^k}) + e^k } \\
			&\quad + {2}{\lambda^{-1}} \iprod{\nabla f(w^{k+1}) -\nabla f(w^k)}{e^k}  - {2}{\lambda^{-2}}\iprod{e^k}{w^{k+1}-w^k}.
		\end{align*}
        \endgroup
	\end{lemma}
 
\begin{proof}
	Applying the definition of $\Fnor$ and \cref{eq:update-z}, we can expand $\Fnor(z^{k+1})$ as follows
		\begin{align*}
			\Fnor(z^{k+1})&=\nabla f(w^{k+1}) + {\lambda^{-1}}({z^{k+1}-w^{k+1}}) \\
   &= \Fnor(z^k) + ({\nabla f(w^{k+1})-\nabla f(w^k)}) + {\lambda^{-1}}({z^{k+1}-z^k+w^k-w^{k+1}})\\
			& = \left[{1-\frac{n\alpha_k}{\lambda}}\right]\Fnor(z^k) + ({\nabla f(w^{k+1})-\nabla f(w^k)}) - {\lambda^{-1}}({w^{k+1}-w^k}) + {\lambda^{-1}} e^k.
		\end{align*}	
	Using the $\sL$-smoothness of $f$, we can infer
 \begingroup
 \allowdisplaybreaks
	\begin{align*}
		&\hspace{-2ex}\norm{\Fnor(z^{k+1})}^2\\&=\left[{1-\frac{n\alpha_k}{\lambda}}\right]^2\norm{\Fnor(z^k)}^2 +\norm{\nabla f(w^{k+1}) - \nabla f(w^k)}^2  - \frac{2}{\lambda^2}\iprod{e^k}{w^{k+1}-w^k}\\
		&\quad + 2\left[{1-\frac{n\alpha_k}{\lambda}}\right]\iprod{\Fnor(z^k)}{({\nabla f(w^{k+1}) -\nabla f(w^k)}) - {\lambda^{-1}}({w^{k+1}-w^k})+{\lambda^{-1}}e^k }\\
		&\quad  + 2\iprod{\nabla f(w^{k+1}) -\nabla f(w^k)}{{\lambda^{-1}}e^k-{\lambda^{-1}}({w^{k+1}-w^k})} + \frac{1}{\lambda^2}\norm{w^{k+1}-w^k}^2 + \frac{1}{\lambda^2}\norm{e^k}^2	\\
		&\leq \left[{1-\frac{n\alpha_k}{\lambda}}\right]^2\norm{\Fnor(z^k)}^2 + \left[\sL+\frac{1}{\lambda}\right]^2\norm{w^{k+1}-w^k}^2 \\
		&\quad + {2}{\lambda^{-1}} \iprod{\nabla f(w^{k+1}) -\nabla f(w^k)}{e^k} + {\lambda^{-2}}\norm{e^k}^2 - 2{\lambda^{-2}}\iprod{e^k}{w^{k+1}-w^k}\\
		&\quad + \frac{2}{\lambda}\left[{1-\frac{n\alpha_k}{\lambda}}\right]\iprod{\Fnor(z^k)}{\lambda({\nabla f(w^{k+1}) -\nabla f(w^k)}) - ({w^{k+1}-w^k}) +e^k },
	\end{align*}
 \endgroup
	as desired.
\end{proof}

\subsection{\texorpdfstring{Proof of \cref{lem:est-err}}{Proof of Lemma 2}} \label{proof:est-err}
\begin{proof} We start with the proof of part (a). Using the definition of the error term $e^k$ in \cref{eq:error-e} and the triangle inequality, we have
 \begin{align}
     \nonumber \norm{e^k} &\leq \alpha_k {\sum}_{i=1}^{n}\norm{\Fnor(z^k_i) - \Fnor(z^k)}+\alpha_k \left\|{{\sum}_{i=1}^{n} \nabla f(w^k_i,\pi^k_i) - \nabla f(w^k,\pi^k_i)}\right\| \\
     \nonumber& \quad + \alpha_k \left\|{{\sum}_{i=1}^{n} \nabla f(w^k,\pi^k_i) - \nabla f(w^k_i)} \right\| \\
	\nonumber &\leq \alpha_k {\sum}_{i=1}^{n} \big[\norm{\Fnor(z^k_i) - \Fnor(z^k)}+ \sL \norm{ w^k_i - w^k} \big] + \alpha_k \left\|{{\sum}_{i=1}^{n} \nabla f(w^k) - \nabla f(w^k_i)}\right\| \\
	&\leq \alpha_k {\sum}_{i=1}^{n}\norm{\Fnor(z^k_i) - \Fnor(z^k)}+ \frac{2\sL}{1-\lambda\rho} \cdot \alpha_k {\sum}_{i=1}^{n}\norm{z^k_i - z^k},
 \label{eq:lem-est-err-est-1}
 \end{align}
	where we applied the Lipschitz continuity of the proximity operator, $\| w^k_i - w^k \| \leq \| z^k_i - z^k\|/(1-\lambda \rho)$, in the last line, cf. \cref{eq:prox-coco}. Similarly, it holds that
	%
    \begingroup
    \allowdisplaybreaks
		\begin{align}
			\nonumber\norm{\Fnor(z^k_i) - \Fnor(z^k)} &= \norm{\nabla f(w^k_i) - \nabla f(w^k) + \lambda^{-1}[(z^k_i-z^k)-(w^k_i-w^k)]} \\
			&\hspace{-4ex}\leq (\sL+\lambda^{-1})\norm{w^k_i-w^k} + \lambda^{-1} \norm{z^k_i-z^k}\leq \frac{\sL + 2\lambda^{-1}-\rho}{1-\lambda\rho}\norm{z^k_i-z^k}. \label{eq:lem-4-est-2}
		\end{align}
    \endgroup
	Combining \cref{eq:lem-4-est-2} and \cref{eq:lem-est-err-est-1}, we readily obtain
	\begin{equation}
		\label{eq:lem-est-err-est-3}
		\norm{e^k} \leq \sC_r \alpha_k {\sum}_{i=1}^{n} \norm{z^k_i-z^k} \quad \text{where} \quad \sC_r := \frac{3\sL + 2\lambda^{-1}-\rho}{1-\lambda\rho} =\frac{\sqrt{\CL}}{2}.
	\end{equation} 
 Based on \cref{eq:norm-am-update}, \cref{eq:update-z} and applying \cref{eq:lem-4-est-2}, it follows
	\begin{align}
    \nonumber\norm{z^k_{i+1}-z^k} & =\alpha_k \Big \| i \Fnor(z^k) + {\sum}_{j=1}^{i} \big[\Fnor(z^k_j) - \Fnor(z^k) + \nabla f(w^k_j,\pi^k_j) - \nabla f(w^k_j) \big] \Big\|\\ \nonumber
    &\leq i \alpha_k \norm{\Fnor(z^k)} + {(\sL + 2\lambda^{-1}-\rho)}{(1-\lambda\rho)^{-1}} \alpha_k {\sum}_{j=1}^{i}\norm{z^k-z^k_j} \\ & \hspace{4ex} + \alpha_k \Big\|{\sum}_{j=1}^{i} \big[\nabla f(w^k_j,\pi^k_j) - \nabla f(w^k_j) \big]\Big\|
   \label{eq:lem-est-err-est-4}
	\end{align}		
 for all $i=0,\ldots,n-1$. We now continue with the last term in the estimate \cref{eq:lem-est-err-est-4}. Setting $\Upsilon_i := \norm{{\sum}_{j=1}^{i}[\nabla f(w^k,\pi^k_j) - \nabla f(w^k)] }$ and using the triangle inequality, the $\sL$-Lipschitz continuity of the component gradients $\nabla f(\cdot,i)$, $i \in [n]$, and the $(1-\lambda\rho)^{-1}$-Lipschitz continuity of the proximity operator $\prox{\lambda\vp}$, it holds that
 \begin{equation}
     \label{eq:lem-est-err-est-5}
     \begin{aligned}
    &\hspace{-3ex}\Big\|{\sum}_{j=1}^{i}\big[\nabla f(w^k_j,\pi^k_j) - \nabla f(w^k_j)\big]\Big\|
		\\& \leq {\sum}_{j=1}^{i} \big[\norm{\nabla f(w^k_j,\pi^k_j) - \nabla f(w^k,\pi^k_j)} + \norm{\nabla f(w^k) - \nabla f(w^k_j)} \big] + \Upsilon_i  \\
		&\leq  2\sL {\sum}_{j=1}^{i} \norm{w^k_j-w^k} + \Upsilon_i \leq \frac{2\sL}{1-\lambda\rho} {\sum}_{j=1}^{i} \norm{z^k_j-z^k} + \Upsilon_i. 
     \end{aligned}
 \end{equation}
	Hence, combining \cref{eq:lem-est-err-est-4} and \cref{eq:lem-est-err-est-5}, we can infer
	\begin{equation}
		\label{eq:lem-est-err-est-6}
		\norm{z^k_{i+1}-z^k}\leq \alpha_k \Big[i \norm{\Fnor(z^k)} +  \sC_r {\sum}_{j=1}^{i}\norm{z^k_j-z^k} + \Upsilon_{i} \Big] \quad \forall~i \in [n-1].
	\end{equation}
Consequently, summing the term $\norm{z^k_i-z^k}$ from $i=1$ to $n$ and using \cref{eq:lem-est-err-est-6}, we obtain
		\begin{align*}
			{\sum}_{i=1}^{n} \norm{z^k_i-z^k} & = {\sum}_{i=1}^{n-1} \norm{z^k_{i+1}-z^k} \\ &\leq  \sC_r\alpha_{k} {\sum}_{i=1}^{n-1}{\sum}_{j=1}^{i}\norm{z^k-z^k_j} + \alpha_k {\sum}_{i=1}^{n-1}[i\norm{\Fnor(z^k)} + \Upsilon_{i}]  \\
			&\leq \sC_rn \alpha_{k} {\sum}_{i=1}^{n} \norm{z^k_i-z^k} + \frac{n(n-1)\alpha_k}{2} \norm{\Fnor(z^k)}  +  \alpha_k {\sum}_{i=1}^{n-1}\Upsilon_{i}.
		\end{align*}
	Rearranging the terms in the previous inequality, it follows
	\begin{equation*}
			{\sum}_{i=1}^{n} \norm{z^k_i-z^k}  \leq \frac{ \alpha_k [\frac{n^2-n}{2}\norm{\Fnor(z^k)} + {\sum}_{i=1}^{n-1}\Upsilon_{i}]}{1-\sC_rn \alpha_k} \leq \alpha_k \Big[ n^2\norm{\Fnor(z^k)} + 2{\sum}_{i=1}^{n-1}\Upsilon_{i} \Big],
	\end{equation*}
 where the last line is due to $\sC_rn \alpha_k\leq \frac12$. Inserting this estimate into \cref{eq:lem-est-err-est-3}, we obtain 
	\[
			\norm{e^k} \leq \sC_r\alpha_k^2 \Big[n^2\norm{\Fnor(z^k)} +2{\sum}_{i=1}^{n-1}\Upsilon_{i} \Big].
	\]
Taking squares on both sides of this inequality and using $(\sum_{i=1}^{j} a_i)^2 \leq j\sum_{i=1}^{j} a_i^2$ with $j=2$ and $j=n-1$, this yields
\begin{equation}
    \label{eq:lem-est-use-later}
        \norm{e^k}^2 \leq 2\sC_r^2  \alpha_k^4 \Big[n^4\norm{\Fnor(z^k)}^2 + 4(n-1){\sum}_{i=1}^{n-1}\Upsilon_{i}^2 \Big]. 
\end{equation}
Thus, noticing that $\Upsilon_{i}^2 \leq i{\sum}_{j=1}^{n}\norm{\nabla f(w^k,j) - \nabla f(w^k) }^2=in\sigma_k^2$ and $\sum_{i=1}^{n-1} i = \frac{n(n-1)}{2}$, we finally obtain 
\[
 \norm{e^k}^2 \leq 2\sC_r^2 n^4 \alpha_k^4 [\norm{\Fnor(z^k)}^2 + 2 \sigma_k^2].
\]

We continue with the proof of part (b). 
Taking the conditional expectation $\Exp_k[\cdot]$ in \cref{eq:lem-est-use-later}, we have
\begin{equation}\label{eq:lem-est-(b)-1}
    \Exp_k[\norm{e^k}^2] \leq  2\sC_r^2 \alpha_k^4 \Big[n^4\norm{\Fnor(z^k)}^2 + 4(n-1){\sum}_{i=1}^{n-1}\Exp_k[\Upsilon_{i}^2] \Big].
\end{equation}
According to \cref{lem:sample-without-re}, it holds that
\begin{align*}
\Exp_k[\Upsilon_{i}^2] &= \Exp_k\Big[\Big\|{\sum}_{j=1}^{i} (\nabla f(w^k,\pi^k_j) - \nabla f(w^k) )\Big\|^2\Big] \\ &= \frac{i(n-i)}{n(n-1)} {\sum}_{j=1}^{n} \| \nabla f(w^k,j) - \nabla f(w^k) \|^2 = \frac{i(n-i)\sigma_k^2}{n-1}\leq \frac{n^2\sigma_k^2}{4(n-1)}.
\end{align*}
Using this in \cref{eq:lem-est-(b)-1}, it follows $
 \Exp_k[\norm{e^k}^2] \leq  2\sC_r^2 n^4 \alpha_k^4 [\norm{\Fnor(z^k)}^2 + \frac{\sigma_k^2}{n} \big]$.
\end{proof}
\vspace{-4ex}

\subsection{\texorpdfstring{Proof of \cref{lem:merit-descent-0}}{Proof of Lemma 5}}\label{subsec:proof-lem:merit-descent-0}

\begin{proof} Applying \cref{lem:first-descent,lem:Fnor-bound}, it follows 
	\begingroup
	\allowdisplaybreaks
	\begin{align}
		\label{eq:lem-3-est-1}
		&\hspace{-1ex} \mer(z^{k+1}) - \mer(z^k) \notag \\&\leq \left[{\frac{\sL+\rho}{2}-\frac{1}{\lambda}}\right]\norm{w^{k+1}-w^k}^2 + \iprod{\Fnor(z^k)+\lambda^{-1}(z^{k+1}-z^k)}{w^{k+1}-w^k} \notag \\&\quad  + \frac{\tau\lambda}{2}[{\norm{\Fnor(z^{k+1})}^2 -\norm{\Fnor(z^k)}^2}] \notag  \\
		&\leq \left[{\frac{\sL+\rho}{2}-\frac{1}{\lambda}}\right]\norm{w^{k+1}-w^k}^2 + 	\iprod{\Fnor(z^k)+\lambda^{-1}(z^{k+1}-z^k)-\lambda^{-1}\tau e^k}{w^{k+1}-w^k} \notag  \\
		&\quad  + \frac{\tau\lambda}{2}\left[\left({1-\frac{n \alpha_k}{\lambda}}\right)^2-1\right]\norm{\Fnor(z^k)}^2 - \tau\left[{1-\frac{n \alpha_k}{\lambda}}\right]\iprod{\Fnor(z^k)}{{w^{k+1}-w^k}}  \notag  \\ &\quad + \tau\left[{1-\frac{n \alpha_k}{\lambda}}\right]\iprod{\Fnor(z^k)}{\lambda({\nabla f(w^{k+1}) -\nabla f(w^k)})+ e^k }  \notag  \\
		&\quad  +\frac{\tau\lambda}{2}[\sL+\lambda^{-1}]^2\norm{w^{k+1}-w^k}^2 + \tau \iprod{\nabla f(w^{k+1}) -\nabla f(w^k)}{e^k} + \frac{\tau}{2\lambda}\norm{e^k}^2 \notag  \\
		& = -\tau n \alpha_k\left[{1-\frac{n \alpha_k}{2\lambda}}\right]\norm{\Fnor(z^k)}^2 + \iprod{h^k}{w^{k+1}-w^k}\notag  \\
		&\quad  + \tau\left[{1-\frac{n \alpha_k}{\lambda}}\right] \iprod{\Fnor(z^k)}{\lambda({\nabla f(w^{k+1}) - \nabla f(w^k)}) + e^k} + \tau\iprod{\nabla f(w^{k+1}) - \nabla f(w^k)}{e^k} \notag  \\
		&\quad + \frac{\tau}{2\lambda}\norm{e^k}^2 + \left[ \frac{\sL\tau(\sL\lambda+2) + (\sL+\rho)}{2} - \frac{2-\tau}{2\lambda}\right]\norm{w^{k+1}-w^k}^2,
	\end{align}
	\endgroup
 where $h^k := \Fnor(z^k)+{\lambda^{-1}}(z^{k+1}-z^k)-\tau[{1-\frac{n \alpha_k}{\lambda}}] \Fnor(z^k) - {\lambda^{-1}}\tau e^k$. 
Furthermore, according to the definition of $\Fnor$ and $e^k$ in \cref{eq:update-z} and \cref{eq:error-e}, we have 
		\begin{align*}
			 h^k &= \prt{1-\tau}\Fnor(z^k) + \frac{1}{\lambda}(z^{k+1}-z^k) - \frac{\tau}{\lambda}[{e^k - n \alpha_k \Fnor(z^k)}] \\
			&= 	\prt{1-\tau}\Fnor(z^k) + \frac{1-\tau}{\lambda}(z^{k+1}-z^k)=  	\prt{1-\tau}\left[{\frac{1}{\lambda} - \frac{1}{n \alpha_k}}\right](z^{k+1}-z^k)  + \frac{1-\tau}{n \alpha_k}e^k. \notag
		\end{align*}
	Inserting this expression into the estimate \cref{eq:lem-3-est-1}, we obtain
	\begingroup
    \allowdisplaybreaks
		\begin{align}
        \label{eq:lem-3-est-3-1}
			&\hspace{-.5ex} \mer(z^{k+1}) - \mer(z^k)\\ \notag & \leq -\tau n \alpha_k\left[{1-\frac{n \alpha_k}{2\lambda}}\right]\norm{\Fnor(z^k)}^2 -  \prt{1-\tau}\left[{\frac{1}{n \alpha_k} -\frac{1}{\lambda} }\right] \iprod{z^{k+1}-z^k}{w^{k+1}-w^k} + \frac{\tau}{2\lambda}\norm{e^k}^2 \\ \notag
			&\quad + \left[ \frac{\sL\tau(\sL\lambda+2) + (\sL+\rho)}{2} - \frac{2-\tau}{2\lambda} \right] \norm{w^{k+1}-w^k}^2  + \frac{1-\tau}{n \alpha_k}\iprod{e^k}{w^{k+1}-w^k} \\ \notag & \quad + \tau\iprod{\nabla f(w^{k+1}) - \nabla f(w^k)}{e^k}  
			+ \tau\left[{1-\frac{n \alpha_k}{\lambda}}\right] \iprod{\Fnor(z^k)}{\lambda(\nabla f(w^{k+1}) - \nabla f(w^k)) + e^k}.
		\end{align}
	\endgroup
    We now estimate and bound the different terms appearing in \cref{eq:lem-3-est-3-1}. First, due to $\tau\in (0,1)$ and $ n \alpha_k\in (0,\lambda)$ for all $k$, the coefficient in front of the inner product $\iprod{z^{k+1}-z^k}{w^{k+1}-w^k}$ is negative. Hence, using \cref{eq:prox-coco}, we have
	\[- \iprod{z^{k+1}-z^k}{w^{k+1}-w^k} \leq -(1-\lambda\rho) \norm{w^{k+1}-w^k}^2. \]
    Applying the Cauchy-Schwartz inequality, the Lipschitz continuity of $\nabla f$, Young's inequality---$\iprod{a}{b}\leq \frac{\varepsilon}{2} \|a\|^2 + \frac{1}{2\varepsilon}\|b\|^2,\ a,b\in\R^d$ and $\varepsilon>0$---(with $a = w^{k+1}-w^k$, $b =e^k$, and $\varepsilon = 1$), we further obtain 
    \[\iprod{\nabla f(w^{k+1}) - \nabla f(w^k)}{e^k}  \leq \sL \norm{w^{k+1}-w^k}\norm{e^k} \leq \frac{\sL}{2} \|w^{k+1}-w^k\|^2 + \frac{\sL}{2}\|e^k\|^2\] 
    and, similarly, $\iprod{e^k}{w^{k+1}-w^k} \leq \frac12\|w^{k+1}-w^k\|^2 + \frac{1}{2}\|e^k\|^2$. Moreover, setting $a =\Fnor(z^k)$,  $b=\lambda(\nabla f(w^{k+1}) - \nabla f(w^k))$ and $\varepsilon =n \alpha_k$ in Young's inequality and applying the Lipschitz continuity of $\nabla f$, it holds that
	\begin{align*}
		\iprod{\Fnor(z^k)}{\lambda({\nabla f(w^{k+1}) - \nabla f(w^k)})} &\leq \frac{n \alpha_k}{2}\norm{\Fnor(z^k)}^2 + \frac{\lambda^2}{2 n \alpha_k}\norm{\nabla f(w^{k+1}) - \nabla f(w^k)}^2\\
		&\leq \frac{n \alpha_k}{2}\norm{\Fnor(z^k)}^2 + \frac{\sL^2\lambda^2}{2n \alpha_k}\norm{w^{k+1}-w^k}^2.
	\end{align*}
	Repeating this step once more with $a =\Fnor(z^k)$, $b=e^k$ and $\varepsilon=\frac{n \alpha_k}{2}$, we have $\iprod{\Fnor(z^k)}{e^k} \leq \frac{n \alpha_k}{4}\norm{\Fnor(z^k)}^2 + \frac{1}{n \alpha_k}\norm{e^k}^2$. Plugging these different estimates into \cref{eq:lem-3-est-3-1}, we can conclude
\begingroup
\allowdisplaybreaks
\begin{align*}
&\hspace{-.5ex} \mer(z^{k+1}) - \mer(z^k) \\
& \leq -\tau n \alpha_k\left[{1-\frac{n \alpha_k}{2\lambda}}-\frac34\left(1-\frac{n \alpha_k}{\lambda}\right)\right]\norm{\Fnor(z^k)}^2 \\ & \quad + \left[\frac{\tau}{2\lambda} + \frac{1-\tau}{2n\alpha_k} + \frac{\sL\tau}{2} + \frac{\tau}{n\alpha_k}\left(1-\frac{n \alpha_k}{\lambda}\right) \right] \|e^k\|^2 + \left[ \frac{\sL\tau(\sL\lambda+2) + (\sL+\rho)}{2} - \frac{2-\tau}{2\lambda} \;\;\dots \right. \\ & \quad \left. \dots  - (1-\tau)(1-\lambda\rho)\left({\frac{1}{n \alpha_k} -\frac{1}{\lambda} }\right) + \frac{1-\tau}{2n\alpha_k}+\frac{\sL\tau}{2} + \frac{\sL^2\lambda^2\tau}{2n\alpha_k}\left(1-\frac{n\alpha_k}{\lambda}\right)\right] \norm{w^{k+1}-w^k}^2 \\
& = -\frac{\tau n \alpha_k}{4}\left[{1+\frac{n \alpha_k}{\lambda}}\right]\norm{\Fnor(z^k)}^2 + \left[\frac{1+\tau}{2n\alpha_k} + \frac{\sL\tau}{2} -\frac{\tau}{2\lambda} \right] \|e^k\|^2 \\
& \quad + \left[\frac{3\sL\tau}{2}+\frac{\sL+\rho}{2}-\frac{\tau+2\lambda\rho(1-\tau)}{2\lambda} - \frac{(1-\tau)(1-2\lambda\rho)-\sL^2\lambda^2\tau}{2n\alpha_k}\right] \|w^{k+1}-w^k\|^2.
\end{align*}
\endgroup
By assumption, it holds that $\lambda\rho < \frac14$ and we have $\tau = \frac{1-4\lambda\rho}{2(1-2\lambda\rho+\sL^2\lambda^2)} \leq \frac12$. We then may infer 
\begin{align*}
& \frac{3\sL\tau}{2}+\frac{\sL+\rho}{2}-\frac{(1-2\lambda\rho)\tau + 2\lambda\rho}{2\lambda} - \frac{(1-\tau)(1-2\lambda\rho)-\sL^2\lambda^2\tau}{2n\alpha_k} \leq \frac{5\sL}{4}  - \frac{1}{4n\alpha_k}
\end{align*}
and $\frac{1+\tau}{2n\alpha_k} + \frac{\sL\tau}{2} -\frac{\tau}{2\lambda} \leq \frac{3}{4n\alpha_k} + \frac{\sL}{4}$. The choice of the step sizes $\{\alpha_k\}_k$ implies $\frac{5\sL}{4} \leq \frac{1}{8n\alpha_k}$ and hence, it follows $\frac{5\sL}{4}-\frac{1}{4n\alpha_k} \leq - \frac{1}{8n\alpha_k}$ and $\frac{3}{4n\alpha_k} + \frac{\sL}{4} \leq \frac{1}{n\alpha_k}$. Using these bounds in the previous estimate, we finally obtain
 \begin{equation*}
    \label{eq:lem-3-est-5}
    \begin{aligned}
        \mer(z^{k+1})-\mer(z^k) + \frac{1}{8 n \alpha_k}\norm{w^{k+1}-w^k}^2
        &\leq - \frac{\tau  n \alpha_k}{4}\left[{1+\frac{ n \alpha_k}{\lambda}}\right]\norm{\Fnor(z^k)}^2 + \frac{1}{n\alpha_k} \norm{e^k}^2,
    \end{aligned}
\end{equation*}
as desired. \end{proof} 
\vspace{-4ex}

\section{\texorpdfstring{Global Convergence: Proof of \cref{thm:global_convergence}}{Global Convergence: Proof of Theorem 12}} \label{proof:global_convergence}

\begin{proof} Invoking \cref{lem:merit-descent-1}~(a), we have  
 \[
 \mer(z^{k+1}) \leq \mer(z^k) - \frac{\tau n \alpha_k}{4}\norm{\Fnor(z^k)}^2 + \sD \alpha_k^3,
 \]
	where $\sD:= n^3\Delta(n^3{\textstyle \sum}_{i=1}^{\infty}\alpha_i^3) < \infty$. Thus, the sequence $\{\mer(z^k)\}_k$ satisfies a supermartingale-type recursion. Consequently, applying \cref{Thm:super-martingale-convergece} and using the lower bound $\mer(z^k)\geq \psi(w^k)\geq \psilb$ (as stated in assumption \ref{A3}) and the condition $\sum_{k=1}^\infty \alpha_k^3 < \infty$ (as stated in \cref{eq:ass-step}), we can infer $\mer(z^k) \to \bar \psi$, $k \to \infty$ for some $\bar \psi \in \R$ and
	\begin{equation} \label{eq:esti-t2} 
		 {\sum}_{k=1}^\infty \alpha_k \norm{\Fnor(z^k)}^2  < \infty.
	\end{equation}
	Due to ${\sum}_{k=1}^\infty \alpha_k = \infty$, this immediately implies $\liminf_{k\to\infty}\|\Fnor(z^k)\|=0$. In order to show $\lim_{k\to\infty}\|\Fnor(z^k)\|=0$, let us, on the contrary, assume that $\{\|\Fnor(z^k)\|\}_{k}$ does not converge to zero. Then, there exist $\varepsilon>0$ and two infinite subsequences $\{t_j\}_{j}$ and $\{\ell_j\}_{j}$ such that  $t_j<\ell_j \leq t_{j+1}$,
	\begin{equation}\label{eq:construct subsequence}
		\|\Fnor({z}^{t_j})\|\geq 2\varepsilon,\quad \|\Fnor({z}^{\ell_j})\|<\varepsilon,\quad\text{and}\quad\|\Fnor({z}^{k})\|\geq\varepsilon
	\end{equation}
	for all $k=t_j+1,\dots,\ell_j-1$. Combining this observation with \cref{eq:esti-t2}, this yields
	\[
	\infty>  {\sum}_{k=1}^{\infty}\alpha_{k}\|\Fnor(z^{k})\|^2 \geq \varepsilon^2{\sum}_{j=1}^{\infty}{\sum}_{k=t_j}^{\ell_j-1}\alpha_{k},
	\]
	which implies $\lim_{j\rightarrow \infty} \ \beta_j := {\sum}_{k=t_j}^{\ell_j-1}\alpha_{k} = 0$. Next, applying the triangle and Cauchy-Schwarz inequality, we obtain
	\begingroup
    \allowdisplaybreaks
	\begin{align}
    \nonumber
		\|{z}^{\ell_j} - {z}^{t_j}\|  & \leq {\sum}_{k=t_j}^{\ell_j-1}\sqrt{\alpha_{k}} \left[\frac{\|{z}^{k+1}-z^{k}\|}{\sqrt{\alpha_{k}}}\right] \\ & \leq \left[ {\sum}_{k=t_j}^{\ell_j-1} \alpha_{k} \cdot {\sum}_{k=t_j}^{\ell_j-1} \frac{\|z^{k+1}-z^k\|^2}{\alpha_{k}} \right]^\frac12 \leq \sqrt{\beta_j} \cdot \left[ {\sum}_{k=1}^{\infty} \frac{\|z^{k+1}-z^{k}\|^2}{\alpha_{k}}\right]^\frac12.
    \label{eq:contradict-am}
	\end{align}
 	\endgroup
	Using the recursion \cref{eq:update-z} and \cref{lem:var-bound,lem:est-err}, we further have
		\begin{align*}
			{\sum}_{k=1}^{\infty} \alpha_k^{-1}{\|z^{k+1}-z^{k}\|^2} &\leq 2{\sum}_{k=1}^{\infty} [n^2\alpha_k\norm{\Fnor(z^k)}^2 + \alpha_k^{-1}\norm{e^k}^2]\\
			&\hspace{-16ex}\leq 2 {\sum}_{k=1}^{\infty}[(n^2 + \CL n^4 \alpha_k^2)\alpha_k\norm{\Fnor(z^k)}^2 + \CL n^4 \alpha_k^3\sigma_k^2]
			\\&\hspace{-16ex}\leq
			2(n^2 + \CL n^2 \bar \alpha^2){\sum}_{k=1}^{\infty}\alpha_k\norm{\Fnor(z^k)}^2 + 2\sD n{\sum}_{k=1}^{\infty} \alpha_k^3
			<\infty,
		\end{align*}
	where the last line is by \cref{eq:ass-step}, \cref{eq:esti-t2} and $\CL\sigma_k^2 \leq \Delta(n^3\sum_{i=1}^\infty\alpha_i^3) = \frac{\sD}{n^3}$ (cf. \cref{lem:merit-descent-1}~(a)). Hence, taking the limit $j \rightarrow \infty$ in \cref{eq:contradict-am}, it follows 
	\[ 
	{\lim}_{j\rightarrow \infty} \ \|{z}^{\ell_j} - {z}^{t_j}\| = 0.
	\]
	Moreover, applying \cref{eq:construct subsequence}, the inverse triangle inequality, and \cref{eq:lem-4-est-2}, it holds that
	\begin{equation*} 
		\varepsilon\leq |\|\Fnor({z}^{\ell_j})\|-\|\Fnor({z}^{t_j})\| | \leq \|\Fnor({z}^{\ell_j})-\Fnor({z}^{t_j})\|\leq \frac{\sL+2\lambda^{-1}-\rho}{1-\lambda\rho}\|{z}^{\ell_j}-{z}^{t_j}\|.
	\end{equation*}
    Taking the limit $j \to \infty$, we reach a contradiction
 and thus, we conclude $\lim_{k\to\infty}\|\Fnor(z^k)\|=0$. Finally, recalling $\mer(z^k):=\psi(z^k)+\frac{\lambda\tau}{2}\|\Fnor(z^k)\|$ and $\mer(z^k)\to \bar \psi$, we have
	\[\bar \psi=\lim_{k\to\infty}\mer(z^k)=\lim_{k\to\infty}\psi(w^k) + \lim_{k\to\infty}\frac{\lambda\tau}{2}\|\Fnor(z^k)\| = \lim_{k\to\infty}\psi(w^k). \]
	This completes the proof. \end{proof}
\vspace{-4ex}

\section{Strong Convergence}
\subsection{\texorpdfstring{Proof of \cref{lemma:limit point set}}{Proof of Lemma 19}}\label{proof:limit point set}
\begin{proof} By the definition of the normal map, we have $z^k = \lambda \Fnor(z^k) + w^k - \lambda \nabla f(w^k)$. Since $\Fnor(z^k)\to 0$ (cf. \cref{thm:global_convergence}) and $\{w^k\}_k$ is bounded, we can conclude that the sequence $\{z^k\}_k$ is bounded. This verifies (a). 

The inclusions $\mathcal A_z \subseteq \{z: \Fnor(z) = 0\}$ and $\mathcal A_w \subseteq \crit{\psi}$ follow directly from \cref{thm:global_convergence}. Next, let $\bar w \in \mathcal A_w$ with $w^{\ell_k} \to \bar w$ be arbitrary. Then, we have $z^{\ell_k} \to \bar w - \lambda\nabla f(\bar w) =: \bar z$ and $\bar w = \lim_{k \to \infty} \proxi{\lambda}(z^{\ell_k}) = \proxi{\lambda}(\bar z)$. Conversely, let $\bar w = \proxi{\lambda}(\bar z)$ with $\bar z \in \mathcal A_z$ be given. Thus, by definition, there is $z^{\ell_k} \to \bar z$ and using the Lipschitz continuity of the proximity operator, it holds that $\|w^{\ell_k}-\bar w\|=\|\proxi{\lambda}(z^{\ell_k})-\proxi{\lambda}(\bar z)\|\leq \|z^{\ell_k}-\bar z\|/(1-\lambda \rho) \to 0$. This verifies $\bar w \in \mathcal A_w$ and finishes the proof of part (b).

For part (c) and due to the structure of $\mathcal A_w$, we may pick some arbitrary $\bar z\in\mathcal{A}_z$ and $\bar w\in\mathcal{A}_w$ with $\bar w=\proxi{\lambda}(\bar z)$. 
Let $\{\ell_k\}_k$ be a subsequence such that $z^{\ell_k} \to \bar z$ and $w^{\ell_k} \to \bar w$.   
Applying \cref{thm:global_convergence}, we conclude  $\psi(w^k) \to \bar \psi\in\R$. Since $z \mapsto \psi(\proxl(z))$ is continuous, we further have $\bar \psi= \lim_{k\to\infty} \psi(\proxl(z^{\ell_k})) = \psi(\proxl(\bar z)) $. 
This implies $\lim_{k\to\infty} \psi(w^k) = \psi(\bar w) = \bar \psi$ for all $\bar w\in \mathcal{A}_w$. Finally, we have $\mer(\bar z) = \psi(\bar w) + \frac{\tau\lambda}{2}\|\Fnor(\bar z)\|^2=\psi(\bar w)=\bar \psi$ for all $\tau>0$, i.e., $\mer$ is constant on $\mathcal A_z$.
\end{proof}

\subsection{\texorpdfstring{Proof of \cref{thm:finite-sum}}{Proof of Theorem 21}}\label{proof:finite-sum}
\begin{proof}
The properties in \cref{lemma:limit point set} allow us to apply the uniformization technique in \cite[Lemma 6]{BolSabTeb14} to the KL-type inequality derived in \cref{lem:KL-mer}. Specifically, there exist $\hat\zeta, \hat c>0$, $\hat\eta \in (0,1]$, and $\hat\theta \in [\frac12,1)$ such that for all $\bar z \in \mathcal A_z$ and $z  \in V_{\hat\zeta,\hat\eta} := \{ z \in \Rn : \mathrm{dist}(z,\mathcal A_z) < \hat\zeta \} \cap \{z \in \Rn : 0 < |\mer(z)-\bar\psi| < \hat\eta \}$, we have
%
%
%
%
%
\[
\hat c \cdot \norm{\Fnor(z)} \geq |\mer(z)-\bar \psi|^{\hat\theta} \geq |\mer(z)-\bar \psi|^\vartheta,
\]
where $\vartheta := \max\{\hat\theta,\xi\}$ and $\bar \psi = \psi(\bar w) = \mer(\bar z)$ (for all $\bar z \in \mathcal A_z$). Setting $\addes(s):=\frac{\hat c}{1-\vartheta}\cdot s^{1-\vartheta}$, this can be written as
\begin{equation}
    \label{eq:uni-kl}
    \addes^\prime(|\mer(z)-\bar \psi|)\cdot  \norm{\Fnor(z)} \geq 1.
\end{equation}
	%
Since $\dist(z^k,\mathcal A_z)\to0$ and $\mer(z^k)\to\bar \psi$, there exists $\bar k \geq 1$ such that $z^k\in V_{\hat\zeta,\hat\eta}$ for all $k\geq \bar k$. 
	Applying \cref{lem:merit-descent-1} (a), it holds that
	\[
		\mer(z^{k+1}) \leq \mer(z^k) - \frac{\tau n \alpha_k}{4}\norm{\Fnor(z^k)}^2 -\frac{1}{8 n \alpha_k}\norm{w^{k+1}-w^k}^2 + \sD \alpha_k^3,
	\]
	where $\sD = n^3\Delta(n^3{\textstyle \sum}_{i=1}^{\infty}\alpha_i^3) < \infty$. Defining 
 $u_k := \sD {\sum}_{j=k}^{\infty}\alpha_{j}^3$ and adding $u_{k+1}$ on both sides of this inequality, it follows
 \begin{equation}
		\label{eq:descent}
		\mer(z^{k+1})+u_{k+1}\leq \mer(z^k)+u_k-\frac{\tau n \alpha_k}{4}\norm{\Fnor(z^k)}^2 -\frac{1}{8 n \alpha_k}\norm{w^{k+1}-w^k}^2.
	\end{equation}
 In the following, without loss of generality, let us assume $z^k \notin \mathcal A_z$ or $u_k \neq 0$ for all $k \geq \bar k$.
 Let us set $\delta_k :=  \addes(\mer(z^k) - \bar \psi + u_k)$.
	Due to the monotonicity of the sequence $\{\mer(z^k)+u_k\}_{k}$ and $\mer(z^k)+u_k \rightarrow \bar \psi$,  $\delta_k$ is well defined as $\mer(z^k) - \bar \psi + u_k \geq 0$ for all $k \geq 1$.  
	Hence, for all $k \geq \bar k$, we obtain 
    \begingroup
    \allowdisplaybreaks
	\begin{align}\label{eq:thm-kl-1}
		&\hspace{-1ex}\nonumber \delta_{k} - \delta_{k+1} \geq \addes^\prime(\mer(z^{k}) - \bar \psi + u_k) [\mer(z^k) + u_k - \mer(z^{k+1}) - u_{k+1} ] \\ 
		\nonumber & \geq \addes^\prime(|\mer(z^k) - \bar \psi| + u_k)  [\mer(z^k) + u_k - \mer(z^{k+1}) - u_{k+1} ]  \\
		\nonumber &\geq \addes^\prime(|\mer(z^k) - \bar \psi| + u_k) \left[{  \frac{\tau n \alpha_k}{4}\norm{\Fnor(z^k)}^2 + \frac{1}{8 n \alpha_k}\norm{w^{k+1}-w^k}^2}\right]\\
		&\geq   \frac{\frac{\tau n \alpha_k}{4}\norm{\Fnor(z^k)}^2 + \frac{1}{8 n \alpha_k}\norm{w^{k+1}-w^k}^2}{[\addes^\prime(|\mer(z^k) - \bar \psi|)]^{-1} + [\addes^\prime(u_k)]^{-1}} 
		\geq \frac{\tau n}{4}  \frac{ \alpha_k^2\norm{\Fnor(z^k)}^2 + \frac{1}{2\tau n^2 }\norm{w^{k+1}-w^k}^2 }{\alpha_k\norm{\Fnor(z^k)} + \alpha_k[\addes^\prime(u_k)]^{-1}},
	\end{align}
 \endgroup
	where the first inequality uses the concavity of $\addes$, the second inequality is due to monotonicity of $s \mapsto \addes^\prime(s) = \hat c s^{-\vartheta}$,  the third inequality is from \cref{eq:descent}, the fourth inequality follows from subadditivity of $[\addes^\prime(s)]^{-1}$---i.e., $[\addes^\prime(s_1+s_2)]^{-1}\leq [\addes^\prime(s_1)]^{-1}+ [\addes^\prime(s_2)]^{-1}$ for all $s_1,s_2\geq 0$---and the last inequality applies the KL inequality \cref{eq:uni-kl}. Rearranging the terms in \cref{eq:thm-kl-1}, this further yields 
	\begin{equation}\label{eq:thm-kl-bound}
		\begin{aligned}
			&\hspace{-1ex}\frac{4(\delta_k-\delta_{k+1})}{\tau n} \cdot [{\alpha_k\norm{\Fnor(z^k)} + \alpha_k[\addes^\prime(u_k)]^{-1}}] 
			\\&\geq \frac{1}{2\tau n^2}\norm{w^{k+1}-w^k}^2 + \alpha_k^2\norm{\Fnor(z^k)}^2 \geq \frac{1}{2} \Big[\frac{1}{\sqrt{2\tau}n} \norm{w^{k+1}-w^k} +\alpha_k\norm{\Fnor(z^k)}\Big]^2
		\end{aligned}
	\end{equation}
	for all $k\geq \bar k$, where we used $a^2+b^2\geq (a+b)^2/2$. Multiplying both sides of \cref{eq:thm-kl-bound} with $8$ and taking square root, we obtain 
	%
		\begin{align}
			\nonumber\frac{\sqrt{2}}{\sqrt{\tau}n} \norm{w^{k+1}-w^k} +	2\alpha_k\norm{\Fnor(z^k)}&\leq \sqrt{\frac{16(\delta_k-\delta_{k+1})}{\tau n}\cdot 2[{\alpha_k\norm{\Fnor(z^k)} + \alpha_k[\addes^\prime(u_k)]^{-1}}]}\\
			&\leq \frac{8(\delta_k-\delta_{k+1})}{\tau n} + \alpha_k\norm{\Fnor(z^k)} + \alpha_k[\addes^\prime(u_k)]^{-1}\label{eq:thm-kl-2},
		\end{align}
	where the last inequality is due to the AM-GM inequality, $\sqrt{ab} \leq (a+b)/2$. 
	Summing the inequality \cref{eq:thm-kl-2} from $k = \bar k$ to $T$, we obtain
	\begin{equation}\label{eq:thm-kl-3}
		\frac{\sqrt{2}}{\sqrt{\tau}n} {\sum}_{k=\bar k}^{T} \norm{w^{k+1}-w^k} + {\sum}_{k=\bar k}^{T} \alpha_k\norm{\Fnor(z^k)} \leq \frac{8(\delta_{\bar k}-\delta_{T+1})}{\tau n} + {\sum}_{k=\bar k}^{T} \alpha_k [\addes^\prime(u_k)]^{-1}.
	\end{equation}
	Since $\addes$ is continuous with $\addes(0) = 0$, we have $\delta_{T+1} \to 0$ as $T \to \infty$. In addition, by \cref{eq:kl-step} and using $\vartheta \geq \xi$, it follows ${\sum}_{k=\bar k}^\infty \alpha_k [\addes^\prime(u_k)]^{-1} = \hat c {\sum}_{k=\bar k}^\infty \alpha_k u_k^{\vartheta} <\infty$. Thus, taking the limit $T \to \infty$ in \cref{eq:thm-kl-3}, it holds that
	\[
	{\sum}_{k=1}^{\infty}\alpha_k\norm{\Fnor(z^k)} <\infty \quad \text{and}\quad {\sum}_{k=1}^{\infty}\norm{w^{k+1}-w^k}<\infty.
	\]
	The second estimate implies that $\{w^k\}_{k}$ has finite length, and hence it is convergent to some $w^*$. Finally, by \cref{thm:global_convergence}, the  limit $w^*$ is a stationary point of $\psi$.
\end{proof}

\subsection{\texorpdfstring{Proof of \cref{lem:in-main-proof}}{Proof of Lemma 28}}\label{proof:lem:in-main-proof}
\begin{proof}
Similar to \cref{eq:lem-4-est-2}, we have
	\begin{equation}
		\label{eq:lem-gron-1}
			\norm{\Fnor(z^{k+i}) - \Fnor(z^k)} \leq \frac{\sL + 2\lambda^{-1}-\rho}{1-\lambda\rho} \norm{z^{k+i}-z^k}\quad \forall~i \geq 1 \; \text{and}\;  k \geq 1 .
	\end{equation}
	Applying \cref{lem:var-bound,lem:est-err} and $\max\{\|\Fnor(z^k)\|^2,{\psi(w^k) - \psilb}\} \leq \sP$, it further follows
	\begin{equation} \|e^k\|  \leq \sqrt{\CL} n^2 \alpha_k^2 (\|\Fnor(z^k)\|^2 + {2\sL[\psi(w^k)-\psilb]} )^\frac12 \leq \sqrt{\CL\sP(2\sL+1)} n^2 \alpha_k^2.		\label{eq:lem-gron-2}
	\end{equation}
    Setting $\sC_1 := \sqrt{\CL(2\sL+1)} n$, 
		%
	the update rule \cref{eq:update-z} and the estimate \cref{eq:lem-gron-2} imply
	\begin{align*}
		\norm{z^{k+i}-z^k} & \leq n \; {\sum}_{j=0}^{i-1} \; \alpha_{k+j}\norm{\Fnor(z^{k+j})} + {\sum}_{j=0}^{i-1}\norm{e^{k+j}}\\ 
		& \hspace{-4ex} \leq n \left[\varsigma \norm{\Fnor(z^{k})} +  {\sum}_{j=0}^{i-1} \;\alpha_{k+j} \norm{\Fnor(z^{k+j})-\Fnor(z^{k})} \right] + \sC_1n\sqrt{\sP}\;{\sum}_{j=0}^{i-1}\;\alpha_{k+j}^2,
		\end{align*}
%
	where the last line is due to ${\sum}_{j=0}^{i-1} \;\alpha_{k+j} \leq \varsigma$. Defining $\sC_f := (1-\lambda\rho)^{-1}(\sL+2\lambda^{-1}-\rho)n$ and combining the previous estimates, we obtain 
 \begin{align*}
     &\hspace{-4ex}\norm{\Fnor(z^{k+i}) - \Fnor(z^k)} \leq \sC_f n^{-1} \norm{z^{k+i}-z^k}\\&\leq \sC_f {\sum}_{j=0}^{i-1} \;\alpha_{k+j} \norm{\Fnor(z^{k+j})-\Fnor(z^{k})} + \sC_f \left[\varsigma \norm{\Fnor(z^{k})} + \sC_1\sqrt{\sP} \;{\sum}_{j=0}^{i-1}\;\alpha_{k+j}^2\right].
 \end{align*}
We now apply Gronwall's inequality (\cref{Thm:G}) upon setting $\sQ := \sC_f\max\{1,\sC_1\}$, 
		\begin{align*}
			p&:=\sQ \varsigma \norm{\Fnor(z^{k})} + \sQ \sqrt{\sP} \;{\sum}_{j=0}^{i-1}\;\alpha_{k+j}^2, \quad 
			q:= \sQ,\quad a_j :=\alpha_{k+j}, 
		\end{align*}
  $y_j:=\|\Fnor(z^{k+j}) - \Fnor(z^k)\|$, and $t:=i-1$. This establishes the following upper bound: 
  \[\|\Fnor(z^{k+i}) - \Fnor(z^k)\| \leq \sQ\exp(\sQ\varsigma)\left[\varsigma\|\Fnor(z^k)\|+\sqrt{\sP} \;{\sum}_{j=0}^{i-1}\;\alpha_{k+j}^2\right].\]
	Noticing $\varsigma>0$ and $\sQ\varsigma\exp(\sQ\varsigma) \leq \half$ (per assumption), we can infer 
	\begin{align*}
		\norm{\Fnor(z^k)} &\leq \|\Fnor(z^{k+i}) - \Fnor(z^k)\| + \norm{\Fnor(z^{k+i})} \\&\leq \half \norm{\Fnor(z^k)} + \norm{\Fnor(z^{k+i})} + \frac{\sqrt{\sP}}{2\varsigma} \; {\sum}_{j=0}^{i-1}\;\alpha_{k+j}^2.
	\end{align*} 
	Rearranging the terms yields $\norm{\Fnor(z^{k+i})} + \frac{\sqrt{\sP}}{2\varsigma} \; {\sum}_{j=0}^{i-1}\;\alpha_{k+j}^2 \geq \half \norm{\Fnor(z^k)}$. Taking square and using $a^2 + b^2 \geq \frac12(a+b)^2$, we have $\norm{\Fnor(z^{k+i})}^2 + \frac{\sf P}{4\varsigma^2}({\sum}_{j=0}^{i-1}\;\alpha_{k+j}^2)^2 \geq  \frac18\norm{\Fnor(z^k)}^2$.
\end{proof}

\subsection{\texorpdfstring{Proof of \cref{thm:convergence rate}: Rate of $\{w^k\}_k$}{Proof of Theorem 24}} \label{proof:step-4}

\begin{proof} In this section, we complete the proof of \cref{thm:convergence rate} and provide rates for the iterates $\{w^k\}_k$. 
Applying \cref{eq:thm-kl-3} with $T \to \infty$, we have
\begin{equation}
		\frac{\sqrt{2}}{\sqrt{\tau}n} \|w^k-w^*\| \leq \frac{\sqrt{2}}{\sqrt{\tau}n} {\sum}_{i=k}^{\infty} \norm{w^{i+1}-w^i} \leq \frac{8\delta_{k}}{\tau n} + {\sum}_{i= k}^{\infty} \alpha_i [\addes^\prime(u_i)]^{-1},
	\end{equation}
for all $k\geq \bar k$, where $\delta_k := \addes(r_k)$, $\addes(s):=\frac{\tilde c}{1-\vartheta} s^{1-\vartheta}$, and $\vartheta\in [ \max\{\half,\theta\},1)$. (As argued in the first parts of the proof, due to $w^k \to w^*$ and $z^k \to z^*$, we can directly work with the KL exponent $\theta$ and adjusted constant $\tilde c$ instead of using the uniformized versions $\hat\theta$ and $\hat c$). Hence, based on the rate for $\{r_k\}_k$  and using $R(\vartheta,\gamma) \leq R(\theta,\gamma)$ and \cref{lemma:step size} (b), we have 
	\[\delta_k = \cO(1/k^{(1-\vartheta)R(\vartheta,\gamma)}) \;\; \text{and} \;\; {\sum}_{i= k}^{\infty} \alpha_i [\addes^\prime(u_i)]^{-1}=\cO\Big({\sum}_{i= k}^{\infty} \alpha_i u_i^{\vartheta}\Big)=\cO(1/k^{(3\gamma-1)\vartheta - (1-\gamma)}).
	\] 
	Note that the adjusted KL exponent $\vartheta$ can be selected freely. Thus, we may increase it to ensure $\vartheta > \frac{1-\gamma}{3\gamma-1}$ such that ${\sum}_{i= k}^{\infty} \alpha_i [\addes^\prime(u_i)]^{-1} \to 0$ as $k\to\infty$. Hence, provided $\vartheta > \frac{1-\gamma}{3\gamma-1}$, the convergence rate for $\{w^k\}_k$ is
	\[\|w^k-w^*\| = \cO(k^{-Q(\vartheta,\gamma)}) \quad \text{where}\quad Q(\vartheta,\gamma):=\min\{(1-\vartheta) R(\vartheta,\gamma), (3\gamma-1)\vartheta - (1-\gamma)\}. \]
	We observe that $Q(\vartheta,\gamma) = (3\gamma-1)\vartheta - (1-\gamma)$ if $\vartheta \in (\frac{1-\gamma}{3\gamma-1},\frac{\gamma}{3\gamma-1}]$ and $Q(\vartheta,\gamma) = \frac{(1-\vartheta)(1-\gamma)}{2\vartheta-1}$ if $\vartheta \in (\frac{\gamma}{3\gamma-1},1)$. This means that the mapping $Q(\cdot,\gamma)$ is increasing over the interval $(\frac{1-\gamma}{3\gamma-1},\frac{\gamma}{3\gamma-1}]$ and decreasing over $(\frac{\gamma}{3\gamma-1},1)$. Therefore, we can choose $\vartheta$ in the following way to maximize $Q(\vartheta,\gamma)$: when $\theta \leq \frac{\gamma}{3\gamma-1}$, set $\vartheta = \frac{\gamma}{3\gamma-1}$; otherwise, set $\vartheta = \theta$. In summary, we obtain
	\begin{align*}
	\|w^k-w^*\| = \cO(k^{-R_w(\theta,\gamma)})\quad \text{where}\quad R_w(\theta,\gamma):=\begin{cases}
		2\gamma-1 &\text{if}\; \theta \in [0,\frac{\gamma}{3\gamma-1}]\\[2mm]
		\frac{(1-\gamma)(1-\theta)}{2\theta-1} &\text{if}\; \theta \in (\frac{\gamma}{3\gamma-1},1)
	\end{cases} \;\; \gamma \in (\textstyle\frac12,1),
	\end{align*}
and $R_w(\theta,\gamma):=1$ when $\gamma=1, \theta \in [0,\half]$ and $\alpha > \frac{16\tilde c^2}{\tau n}$. \end{proof}

\bibliography{reference}

\end{document}